\documentclass[10pt,aps,prb, preprint,reqno]{amsart}
\usepackage[margin=.8in]{geometry}
\usepackage{amsmath,amssymb,amsthm,graphicx,amsxtra, setspace}
\usepackage{mathabx}
\usepackage[colorlinks,backref]{hyperref}
\usepackage{esint}
\usepackage{hyperref}
\usepackage[dvipsnames]{xcolor}
\usepackage{dsfont}
\usepackage[utf8]{inputenc}
\usepackage{mathrsfs}
\usepackage{upgreek}
\usepackage{mathtools}
\usepackage[makeroom]{cancel}
\allowdisplaybreaks

\usepackage[cyr]{aeguill}

\colorlet{darkblue}{blue!50!black}

\hypersetup{
	colorlinks,%
	citecolor=blue,%
	filecolor=red,%
	linkcolor=darkblue,%
	urlcolor=blue,%
	pdfnewwindow=true,%
	pdfstartview={FitH}
}

\colorlet{darkblue}{red!100!black}

\newtheorem{theorem}{Theorem}[section]

\newtheorem{lemma}[theorem]{Lemma}
\newtheorem{proposition}[theorem]{Proposition}

\newtheorem{definition}[theorem]{Definition}

\newtheorem{remark}[theorem]{Remark}

\let\originalleft\left
\let\originalright\right
\renewcommand{\left}{\mathopen{}\mathclose\bgroup\originalleft}
\renewcommand{\right}{\aftergroup\egroup\originalright}

\newcommand{\Tr}{\mathop{\mathrm{Tr}}}
\renewcommand{\d}{\/\mathrm{d}\/}

\def\e{\boldsymbol{e}}

\def\EN{|\!|\!|}

\def\T{T\wedge\tau_N}

\def\T3{\mathbb{T}^3}

\def\diver{\mathrm{div}}
\def\supp{\mathrm{supp}}
\def\curl{\mathrm{curl}}
\def\L{\mathrm{L}}
\def\A{\mathrm{A}}

\def\Id{\mathrm{Id}}

\def\f{\boldsymbol{f}}

\def\X{\mathbb{X}}
\def\bm{\boldsymbol{m}}

\def\g{\mathbf{g}}

\def\z{\boldsymbol{z} }
\def\v{\boldsymbol{v}}

\def\W{\mathrm{W}}
\def\Wb{\mathbb{W}}

\def\N{\mathbb{N}}

\def\wi{\widetilde}

\def\u{\mathrm{U}}

\def\u{\boldsymbol{u}}
\def\H{\mathbb{H}}

\newcommand{\R}{\mathbb{R}}

\renewcommand{\d}{\/\mathrm{d}\/}

\begin{document}
\title[Stochastic Euler and Hypo-NSE Equations]{Non-uniqueness of H\"older continuous solutions for stochastic Euler \\and Hypodissipative Navier-Stokes equations}
 
\date{\today}



\author[Kinra]{Kush Kinra}
\address{Centre For Applicable Mathematics (CAM), Tata Institute of Fundamental Research, PO Box 6503, GKVK Post Office, Bangalore 560065, India}
\email{kushkinra@gmail.com}
\author[Koley]{Ujjwal Koley}
\address{Centre For Applicable Mathematics (CAM), Tata Institute of Fundamental Research, PO Box 6503, GKVK Post Office, Bangalore 560065, India}
\email{ujjwal@math.tifrbng.res.in}

\begin{abstract}
We construct infinitely many H\"older continuous, global-in-time, and stationary solutions to the stochastic Euler equations and the hypodissipative Navier-Stokes equations, taking values in the space $C(\R;C^{\vartheta})$. For the Euler case, the H\"older exponent $\vartheta$ satisfies $0<\vartheta<\frac{5}{7}\beta$ with $0<\beta< \frac{1}{200}$, while for the hypodissipative Navier-Stokes equations, $\beta$ must additionally satisfy $0<\beta< \min\left\{ \frac{2(1-2\alpha)}{21}, \frac{1}{200}\right\}$. The construction relies on a modified stochastic convex integration scheme, which is central to the analysis. This scheme incorporates Beltrami flows as building blocks and carefully tracks inductive estimates, both pathwise and in expectation. These refinements allow us to achieve improved H\"older regularity for solutions to the underlying stochastic equations, advancing the scope of convex integration techniques in the stochastic setting.
\end{abstract}

\maketitle

\tableofcontents

\section{Introduction}\setcounter{equation}{0}\label{sec1}
\subsection{The equations}
In this article, we are interested in Euler and hypodissipative Navier-Stokes equations perturbed by additive noise in a three-dimensional torus $\mathbb{T}^3={\R^3}/{(\mathrm{2\pi}\mathbb{Z})^3}$. The stochastic hypodissipative Navier-Stokes equations are given by
\begin{equation}\label{eqn_stochatic_u-SHNSE}
	\left\{
	\begin{aligned}
		\d\u+\left[\nu(-\Delta)^{\alpha}\u+\mathrm{div}\left(\u \otimes \u \right)+\nabla p\right]\d t&=\d \mathrm{W}, &&\text{ in } \ \R\times\mathbb{T}^3,\\
		\mathrm{div}\; \u&=0, &&\text{ in } \ \R\times\mathbb{T}^3,
	\end{aligned}
	\right.
\end{equation}
where $\u(t,x):\R\times\mathbb{T}^3\to\R^3$, $p(t,x):\R\times\mathbb{T}^3\to\R$ represent the velocity field and pressure field, respectively, at time $t$ and position $x$, and viscosity coefficient $\nu>0$. Here $\W$ is a $GG^{\ast}$-Wiener process on some probability space $(\Omega,\mathcal{F},\mathbb{P})$, where $G$ is a Hilbert-Schmidt operator from $U$ to $\L^2_{\sigma}$ for some Hilbert space $U$ and $\L^2_{\sigma}$ denotes the space of $\L^2$ functions which are mean-free and divergence-free. The dissipation is fractional in the sense that it assumes the form of a fractional power $(-\Delta)^{\alpha}$ of
the Laplacian (\cite{Roncal+Stinga_2016}), and hypodissipative in the sense that $\alpha\in(0,\frac{1}{2})$.  
Moreover, the stochastic Euler equations are given by 
\begin{equation}\label{eqn_stochatic_u-SEE}
	\left\{
	\begin{aligned}
		\d\u+\left[\mathrm{div}\left(\u \otimes \u \right)+\nabla p\right]\d t&=\d \mathrm{W}, &&\text{ in } \ \R\times\mathbb{T}^3,\\
		\mathrm{div}\; \u&=0, &&\text{ in } \ \R\times\mathbb{T}^3.
	\end{aligned}
	\right.
\end{equation}
Throughout this work, we focus on analytically weak solutions which satisfies stochastic hypodissipative Navier-Stokes equations \eqref{eqn_stochatic_u-SHNSE} and stochastic Euler equations \eqref{eqn_stochatic_u-SEE} in the following sense:     
\begin{definition}\label{AWS}
	 A triplet $((\Omega,\mathcal{F},\{\mathcal{F}_{t}\}_{t\in\R},\mathbb{P}),\u,\W)$ is called an analytically weak solution to the stochastic hypodissipative Navier-Stokes equations \eqref{eqn_stochatic_u-SHNSE} and stochastic Euler equations \eqref{eqn_stochatic_u-SEE}  for $\nu>0$ and $\nu=0$, respectively, provided
	\begin{enumerate}
		\item $(\Omega,\mathcal{F},\{\mathcal{F}_{t}\}_{t\in\R},\mathbb{P})$ is a stochastic basis with a complete right continuous filtration;
		\item $\W$ is an $\R^3$-valued, divergence-free and spatial mean free, two-sided $GG^{\ast}$-Wiener process with respect to the filtration $\{\mathcal{F}_{t}\}_{t\in\R}$;
		\item the velocity $\u\in C(\R\times\T3)$ $\mathbb{P}$-a.s. and is $\{\mathcal{F}_{t}\}_{t\in\R}$-adapted;
		\item for every $-\infty<s\leq t<\infty$, the following holds $\mathbb{P}$-a.s. 
		\begin{align*}
			&\langle\u(t),\upsilon\rangle + \int_{s}^{t}\langle \diver( \u(r)\otimes\u(r)),\upsilon \rangle\d r 
			  =  \langle\u(s),\upsilon\rangle - \nu \int_{s}^{t}\langle(-\Delta)^{\alpha}\u(r),\upsilon\rangle \d r  + \langle\W(t)-\W(s),\upsilon\rangle
		\end{align*}
	for all $\upsilon\in C^{\infty}(\T3)$ with $\diver\upsilon=0.$
	\end{enumerate}
\end{definition}
Our first aim is to prove non-uniqueness of solutions in the class of H\"older continuous functions under suitable assumption on noise.
Secondly, we prove   that global stationary solutions exist in the H\"older space and are not unique. The interpretation of stationarity in this context is the shift invariance of laws of governing solutions on the space of trajectories (see Definition \ref{SS} below and \cite{Hofmanova+Zhu+Zhu_Arxiv}).  More precisely, we set up the following joint trajectory space for the driving Wiener process and the solution:
\begin{align*}
	\mathcal{T} := C(\R;C^{\kappa})\times C(\R;C^{\kappa}),
\end{align*}
for some $\kappa>0$. Also,  assume that $S_{t}$, $t\in\R$, is a shift on trajectories which is given by 
\begin{align*}
	S_{t}(\u,\W)(\cdot):= (\u(\cdot+t),\W(\cdot+t)-\W(t)), \;\;\; t\in\R,\;\; (\u,\W)\in\mathcal{T}.
\end{align*}
 We take into account the second component's shift regulates differently to guarantee that the shift $S_{t}\W$ stays a Wiener process for a Wiener process $\W$.
\begin{definition}\label{SS}
	 A triplet $((\Omega,\mathcal{F},\{\mathcal{F}_{t}\}_{t\in\R},\mathbb{P}),\u,\W)$  is called a stationary solution to the hypodissipative Navier-Stokes equations \eqref{eqn_stochatic_u-SHNSE}  and stochastic Euler equations \eqref{eqn_stochatic_u-SEE} provided it satisfies \eqref{eqn_stochatic_u-SHNSE} and \eqref{eqn_stochatic_u-SEE}, respectively, in the sense of Definition \ref{AWS}  with shift invariant law, that is, 
	\begin{align*}
		\mathcal{L}[S_{t}(\u,\W)]=\mathcal{L}[\u,\W], \;\;\;\; \text{ for all } \;\; t\in\R.
	\end{align*} 
\end{definition}

\subsection{Convex integration related results} 
Convex integration has its beginnings, at least in the deterministic setting, in the work of Nash \cite{Nash_1954} and Kuiper \cite{Kuiper_1955} on $C^1$-isometric embeddings of Riemannian manifolds. In the context of fluid flow models, convex integration techniques have been used successfully by De Lellis and Sz\'ekelyhidi \cite{DeLellis+Szekelyhidi_2009} to construct non-unique weak solutions for the incompressible Euler equations in the class of $\L^{\infty}$ functions. Since their method was based on Tartar's plane wave analysis (\cite{Tartar_1979}), it was not supporting to construct continuous solutions.  A significant advancement occurred when De Lellis and Sz\'ekelyhidi (\cite{DeLellis+Szekelyhidi_2013}) made use of Beltrami waves (introduced in \cite{Constantin+Majda_1988} long back) as building blocks. Later, in \cite{DeLellis+Szekelyhidi_2014}, they were able to construct infinitely many weak solutions in the H\"older space $C^{\frac{1}{10}-}$ with the help of Beltrami waves, which proved to be a significant step towards well-known Onsager's conjecture (\cite{Onsage_1949}). Finally, the author in \cite{Isett_2018} was able to provide a rigorous proof of Onsager's conjecture, see \cite{Buckmaster+DeLellis+Szekelyhidi+Vicol_2019} also. A new class of construction blocks, known as Mikado flows, was taken into consideration by Daneri and Sz\'ekelyhidi in \cite{Daneri+Szekelyhidi_2017}, which led to significant advancements over the original techniques. A short while later, the authors in \cite{Buckmaster+Vicol_2019} presented an additional kind of building blocks, known as intermittent Beltrami flows, to demonstrate the non-uniqueness of weak solutions to the 3D Navier-Stokes equations. We also refer readers to \cite{Buckmaster+Vicol_2019_Notes} for a detailed study of convex integration techniques.
 
 Given this long list of deterministic results on non-uniqueness, one would naturally wonder if the situation is different when stochastic forces are present. In fact, it is a well-known that certain PDEs can be regularized by applying random forces to the equation under the right circumstances. This is because some ill-posed deterministic problems have well-posed stochastic counterparts. To put it briefly, the phenomenon of regularization by noise occurs when there is enough active noise, acting non-trivially in multiple directions, to push solutions away from singularities in the underlying vector field, see \cite{Bagnara+Maurelli+Xu_Arxiv_2023,Coghi+Maurelli_Arxiv_2023,Flandoli+Gubinelli+Priola_2010,Flandoli+Luo_2021}, etc. 
 
 Given these outcomes, the ill-posedness results from convex integration techniques previously described did not exclude the possibility that their stochastic counterparts could be well-posed.  But, the authors in \cite{Hofmanova+Zhu+Zhu_2024} dashed widespread expectations that (pathwise) uniqueness may exist for well-known fluid dynamical stochastic partial differential equations (SPDEs). More precisely, the authors in \cite{Hofmanova+Zhu+Zhu_2024} introduced a stochastic version of convex integration technique and establish the non-uniqueness in law of stochastic 3D Navier-Stokes equations. Note that the authors in \cite{Hofmanova+Zhu+Zhu_2024} used stopping time arguments to produce the required result. Later, many mathematicians followed the work \cite{Hofmanova+Zhu+Zhu_2024} and obtained non-uniqueness of many other stochastic models, see \cite{Chen+Dong+Zhu_2024,Hofmanova+Zhu+Zhu_2022,Hofmanova+Zhu+Zhu_2023,Koley+Yamazaki_Arxiv_2022,Lu_submitted,Rehmeier+Schenke_2023,Yamazaki_2024}, and references therein. Thereafter, the authors in \cite{Chen+Dong+Zhu_2024} develop a new stochastic convex integration scheme which allow us to avoid stopping time arguments, see also \cite{Hofmanova+Zhu+Zhu_Arxiv,Lu+Zhu_Arxiv}, etc. More specifically, the authors in \cite{Chen+Dong+Zhu_2024} were able to iterate inductive estimates in expectation in the stochastic convex integration scheme and utilized stationary Mikado flows as building blocks. Recently, the authors in \cite{Lu+Zhu_Arxiv} used Beltrami waves as building blocks and iterated inductive estimates both pathwise as well as in expectation in their inductive scheme.  This makes the calculations easier and improves the solutions' regularity.

 The goal of this work is to implement the concepts of \cite{Chen+Dong+Zhu_2024,Lu+Zhu_Arxiv} to stochastic hypodissipative Navier-Stokes equations as well as stochastic Euler equations. In particular, inspired by the work of \cite{Chen+Dong+Zhu_2024,Lu+Zhu_Arxiv}, we demonstrate existence of infinitely many \emph{global (in-time)} solutions by implementing the idea of iterating both pathwise estimates and expectation estimates in the inductive scheme. Note that convex integration techniques have also been applied to the fractional (hypodissipative as well as hyperdissipative) Navier-Stokes equations, in two and three spatial dimensions, and in both the deterministic and stochastic cases, we refer the readers to the works \cite{Colombo+DeLellis+DeRosa_2018,DeRosa_2019,Luo+Titi_2020,Yamazaki_2022_SIAM}, and references therein. Moreover, the Euler equations (deterministic and stochastic) have a rich literature in the context of convex integration, we refer readers to the works \cite{Buckmaster+DeLellis+Isett+Szekelyhidi_2015,Buckmaster+DeLellis+Szekelyhidi+Vicol_2019,Buckmaster+Vicol_2019_Notes,Daneri+Szekelyhidi_2017,DeLellis+Kwon_2022,DeLellis+Szekelyhidi_2009,DeLellis+Szekelyhidi_2013,DeLellis+Szekelyhidi_2014,Giri+Kwon_2022,Hofmanova+Zhu+Zhu_2022,Hofmanova+Zhu+Zhu_Arxiv,Isett_2018,Lu_submitted,Lu+Zhu_Arxiv}, and references therein.

\subsection{Scope of the paper}
As already studied in the literature, one can establish existence of infinitely many global  (in-time) and stationary solutions to the three-dimensional stochastic Euler equations perturbed by an additive noise in $\L^2$ (\cite{Hofmanova+Zhu+Zhu_Arxiv}) and $C^{\vartheta}$ for some small $\vartheta>0$ (\cite{Lu+Zhu_Arxiv}) as well. More specifically, very recent works \cite{Lu_submitted,Lu+Zhu_Arxiv} by L\"u et. al. established existence of H\"older continuous solutions for the system \eqref{eqn_stochatic_u-SEE} which is far away from the $C^{\frac13-}$ regularity (see Remark \ref{Compare} below). However, for deterministic Euler equations, it has been established in \cite{Buckmaster+DeLellis+Szekelyhidi+Vicol_2019,Isett_2018} that one can construct the solutions belonging to the H\"older space $C^{\frac13-}$.  Therefore, in this context, there is a huge gap between results on deterministic and stochastic Euler equations, and our main aim is to bridge this gap. Indeed, in an attempt to reduce the gap, our first main result is focused on the non-uniqueness of the global (in-time) H\"older continuous solutions with H\"older exponent $\vartheta$ to the stochastic hypodissipative Navier-Stokes equations \eqref{eqn_stochatic_u-SHNSE} and stochastic Euler equations \eqref{eqn_stochatic_u-SEE} for  a noteworthy H\"older exponent $\vartheta>0$. 
We also mention that, convex integration solutions for stochastic hypodissipative Navier-Stokes equations have been discussed in \cite{Rehmeier+Schenke_2023,Yamazaki_2024}, and these results are based on the stopping time arguments. Let us now state our first main result for \eqref{eqn_stochatic_u-SHNSE} and \eqref{eqn_stochatic_u-SEE}, whereas the proofs follow from Theorems \ref{MR-SHNSE} and \ref{MR-SEE}, respectively.
\begin{theorem}\label{FMR1}
     Assume that $\{\mathcal{F}_{t}\}_{t\in\R}$ is the normal filtration generated by the Wiener process $\W$, $\alpha\in(0,\frac12)$ and  $\Tr((-\Delta)^{\frac{5}{2}}GG^{\ast})<\infty$, then there exist infinitely many analytically weak solutions to \eqref{eqn_stochatic_u-SHNSE} in the sense of Definition \ref{AWS}, which belongs to $C(\R;C^{\vartheta})$ for any $\vartheta\in(0,\frac57\beta)$ with $\beta\in\left(0, \min \left\{ \frac{2(1-2\alpha)}{21}, \frac{1}{200}\right\}\right)$.
\end{theorem}
\begin{theorem}\label{FMR2}
	  Let $\{\mathcal{F}_{t}\}_{t\in\R}$ be the normal filtration generated by the Wiener process $\W$ and $\Tr((-\Delta)^{\frac{5}{2}}GG^{\ast})<\infty$, then there exist infinitely many analytically weak solutions to \eqref{eqn_stochatic_u-SEE} in the sense of Definition \ref{AWS}, which belongs to $C(\R;C^{\vartheta})$ for any $\vartheta\in(0,\frac57\beta)$ with  $\beta\in\left(0,\frac{1}{200}\right)$.
\end{theorem}
\begin{remark}\label{Compare}
Here, we shall mention the difference between the previous works and this work for stochastic Euler equations. In the work \cite{Lu_submitted} (see \cite[Remark 1.2]{Lu_submitted}), the author demonstrate the existence of H\"older continuous solutions to system \eqref{eqn_stochatic_u-SEE} with $\vartheta\in\left(0 \; , \; \frac{1}{73\cdot (62)^5}\right)$, whereas in \cite{Lu+Zhu_Arxiv} (see \cite[Theorem 1.2]{Lu+Zhu_Arxiv}), the authors obtained similar result for $\vartheta\in\left(0 \; , \; \frac{1}{120\cdot 7^5}\right)$. In this work, we have constructed solutions to system \eqref{eqn_stochatic_u-SEE} for any $\vartheta\in(0,\frac57\beta)$ with $\beta\in\left(0,\frac{1}{200}\right)$.
\end{remark}

The second result establishes the existence and non-uniqueness of stationary solutions to stochastic hypodissipative Navier-Stokes equations \eqref{eqn_stochatic_u-SHNSE} and stochastic Euler equations \eqref{eqn_stochatic_u-SEE} in the H\"older space. Let us now state our second main result for \eqref{eqn_stochatic_u-SHNSE} and \eqref{eqn_stochatic_u-SEE}.

\begin{theorem}
	Suppose that $\alpha\in(0,\frac12)$ and  $\Tr((-\Delta)^{\frac{5}{2}}GG^{\ast})<\infty$, then there exist infinitely many stationary solutions to \eqref{eqn_stochatic_u-SHNSE} in the sense of Definition \ref{SS}. Furthermore, solution $\u$ belongs to $C(\R;C^{\vartheta})$ for any $\vartheta\in(0,\frac57\beta)$ with $\beta\in\left(0,\min\left\{ \frac{2(1-2\alpha)}{21}, \frac{1}{200}\right\}\right)$ satisfying
	\begin{align*}
		\sup_{t\in\R}\mathbb{E}\left[\sup_{t\leq s\leq t+1}\|\u(s)\|_{C^{\vartheta}}\right]<\infty.
	\end{align*}
\end{theorem}
\begin{theorem}
	Suppose that  $\Tr((-\Delta)^{\frac{5}{2}}GG^{\ast})<\infty$, then there exist infinitely many stationary solutions to \eqref{eqn_stochatic_u-SEE} in the sense of Definition \ref{SS}. Furthermore, solution $\u$ belongs to $C(\R;C^{\vartheta})$ for any $\vartheta\in(0,\frac57\beta)$ with  $\beta\in\left(0,\frac{1}{200}\right)$ satisfying
	\begin{align*}
		\sup_{t\in\R}\mathbb{E}\left[\sup_{t\leq s\leq t+1}\|\u(s)\|_{C^{\vartheta}}\right]<\infty.
	\end{align*}
\end{theorem}

\subsection{Organization of the article}

 The organization of the rest of the article is as follows: In Section \ref{sec2}, we assemble the basic notations that are utilized throughout the paper.  The main focus of our proofs of Theorems \ref{FMR1} and \ref{FMR2} is Sections \ref{sec3}, \ref{sec4}, and \ref{sec5}, where analytically weak solutions are constructed via stochastic convex integration. To prove our main Theorem \ref{MR-SHNSE}, we present the main Proposition \ref{Iterations} in Section \ref{sec3}. The establishment of stochastic convex integration with pathwise estimates, namely the demonstration of iteration Proposition \ref{Iterations}, is the focus of Sections \ref{sec4} and \ref{sec5}. Using the results from Section \ref{sec3}, we prove the existence of non-unique stationary solutions to the stochastic hypodissipative Navier-Stokes equations \eqref{eqn_stochatic_u-SHNSE} in Section \ref{sec6}. Section \ref{sec7} has shown that the three-dimensional stochastic Euler equations \eqref{eqn_stochatic_u-SEE} have infinitely many global (in-time) and stationary solutions.  Proof of multiple key findings that we employ in the sequel can be found in Appendix \ref{SupportingResults}. The construction and features of Beltrami waves are recalled in Appendix \ref{Beltrami} from \cite{Buckmaster+Vicol_2019_Notes}. The estimations for the transport equations that are extensively utilized in the proof of Proposition \ref{Iterations} are discussed in Appendix \ref{Transport}. Some lemmas that are helpful in the sequel are provided in Appendix \ref{Useful-Lemmas}.

\section{Preliminaries}\setcounter{equation}{0}\label{sec2}
 The notation $a\lesssim b$ is used in the sequel to indicate if a constant $c > 0$ exists such that $a \leq cb$. 

\subsection{Function spaces} 

 Let  $\X$ be a given Banach space with the norm $\|\cdot\|_{\X}$ and $t\in\R$. The space of continuous functions from $[t,t+1]$ to $\X$ is denoted by ${\rm C}_{t}\X:={\rm C}([t,t+1];\X)$ and endowed with the norm
\begin{align*}
	\|\u\|_{{\rm C}_t\X}:=\sup_{t\leq s\leq t+1}\|\u(s)\|_{\X}. 
\end{align*}
For $\alpha\in(0,1)$,  the space of $\alpha$-H\"older continuous functions from $[t,t+1]$ to $\X$ is denoted by ${\rm C}_t^{\alpha}\X$ and equipped with the norm 
\begin{align*}
		\|\u\|_{{\rm C}_t^{\alpha}\X}:= \sup_{s,r\in[t,t+1], s\neq r}\frac{\|\u(s)-\u(r)\|_{\X}}{|s-r|^{\alpha}}  + \|\u\|_{{\rm C}_t\X}. 
\end{align*}
 Also, for $\alpha\in(0,1)$, the space of $\alpha$-H\"older continuous functions from $\T3$ to $\R^3$ is denoted by ${\rm C}^{\alpha}:={\rm C}^{\alpha}(\mathbb{T}^3)$ and equipped  with the norm 
\begin{align*}
	\|\u\|_{{\rm C}^{\alpha}}:= \sup_{x,y\in\T3, x\neq y}\frac{|\u(x)-\u(y)|}{|x-y|^{\alpha}}  + \sup_{x\in\T3}|\u(x)|. 
\end{align*}
 The space of standard $\L^p$-integrable functions from $\T3$ to $\R^3$ is denoted by $\L^p$. For $s>0$ and $p>1$, the Sobolev space $$\Wb^{s,p}:=\{\u\in\L^p; \;\|\u\|_{\Wb^{s,p}}:=\|(\Id-\Delta)^{\frac{s}{2}}\u\|_{\L^p}<\infty\}.$$ 
We set 
$$\L^2_{\sigma}:=\{\u\in\L^2; \; \int_{\T3}\u(x)\d x=\boldsymbol{0}, \; \diver\u=0\}.$$
For $s>0$, we also write $\H^s:=\Wb^{s,2}\cap\L^2_{\sigma}$. For $t\in\R$ and a domain $D\subset\R$, we write the space of ${\rm C}^N$-functions from $[t,t+1]\times\T3$ and $D\times\T3$ to $\R^3$, respectively, by ${\rm C}^N_{t,x}$ and ${\rm C}^N_{D,x},$ for any $N\in\N_0:=\N\cup\{0\}$. The spaces are endowed with the norms 
\begin{align*}
	\|\u\|_{{\rm C}^N_{t,x}} = \sum_{\substack{0\leq n+|\alpha|\leq N \\ n\in\N_0, \alpha \in\N_0^3}} \|\partial_t^{n}D^{\alpha}\u\|_{{\rm L}^{\infty}_{[t,t+1]}\L^{\infty}} \;\;\; \text{ and}  \;\;\; 
	\|\u\|_{{\rm C}^N_{D,x}} = \sum_{\substack{0\leq n+|\alpha|\leq N \\ n\in\N_0, \alpha \in\N_0^3}} \sup_{t\in D} \|\partial_t^{n}D^{\alpha}\u\|_{\L^{\infty}},
\end{align*}
respectively.
\begin{remark}
	 To indicate the trace-free part of the tensor product, we use the symbol $\mathring{\otimes}$. We represent the trace-free part of a tensor $T$ as $\mathring{T}:= T-\frac13 \Tr(T)\Id$.
\end{remark}
\subsection{Inverse divergence operator $\mathcal{R}$}\label{IDO-R}
Let us recall the definition of inverse divergence operator $\mathcal{R}$ from \cite[Definition 4.2]{DeLellis+Szekelyhidi_2013}. The operator $\mathcal{R}$ acts on a vector field $\u=(u^1,u^2,u^3)$ with $\int_{\T3}\u(x)\d x =\boldsymbol{0}$ as 
\begin{align*}
	(\mathcal{R}\u)^{ij} = (\partial_i\Delta^{-1}u^j+\partial_j\Delta^{-1}u^i)-\frac12 (\delta_{ij}+\partial_i\partial_j\Delta^{-1})\diver\Delta^{-1}\u,
\end{align*}
for $i,j\in\{1,2,3\}$. We also know that $\mathcal{R}\u(x)$ is a symmetric trace-free matrix for each $x\in\T3$, and $\mathcal{R}$ is a right inverse of the $\diver$ operator, that is, $\diver(\mathcal{R}\u)=\u$. Interestingly, from \cite[Theorem B.3]{Cheskidov+Luo_2022}, we also have for $1\leq p\leq\infty$
\begin{align}\label{IDO_B}
	\|\mathcal{R}\u\|_{\L^p}\lesssim \|\u\|_{\L^p}.
\end{align}

\subsection{Probabilistic elements}
  We use $\mathbb{E}$ to indicate the expectation under $\mathbb{P}$ for a given probability measure $\mathbb{P}$. We assume that $\W$ is a $\R^3$-valued two-sided $GG^{\ast}$-Wiener process with spatial mean zero and divergence-free. Here, $G$ is a Hilbert-Schmidt operator from $U$ to $\L^2_{\sigma}$ for some Hilbert space $U$ and the Wiener process $\W$ is defined on some probability space $(\Omega,\mathcal{F},\mathbb{P})$.

 Let $\X={\rm C}(\T3)$ or $\X={\rm C}^{\kappa}(\T3)$ be a given Banach space for some $\kappa>0$, $1\leq p<\infty$ and $\delta\in(0,\frac12)$. The norms on function spaces of random variables on $\Omega$ taking values in ${ C}(\R;\X)$ and ${ C}^{\frac12-\delta}(\R;\X)$, respectively, with bounds in ${\rm L}^p(\Omega;C({\rm I};\X))$ and ${\rm L}^p(\Omega ; C^{\frac12-\delta}({\rm I} ; \X))$ for bounded interval ${\rm I}\subset\R$, are denoted by 
\begin{align*}
	\EN\u\EN_{\X,p}^p:= \sup_{t\in\R}\mathbb{E}\left[\sup_{t\leq s\leq t+1}\|\u(s)\|_{\X}^p\right], \;\;\;\; \EN\u\EN_{C_t^{\frac12-\delta}\X,p}^p:= \sup_{t\in\R}\mathbb{E}\left[\|\u\|_{C_t^{\frac12-\delta}\X}^p\right].
\end{align*}
 Crucially, the bounds are independent of the interval's location within $\R$ and only rely on the length of ${\rm I}$. Furthermore, for some $\kappa>0$, we establish the appropriate norms with $\X$ substituted by $\L^1$ and $\H^{\frac32 +\kappa}$. The fact that $\u$ has a uniform moment of order $p$ locally in ${C}(\R \times \T3)$ is denoted as $\EN\u\EN_{\mathrm{C}^0,p}<\infty$ in the discussion that follows.


\section{Stochastic convex integration set-up and first main result for \eqref{eqn_stochatic_u-SHNSE}}\setcounter{equation}{0}\label{sec3}

In order to avoid the stopping time arguments as discussed in the literature  (see \cite{Hofmanova+Zhu+Zhu_2022,Hofmanova+Zhu+Zhu_2023,Hofmanova+Zhu+Zhu_2024,Rehmeier+Schenke_2023,Yamazaki_2024}, etc. and references therein), the authors in \cite{Chen+Dong+Zhu_2024} introduced a new version of stochastic convex integration which allow us to obtain solution on whole real line $\R$, see also \cite{Hofmanova+Zhu+Zhu_Arxiv,Lu+Zhu_Arxiv}, etc. Inspired by aforementioned works, we are also interested in finding the solution on whole real line $\R$. More precisely, the authors in \cite{Chen+Dong+Zhu_2024} used stationary Mikado flows as building blocks and measured the iterative estimates in expectations in stochastic convex integration scheme. Due to the quadratic nature of the nonlinear term in the equation, it must estimate higher moments at step $q$ than the moment bounds necessary at step $q + 1$. It could be inferred from this that one has to bound all finite moments at every step, which could potentially blow up during iterations. Recently, the authors in \cite{Lu+Zhu_Arxiv} used  Beltrami waves as building blocks  analogous to the deterministic scenario shown in \cite{Buckmaster+Vicol_2019_Notes}, therefore they introduced pathwise as well as expectation estimates in their inductive scheme. In this work, we also use Beltrami waves as building blocks (see Appendix \ref{Beltrami}) to construct the velocity perturbation. To be more precise, this concept uses the cutoff method to regulate the noise growth, which allows pathwise estimates to be introduced throughout the inductive scheme. The desired uniform estimates are derived by combining this with moment bounds.  The benefit of this technique is that it only requires lower moments of the solutions and does not require higher moments, in contrast to the inductive scheme in \cite{Chen+Dong+Zhu_2024,Hofmanova+Zhu+Zhu_Arxiv} that relies solely on moment bounds. This change also aids in improving the solutions' regularity.

Next, our aim is to develop an iterative process that will enable us to prove Proposition \ref{Iterations}.   To achieve this, we break down the stochastic hypodissipative Navier-Stokes equations \eqref{eqn_stochatic_u-SHNSE} into two components: a random partial differential equation and a linear component that incorporates noise. In particular, we decompose $\u=\v+\z$ such that $\z$ solves the following stochastic linear system:
\begin{equation}\label{eqn_z_alpha}
	\left\{
	\begin{aligned}
		\d\z  + \left[\nu(-\Delta)^{\alpha}\z+\z\right]\d t &=\d \mathrm{W},\\
		\mathrm{div}\;\z&=0,
	\end{aligned}
	\right.
\end{equation}
where $\W$ is $\R^3$-valued two-sided trace-class $GG^{\ast}$-Wiener process with spatial mean zero and divergence-free, and $\v$ solves the following non-linear random system:
\begin{equation}\label{eqn_random_v2}
	\left\{
	\begin{aligned}
		\partial_t\v  +\nu(-\Delta)^{\alpha}\v +\mathrm{div}\left((\v+\z) \otimes (\v+\z) \right)  -\z +\nabla p &=0,\\
		\mathrm{div}\;\v&=0.
	\end{aligned}
	\right.
\end{equation}
The pressure term associated with $\v$ is denoted by $p$, and $\z$ is divergence-free based on the assumptions on the noise $\W$. 

 The regularity of $\z$ on a given stochastic basis $(\Omega,\mathcal{F},\{\mathcal{F}_{t}\}_{t\in\R},\mathbb{P})$ can be obtained in light of the factorization technique, with $\{\mathcal{F}_{t}\}_{t\in\R}$ serving as the standard filtration. Specifically, one can use the same reasoning as \cite[Theorem 5.16]{DaZ} to show the following result for regularity of $\z$. 

\begin{proposition}\label{thm_z_alpha}
	Suppose that $\Tr\left((-\Delta)^{\frac52}GG^{\ast}\right)<\infty$. Then for any $\delta\in(0,\frac12)$, $p\geq2$
	\begin{align}\label{Z-regularity}
		\sup_{t\in\R}	\mathbb{E}\left[\|\z\|^p_{C^{\frac12-\delta}_{t}C^{0}_x}\right]\lesssim \sup_{t\in\R}	\mathbb{E}\left[\|\z\|^p_{C^{\frac12-\delta}_{t}\H^{\frac52}}\right] \leq (p-1)^{\frac{p}{2}}L^{p},
	\end{align}
	where $L\geq \max\left\{1 , \frac{1+\bar{M}}{\sqrt{6r}}\right\}$ depends on $\Tr\left((-\Delta)^{\frac52}GG^{\ast}\right)<\infty$, $\delta$ and is independent of $p$. Here $\bar{M}$ and $r$ are same as in Proposition \ref{Iterations}. 
\end{proposition}
\begin{proof}
The proof of this proposition closely follows those of \cite[Theorem 5.16]{DaZ}, \cite[Proposition 3.1]{Lu+Zhu_Arxiv}, \cite[Proposition 3.2]{Hofmanova+Zhu+Zhu_Arxiv}, and several others. Therefore, we omit the details here. We note, however, that the condition $L\geq \max\left\{1 , \frac{1+\bar{M}}{\sqrt{6r}}\right\}$ is not necessary for the present proposition ($L\ge 1$ works), but will be required for the proof of Proposition \ref{Iterations} below.
\end{proof}

\subsection{Iterative proposition}\label{Itra+Pro}
Now, we apply the convex integration method to the nonlinear equation \eqref{eqn_random_v2}. The convex integration iteration is indexed by a parameter $q\in\N_0$. We consider an increasing sequence $\{\lambda_{q}\}_{q\in\N_0}$ which diverges to $\infty$, and a bounded sequence $\{\delta_{q}\}_{q\in\N_0}$ which decreases to $0$.  We choose $a \in\N$ sufficiently large, $b\in\N$ and $\beta\in(0,1)$ and let
\begin{align*}
	\lambda_{q} & =a^{b^{q}},\;\; \\  \delta_{0} = 25 r^2 L^4 , 
    \;\; \delta_{1}=3rL^2,\;\;  \delta_{q} & = \frac12\lambda_{2}^{2\beta}\lambda_{q}^{-2\beta},\;\; q\geq2.
\end{align*}
 At each step $q$, a pair $(\v_{q}, \mathring{R}_{q})$ is constructing solving the following system:
\begin{equation}\label{eqn_v_q}
	\left\{
\begin{aligned}
	\partial_t\v_q  +\nu(-\Delta)^{\alpha}\v_q  +\mathrm{div}\left((\v_q+\z_q) \otimes (\v_q+\z_q) \right)  -\z_q +\nabla p_q &=\diver \mathring{R}_q,  \\
	\mathrm{div}\;\v_q&=0.
\end{aligned}
\right.
\end{equation}
 In the above, we define
\begin{align}\label{cutoff}
	\z_q(t,x) := \chi_q\left(\|\wi\z_q(t)\|_{C^0_{x}}\right)\wi\chi_q\left(\|\wi\z_q(t)\|_{C^1_{x}}\right)\wi\z_q(t,x),
\end{align}
where $\chi_q$ and $\wi\chi_q$  are non-increasing smooth functions such that
\begin{align}\label{cut-off_functions}
	\chi_q(x)=\begin{cases}
		1, &x\in[0,\frac{1}{7\cdot4\cdot2\cdot2}\lambda_{q}^{\frac13-\varepsilon}],\\
		0, &x\in(\frac{1}{7\cdot4\cdot2\cdot2}\lambda_{q}^{\frac13},\infty),
	\end{cases}
\;\; \text{ and } \;\;
\wi\chi_q(x)=\begin{cases}
	1, &x\in[0,\lambda_{q}^{\frac23-\varepsilon}],\\
	0, &x\in(\lambda_{q}^{\frac23},\infty),
\end{cases}
\end{align}
with their derivatives bounded by 1, which requires 
\begin{align}\label{Slope}
	\frac{1}{\lambda_{q}^{\frac13}-\lambda_{q}^{\frac13-\varepsilon}}\leq \frac{1}{7\cdot4\cdot2\cdot2} \;\; \text{ and }\;\; \frac{1}{\lambda_{q}^{\frac23}-\lambda_{q}^{\frac23-\varepsilon}}\leq 1,
\end{align}
and $\wi\z_q=\mathbb{P}_{\leq f(q)}\z$ with $f(q)=\lambda_{q}^{\frac13}$. Here, the Fourier multiplier operator $\mathbb{P}_{\leq f(q)}$ projects a function onto its Fourier frequencies $\leq f(q)$ in absolute value, and $\varepsilon$ is an additional small parameter (see Subsection \ref{CoP}). On the right side of \eqref{eqn_v_q}, $\mathring{R}_{q}$ is a $3\times3$ trace-free matrix, and we put the trace part into the pressure.

\begin{remark}\label{Rem3.2}
Note that the cutoff function \eqref{cutoff} introduced in this article differs from the one used in \cite{Lu+Zhu_Arxiv} and plays a crucial role in enhancing the regularity of the solution. Indeed, by the definition of $\z_q$, we have 
	\begin{align}\label{z_q_C0}
		\|\z_q\|_{C^0_{t,x}} \leq \frac{1}{7\cdot4\cdot2\cdot2}\lambda_q^{\frac13} \leq \frac12\lambda_{q}^{\frac13} \;\;\; \text{ and }\;\;\; \|\z_q\|_{C^0_tC^1_{x}} \leq \lambda_{q}^{\frac23}.
	\end{align}
It is also worth noting that, due to the specific choice of the shapes of the cutoff functions $\chi_q$ and $\wi\chi_q$, we obtain different bounds for the $C^0_{t,x}$-norm and the $C^0_tC^1_x$-norm of $\z_q$. This represents an important departure from the work in \cite{Lu+Zhu_Arxiv}, where the authors derived the same bounds for both norms.
\end{remark}

\begin{remark}\label{Rem3.3}
	By the Sobolev embedding $\|\f\|_{C^0_x}\lesssim\|\f\|_{\H^{\frac52}}$ and \eqref{Z-regularity}, we observe that for any $p\geq2$ and $\delta\in(0,\frac12)$
	
    \begin{align}
		\EN \z_q\EN_{C^0,p} & \lesssim \EN\wi\z_q\EN_{\H^{\frac52},p}\lesssim \EN\z\EN_{\H^{\frac52},p}\lesssim (p-1)^{\frac12}L,\label{z_q_EN1}\\
		\EN \z_q\EN_{C^1,p} & \lesssim \EN \wi \z_q\EN_{C^1,p} \lesssim f(q)\EN\z\EN_{\H^{\frac52},p}\lesssim  (p-1)^{\frac12}L\lambda_{q}^{\frac13},\label{z_q_EN2}\\
		\EN\z_q \EN_{C_t^{\frac12-\delta}C^0_x,p} & \lesssim \EN \wi\z_q\EN_{C_t^{\frac12-\delta}C^1_x,2p} \EN \wi\z_q\EN_{C_t^{\frac12-\delta}C^0_x,2p} \lesssim \lambda_{q}^{\frac13}\EN\z\EN^2_{C_t^{\frac12-\delta}\H^{\frac52},2p}\lesssim (2p-1)L^2 \lambda_{q}^{\frac13},\label{z_q_EN3}\\
		\EN\z_q \EN_{C_t^{\frac12-\delta}C^1_x,p} & \lesssim \EN \wi\z_q\EN_{C_t^{\frac12-\delta}C^1_x,2p} \EN \wi\z_q\EN_{C_t^{\frac12-\delta}C^1_x,2p} \lesssim \lambda_{q}^{\frac23}\EN\z\EN^2_{C_t^{\frac12-\delta}\H^{\frac52},2p}\lesssim (2p-1)L^2 \lambda_{q}^{\frac23}.\label{z_q_EN4}
	\end{align} 
\end{remark}
The above estimates \eqref{z_q_EN1}-\eqref{z_q_EN4} will be used in the sequel.
Let us also give an estimate on $\z_{q+1}-\z_{q}$, which will be used in the sequel. The proof of the following lemma is given in Appendix \ref{A.2}.
\begin{lemma}\label{Lemma-z_diff_EN}
	Suppose that $\Tr\left((-\Delta)^{\frac52}GG^{\ast}\right)<\infty$. For any $p\geq2$, we have for $0<\varepsilon<\frac13$,
	\begin{align}\label{z_diff_EN}
		\EN\z_{q+1}-\z_{q}\EN_{C^0,p}\lesssim 
		 pL^2\lambda_{q}^{-(\frac13-\varepsilon)}.
	\end{align}
\end{lemma}

Under the above assumptions, we present the following iteration proposition which plays a crucial role to prove the main results of this work. The proof of the following proposition is provided in Sections \ref{sec4} and \ref{sec5}.  
 
\begin{proposition}\label{Iterations}
	Suppose that $\Tr\left((-\Delta)^{\frac52}GG^{\ast}\right)<\infty$ and let $r>1$ be fixed. Then, for any $\beta\in(0,\min \left\{ \frac{2(1-2\alpha)}{21}, \frac{1}{200}\right\})$, there exists $a_0>1$ such that for any $a>a_0$ the following holds: Let $(\v_{q},\mathring{R}_{q})$ for some $q\in\N_0$ be an $\{\mathcal{F}_{t}\}_{t\in\R}$-adapted solution to the system \eqref{eqn_v_q} satisfying 
		\begin{align}
		\|\v_{q}\|_{C^0_{t,x}} &\leq \lambda_{q}^{\frac13}, \label{v_q_C0}\\
		\|\v_{q}\|_{C^1_{t,x}} & \leq \lambda_{q}^{\frac75}\delta_{q}^{\frac12},
        \label{v_q_C1}\\   
        \EN\v_{q}\EN_{C^0,2r} &    \leq 
             6rL^2 - \delta_{q}^{\frac12}   \label{v_q_EN}\\
		\|\mathring{R}_q\|_{C^0_{t,x}} & \leq \lambda_{q}^{\frac23}, \label{R_q_C0}\\
		\EN \mathring{R}_{q}\EN_{C^0,r} &\leq \delta_{q+1}   \label{R_q_EN}.
	\end{align}
There exists an {$\{\mathcal{F}_{t}\}_{t\in\R}$}-adapted process $(\v_{q+1},\mathring{R}_{q+1})$ which solves the system \eqref{eqn_v_q}, obeys \eqref{v_q_C0}-\eqref{R_q_EN} at the level $q+1$ and satisfies 
	\begin{align}\label{v_diff_EN}
		\EN\v_{q+1}-\v_{q}\EN_{C^0,2r} \leq \bar{M}  \delta_{q+1}^{\frac12},
	\end{align}
where $\bar{M}$ is a universal constant which will be fixed throughout the iteration.
\end{proposition}

Proposition \ref{Iterations} helps us to prove the existence of an analytically weak solution in the sense of Definition \ref{AWS}, while the next proposition would help us to prove that such solutions are infinitely many. For that purpose, we will need the following convention:
\begin{itemize}
	\item Given an interval $\mathcal{I}\subseteq\R$ an a function $f$ on $\mathcal{I}\times\T3$, $\supp_{t}(f)$ will denote its temporal support, namely 
	\begin{align*}
		\supp_{t}(f):= \{t  \ : \mbox{ there exists } x \mbox{ with } f(x)\neq0\}.
	\end{align*}
	\item Given an interval $\mathcal{I}=[a,b]$, $|\mathcal{I}|$ will denote its length $(b-a)$ and $\mathcal{I}+c$ will denote the concentric enlarged interval $(a-c,b+c)$.
\end{itemize}
The following assertion, which is based on the concept presented in \cite{DeLellis+Kwon_2022}, is proved in Appendix \ref{A.3}.
\begin{proposition}[Bifurcating inductive proposition]\label{BIP}
	Let $(\v_{q},\mathring{R}_{q})$ be as in the statement of Proposition \ref{Iterations}. For any interval $\mathcal{I}\subset \R$ with $|\mathcal{I}|\geq 3m_q$ (where $m_q$ is defined in \eqref{m_q}), we can produce a first pair $(\v_{q+1},\mathring{R}_{q+1})$ and a second pair $(\wi\v_{q+1},\widetilde{\mathring{R}}_{q+1})$ which share the same initial data, satisfy the same conclusions of Proposition \ref{Iterations} and additionally
	\begin{align}\label{BIP-1}
		\EN\v_{q+1}-\wi\v_{q+1}\EN_{\L^{2}_{x},2} \geq \delta_{q+1}^{\frac12}, \;\; \supp_{t}(\v_{q+1}-\wi\v_{q+1})\subset \mathcal{I}.
	\end{align}
	Moreover, if we are given two pairs $( \v_{q}, {\mathring{R}}_{q})$ and $(\wi\v_{q},\widetilde{\mathring{R}}_{q})$ satisfying \eqref{v_q_C0}-\eqref{R_q_EN} and 
	\begin{align*}
		\supp_{t}(\v_{q}-\wi\v_{q},\mathring{R}_{q}-\widetilde{\mathring{R}}_{q}) \subset \mathcal{J},
	\end{align*}
	for some interval $\mathcal{J}\subset\R$, we can exhibit corrected counterparts $( \v_{q+1}, {\mathring{R}}_{q+1})$ and $(\wi\v_{q+1},\widetilde{\mathring{R}}_{q+1})$ again satisfying the same conclusion of Proposition \ref{Iterations} together with the following control on the support of their difference:
	\begin{align*}
		\supp_{t}(\v_{q+1}-\wi\v_{q+1}, \mathring{R}_{q+1}-\widetilde{\mathring{R}}_{q+1}) \subset \mathcal{J}+\lambda_{q}^{-\frac85}.
	\end{align*}
\end{proposition}

The iteration starts at 
$
	\v_{0}=\boldsymbol{0}$ and  $\mathring{R}_{0}=-\mathcal{R}\z_0+\z_0  \mathring{\otimes} \z_0,$ 
where $\mathcal{R}$ denotes the inverse-divergence operator. It implies that \eqref{v_q_C0}-\eqref{v_q_EN} are fulfilled for $q=0$ immediately. From \eqref{z_q_C0}, we have 
\begin{align*}
	\|\mathring{R}_{0}\|_{C^0_{t,x}} & \leq \|\z_{0}\|^2_{C^0_{t,x}} + \|\z_{0}\|_{C^0_{t,x}} 
	  \leq \frac14\lambda^{\frac23}_{0} + \frac12\lambda^{\frac13}_{0}  \leq {\lambda^{\frac23}_{0}},
\end{align*}
and from \eqref{z_q_EN1}, we get
\begin{align*}
	\EN \mathring{R}_{0}\EN_{C^0,r} & \leq \EN\z_{0}\EN^2_{C^0,2r} + \EN\z_{0}\EN_{C^0,r} 
	  \leq 2rL^2+rL\leq \delta_{1}.
\end{align*}
 Hence \eqref{R_q_C0} and \eqref{R_q_EN} are satisfied at the level $q=0$.

\subsection{Main result} We have just established the first iteration. Therefore, in view of Proposition \ref{Iterations}, we demonstrate the following result for the system \eqref{eqn_stochatic_u-SHNSE}:

\begin{theorem}\label{MR-SHNSE}
    Suppose that $\Tr\left((-\Delta)^{\frac52}GG^{\ast}\right)<\infty$ and let $r>1$ be fixed. Then, for any $\beta\in \left(0,\min \left\{ \frac{2(1-2\alpha)}{21}, \frac{1}{200}\right\}\right)$, there exist infinitely many $\{\mathcal{F}_t\}_{t\in\R}$-adapted process $\u(\cdot)$ which belongs to $C(\R;C^{\vartheta})$ $\mathbb{P}$-a.s. for $\vartheta\in(0,\frac57\beta)$ and is an analytically weak solution to \eqref{eqn_stochatic_u-SHNSE} in the sense of Definition \ref{AWS}. Moreover, the solutions satisfy 
	\begin{align}\label{u_EN}
		\EN\u\EN_{C^{\vartheta},2r}<\infty.
	\end{align}
\end{theorem}
\begin{proof}
	As we know the existence of starting pair $(\v_{0},\mathring{R}_0)$, we are able to acquire a sequence of solutions $(\v_{q},\mathring{R}_q)$ in view of Proposition \ref{Iterations}. By \eqref{v_q_C1}, \eqref{v_diff_EN} and interpolation inequality, we establish the following summable series for any $\vartheta\in(0,\frac{5}{7}\beta)$:
	\begin{align*}
		\sum_{q\geq0}\EN\v_{q+1}-\v_{q}\EN_{C^{\vartheta},2r} & \leq \sum_{q\geq0} \EN\v_{q+1}-\v_{q}\EN_{C^{0},2r}^{1-\vartheta} \EN\v_{q+1}-\v_{q}\EN_{C^{1},2r}^{\vartheta} 
	\nonumber \\ &  \lesssim \sum_{q\geq0} \delta_{q+1}^{\frac{1-\vartheta}{2}}\lambda_{q+1}^{\frac75\vartheta} \delta_{q+1}^{\frac{\vartheta}{2}} 
	  \leq \sqrt{3r}L a^{\frac75 b \vartheta}+ \lambda_2^{\beta}\sum_{q\geq1}\lambda_{q+1}^{\frac75\vartheta-\beta}<\infty.
	\end{align*}
Therefore, there exists a limiting function $\v=\lim\limits_{q\to\infty}\v_{q}$ which belongs to ${\L}^{2r}(\Omega;C(\R;C^{\vartheta}))$. Since $\v_{q}$ is $\{\mathcal{F}_t\}_{t\in\R}$-adapted for every $q\in\N_0$, the limit $\v$ is $\{\mathcal{F}_t\}_{t\in\R}$-adapted as well. Now,  using interpolation inequality, and combining \eqref{z_q_EN2} and \eqref{z_diff_EN}, we demonstrate that for the same $\vartheta$ as above and any $p\geq2$
	\begin{align*}
	\sum_{q\geq0}\EN\z_{q+1}-\z_{q}\EN_{C^{\vartheta},p} & \leq \sum_{q\geq0} \EN\z_{q+1}-\z_{q}\EN_{C^{0},p}^{1-\vartheta} \EN\z_{q+1}-\z_{q}\EN_{C^{1},p}^{\vartheta} 
	\nonumber \\ & \lesssim \sum_{q\geq0} \left[pL^2\lambda_{q}^{-(\frac13-\varepsilon)}\right]^{1-\vartheta}\left[(p-1)^{\frac12}L\lambda_{q+1}^{\frac13}\right]^{\vartheta}
	= \left[pL^2\right]^{1-\vartheta} \left[(p-1)^{\frac12}L\right]^{\vartheta} \sum_{q\geq0} \lambda_{q}^{-(\frac13-\varepsilon)(1-\vartheta)+\frac{2\vartheta}{3}}<\infty,
\end{align*}
since $\vartheta<\frac{\frac13-\varepsilon}{1-\varepsilon}$. Then we obtain $\lim\limits_{q\to\infty}\z_{q}=\z$ in $\L^p(\Omega;C(\R;C^{\vartheta}))$ for any $p\geq2$. Furthermore, it also follows from \eqref{R_q_EN} that $\lim\limits_{q\to\infty}\mathring{R}_q=0$ in $L^1(\Omega;C(\R;C^{0}))$. Thus $\v$ is an analytically weak solution to \eqref{eqn_random_v2}. Hence $\u=\v+\z$ is an $\{\mathcal{F}_t\}_{t\in\R}$-adapted analytically weak solution to \eqref{eqn_stochatic_u-SHNSE} and the estimate \eqref{u_EN} holds for $\u$. 

On the other hand, fix $\bar{q}\in\N$. At the $\bar{q}$-th step using Proposition \ref{BIP}, we can produce two distinct pairs, one which we keep denoting as above and the other which we denote by $(\wi\v_{\bar{q}},\widetilde{\mathring{R}}_{\bar{q}})$ and satisfies \eqref{BIP-1}, namely
\begin{align*}
		\EN\v_{\bar{q}}-\wi\v_{\bar{q}}\EN_{\L^{2}_{x},2} \geq \delta_{\bar{q}}^{\frac12}, \;\; \supp_{t}(\v_{\bar{q}}-\wi\v_{\bar{q}})\subset \mathcal{I},
\end{align*}
with $\mathcal{I}=(10, 10+ 3m_{\bar{q}-1})$. Applying now the Proposition \eqref{Iterations}, we can build a new sequence $(\wi\v_{q},\widetilde{\mathring{R}}_{q})$ of approximate solutions which satisfy \eqref{v_q_C0}-\eqref{R_q_EN} and \eqref{v_diff_EN}, inductively. Analogously as above, this second sequence converges to an analytically weak solution (in the sense of Definition \ref{AWS}) $\wi\v$ to system \eqref{eqn_stochatic_u-SHNSE}. We also remark that for any $q\geq \bar{q}$,
\begin{align*}
	\supp_{t}(\v_{q}-\wi\v_{q})\subset \mathcal{I}+ \sum_{q=\bar{q}}^{\infty}\lambda_{q}^{-\frac85}\subset [9,\infty),
\end{align*}
where we choose $a$ to be even larger than chosen above, if necessary, and hence $\wi\v_{q}$ shares initial data with $\v_{q}$ for all $q$. Consequently, two solutions $\wi\v_{q}$ and $\v_{q}$ have the same initial data. However, the new solution $\wi\v$ differs from $\v$ because 
\begin{align*}
	\EN\v-\wi\v\EN_{\L^{2}_{x},2} & \geq \EN\v_{\bar{q}}-\wi\v_{\bar{q}}\EN_{\L^{2}_{x},2} - \sum_{q=\bar{q}+1}^{\infty}\EN\v_{q+1}-\v_{q}-(\wi\v_{q+1}- \wi\v_{q})\EN_{\L^{2}_{x},2}
	\nonumber\\ & \geq \EN\v_{\bar{q}}-\wi\v_{\bar{q}}\EN_{\L^{2}_{x},2} - (2\pi)^{\frac32} \sum_{q=\bar{q}+1}^{\infty}(\EN\v_{q+1}-\v_{q}\EN_{C^0,2}+\EN\wi\v_{q+1}- \wi\v_{q}\EN_{C^0,2})
	 \geq \delta_{\bar{q}}^{\frac12}- 2(2\pi)^{\frac32}\bar{M} \sum_{q=\bar{q}+1}^{\infty}\delta_{q+1}^{\frac12} > 0,
\end{align*}
where we choose $a$ sufficiently large such that $\delta_{\bar{q}}^{\frac12}> 2(2\pi)^{\frac32}\bar{M} \sum\limits_{q=\bar{q}+1}^{\infty}\delta_{q+1}^{\frac12}$. By changing the choice of time interval $\mathcal{I}$ and the choice of $\bar{q}$, we can easily generate infinitely many solutions. This completes the proof.
\end{proof}

\section{Convex integration scheme for \eqref{eqn_stochatic_u-SHNSE}}\setcounter{equation}{0}\label{sec4}
The main aim of this section is to give the proof of Proposition \ref{Iterations}, that is, for a given pair $(\v_{q}, \mathring{R}_q)$ satisfying the iterations \eqref{v_q_C0}-\eqref{R_q_EN}, there exists a pair $(\v_{q+1}, \mathring{R}_{q+1})$ satisfying the iterations \eqref{v_q_C0}-\eqref{R_q_EN} at the level $q+1$ and \eqref{v_diff_EN}.  First of all, we fix the parameters during the construction and continue with a mollification step. The construction of required pair is based on the work \cite{Lu+Zhu_Arxiv} where the authors adapted the ideas developed in the articles \cite{Buckmaster+Vicol_2019_Notes} and \cite{Hofmanova+Zhu+Zhu_Arxiv}. In order to define new iteration $v_{q+1}$, we define new perturbation $\omega_{q+1}$ such that $\v_{q+1}=\v_{q}+\omega_{q+1}$. More precisely, the construction of new amplitude functions $a_{\zeta}$ is an modification of the amplitude functions considered in \cite{Hofmanova+Zhu+Zhu_Arxiv} (see \eqref{varrho} and \eqref{ak} below), meanwhile, in accordance with \cite{Buckmaster+Vicol_2019_Notes}, we acquire an acceptable transport error by including the solutions for transport equations into Beltrami waves. We propose a cutoff function \eqref{cutoff} to control the growing amount of noise, since we need pathwise positive lower bounds of the solutions to the transport equations. Additionally, in Proposition \ref{Iterations}, we include pathwise estimations of the velocity field $\v_{q}$ in the inductive iteration. Finally, we construct new stress   $\mathring{R}_{q+1}$ with the help of $\v_{q}$ and $\omega_{q+1}$.

\subsection{Choice of parameters}\label{CoP}
We fix $b=5.75=\frac{23}{4}$ and choose any $0<\beta< \min\left\{ \frac{2(1-2\alpha)}{21}, \frac{1}{200}\right\}$ and $0<\varepsilon< \frac{1}{198.375}-\beta $. The parameter $a\in2^{3\N}$ has been chosen sufficiently large such that $a^{\frac13}\in\N$ and $224<a^{\varepsilon}$. The choice of $a$ guarantee $f(q)\in\N$ and both the inequalities in \eqref{Slope} hold.  The subsequent sections will make advantage of the aforementioned parameter requirements. We increase $a$ in the sequel to absorb various implicit and universal constants.

\subsection{Mollification}
To ensure smoothness during the construction, we substitute a mollified velocity field $\v_{\ell}$ for $\v_q$. To do this, we select a small parameter
\begin{align}\label{ell}
	\ell= \lambda_{q}^{-\frac85}.
\end{align}
 Let $\{\uppsi_{\epsilon}\}_{\epsilon>0}$ be a family of mollifiers on $\R^3$, and $\{\psi_{\epsilon}\}_{\epsilon>0}$ be a family of mollifiers with support in $(0,1)$. Note that we have considered one-sided mollifiers to reserve the adaptedness.  Let us define a mollification of $\v_{q},\mathring{R}_{q}$ and $\z_q$ in space and time by convolution as follows:
\begin{align*}
		\v_\ell:=(\v_{q}\ast_{x}\uppsi_{\ell})\ast_{t}\psi_{\ell}, \;\;\;  \mathring{R}_\ell&:=( \mathring{R}_{q}\ast_{x}\uppsi_{\ell})\ast_{t}\psi_{\ell},  \;\;\;   \text{ and }\;\;\; \z_\ell :=(\z_q\ast_{x}\uppsi_{\ell})\ast_{t}\psi_{\ell},
\end{align*}
where 
\begin{align}\label{Mol}
	\uppsi_{\ell}(\cdot):= \ell^{-3}\uppsi\left(\frac{\cdot}{\ell}\right) \;\;\text{ and }\;\;
	\psi_{\ell}(\cdot):= \ell^{-1}\psi\left(\frac{\cdot}{\ell}\right).
\end{align}
By the definition, it is immediate to see that $\v_{\ell},\mathring{R}_{\ell}$ and $\z_{\ell}$ are $\{\mathcal{F}_t\}_{t\in\R}$-adapted. From system \eqref{eqn_v_q}, we find that $(\v_{\ell},\mathring{R}_{\ell})$ satisfy the following system:
\begin{equation}\label{eqn_random_v_l}
	\left\{
	\begin{aligned}
		\partial_t\v_\ell +\nu(-\Delta)^{\alpha}\v_{\ell}  +\mathrm{div}\left((\v_\ell+\z_\ell) \otimes (\v_\ell+\z_\ell) \right) +\nabla p_\ell -\z_{\ell} &= \diver (\mathring{R}_{\ell}+R_{com1}),\\
			\mathrm{div}\;\v_\ell&=0,
	\end{aligned}
	\right.
\end{equation}
where the commutator stress
\begin{align}\label{R_com1}
	R_{com1}:= (\v_\ell+\z_\ell) \mathring{\otimes} (\v_\ell+\z_\ell) - \left(\left((\v_q+\z_q) \mathring{\otimes} (\v_q+\z_q)\right)\ast_{x}\uppsi_{\ell}\right)\ast_{t}\psi_{\ell},
\end{align}
and 
\begin{align}\label{p_l}
p_{\ell} :=	(p_q\ast_{x}\uppsi_{\ell})\ast_{t}\psi_{\ell} -\frac13\left(|\v_\ell+\z_\ell|^2- (|\v_q+\z_q|^2\ast_{x}\uppsi_{\ell})\ast_{t}\psi_{\ell} \right).
\end{align}

\subsection{Velocity perturbation}\label{V+P}
It is well known from the literature of convex integration methods that one needs to guarantee  of smoothness of new velocity during the construction.  A perturbation of $\v_{\ell}$ will thus be used to construct the new velocity field $\v_{q+1}$:
\begin{align*}
	\v_{q+1}=\v_{\ell}+\omega_{q+1},
\end{align*}
where the perturbation $\omega_{q+1}$ will be constructed by following the work \cite{Buckmaster+Vicol_2019_Notes}. Next, we discuss transport equations and time cutoffs which will be helpful in the sequel to construct $\omega_{q+1}$.

\subsubsection{Flow maps and cutoffs}
One must estimate the transport error in order to estimate the new stress $\mathring{R}_{q+1}$, that is, the perturbation $\omega_{q+1}$ is transported by the flow of the vector field $\partial_t+(\v_{\ell}+\z_{\ell})\cdot\nabla$. The process of estimating the transport error has been discussed in \cite{Buckmaster+Vicol_2019_Notes} and \cite{Lu+Zhu_Arxiv} for deterministic and stochastic cases, respectively. As discussed in \cite{Buckmaster+Vicol_2019_Notes,Lu+Zhu_Arxiv}, to obtain the appropriate estimate,  we must substitute the nonlinear phase $\zeta\cdot\Phi(t,x)$, where $\Phi$ is transported by the aforementioned described vector field, for the linear phase $\zeta\cdot x$ in the definition of the Beltrami wave $\mathbb{W}_{\zeta,\lambda}$ (see Appendix \ref{Beltrami}). In comparison to the work \cite{Buckmaster+Vicol_2019_Notes}, the authors in \cite{Lu+Zhu_Arxiv} expanded the definition of $\Phi$ for $\R$ with respect to the time variable and we are also following the same.

Let us fix 
\begin{align}\label{m_q}
	m_q=\lambda_{q+1}^{-\frac34}\lambda_{q}^{-\frac34}\delta_{q+1}^{-\frac14}\delta_{q}^{-\frac14}.
\end{align}
 We divide $[k,k+1]$ into temporal intervals of size $m_q$ for any integer $k\in\mathbb{Z}$, and then solve transport equations on these intervals. For $j\in\{0,1,\ldots,\lceil m_q^{-1}\rceil\}$, we define the adapted map $\Phi_{k,j}:\Omega\times\R^3\times[k+(j-1)m_q,k+(j+1)m_q]\to \R^3$ as the $\mathbb{T}^3$-periodic solution of  
\begin{equation}\label{Phi_kj}
	\left\{	\begin{aligned}
		\left(\partial_t+(\v_{\ell}+\z_\ell)\cdot\nabla\right)\Phi_{k,j}&=0,\\
		\Phi_{k,j}(k+(j-1)m_q,x)&=x.
	\end{aligned}
	\right.
\end{equation}
 In Appendix \ref{Transport}, we give the detailed proof of the estimates of $\Phi_{k,j}$. There are two main estimates for $\Phi_{k,j}$ which are as follows:
\begin{align}
	&\sup_{t\in[k+(j-1)m_q,k+(j+1)m_q]}\|\nabla\Phi_{k,j}(t)-\Id\|_{C^0_x} \leq  \lambda_{q}^{-\frac{8}{5}} <\!< 1,\label{Phi1}\\
	\frac12\leq&	\sup_{t\in[k+(j-1)m_q,k+(j+1)m_q]}\|\nabla\Phi_{k,j}(t)\|_{C^0_x}\leq 2.\label{Phi2}
\end{align}

\begin{remark}
To obtain the estimate \eqref{Phi1}, pointwise bounds for both $\v_{q}$ and $\z_{q}$ are required. For this purpose, we introduced the cutoff function $\z_{q}$ (see \eqref{cutoff}) and the inductive estimate \eqref{v_q_C1}. This approach differs from that in \cite{Hofmanova+Zhu+Zhu_Arxiv}, where the authors considered only moment bounds in their inductive estimates. Moreover, the approach in this article differs from that of \cite{Lu+Zhu_Arxiv}, primarily due to the following two aspects:
\begin{itemize}
\item First, as noted in Remark~\ref{Rem3.2}, we construct two smooth truncation functions, $\chi_q$ and $\widetilde{\chi}_q$, which are used to define the cutoff function $\z_q$. This modification has a cascading effect on the estimation of various error terms, thereby leading to an improvement in the Hölder regularity of the solutions.
\item Second, unlike \cite{Lu+Zhu_Arxiv}, where the authors employed different energy functions to construct infinitely many weak solutions, our approach relies on a bifurcating induction proposition (see Proposition \ref{BIP}) to generate infinitely many solutions.
\end{itemize}
\end{remark}

We also let $\eta$ be a non-negative bump function supported in $(-1, 1)$, which is $1$ on $(-\frac14,\frac14)$ and such
that the square of the shifted bump functions
\begin{align*}
	\eta_j(t)=\eta(m_q^{-1}t-j)
\end{align*}
form a partition of unity
\begin{align*}
	\sum_{j} \eta_j^2 (t)=1,
\end{align*}
for all $t\in[0,1]$. We then extend the definition of $\eta$ to the whole real line $\R$. In particular, let $\eta^{(k)}(t)=\eta(t-k)$, then we know $\supp \eta^{(k)} \subset (k-1,k+1)$. Similarly, the shifted bump functions
\begin{align*}
	\eta_{j,k}(t)=\eta(m_q^{-1}(t-k)-j)
\end{align*}
form a partition of unity
\begin{align}\label{eta_jk}
	\sum_{j} \eta_{j,k}^2 (t)=1,
\end{align}
for all $t\in[k,k+1]$.

\subsubsection{Amplitudes} 
We will now start the process to construct the velocity perturbation $\omega_{q+1}$.  Since Proposition \ref{Iterations} uses moment bounds of $\mathring{R}_q$ to the iterative estimates, we must modify the amplitude functions $a_{\zeta}$ such that Lemma \ref{GL} can be used in our context.  Specifically, we first define $\varrho$ as follows
\begin{align}\label{varrho}
	\varrho(t,x)&:= \sqrt{\ell^2+|\mathring{R}_{\ell}(t,x)|^2}.
\end{align}
 Then, we define the amplitude function 
\begin{align}\label{ak}
	a_{\zeta}(t,x):=a_{q+1,j,\zeta}(t,x):=  c_{\ast}^{-\frac12} \varrho^{\frac12}\cdot\eta_{j,k}(t) \cdot \Gamma_{\zeta}^{(j)}\left(\Id-\frac{c_{\ast} \mathring{R}_{\ell}}{\varrho}\right),\;\; \zeta\in\Lambda_{j} \;\; \text{and} \;\; t\in[k,k+1],
\end{align}
where $\Gamma_{\zeta}^{(j)}$ is introduced in Appendix \ref{Beltrami} and $\Lambda_j\subset \mathbb{S}^2\cap\mathbb{Q}^3$ is finite subset. By the definition of $\varrho$, we have
\begin{align*}
	\left\|\Id-\left(\Id-\frac{c_{\ast} \mathring{R}_{\ell}(t,x)}{\varrho}\right)\right\|_{C^0_{t,x}} \leq c_{\ast}.
\end{align*}
Consequently, $\Id-\frac{c_{\ast} \mathring{R}_{\ell}}{\varrho}$ belongs to the domain of the function $\Gamma_{\zeta}^{(j)}$ and we obtain from \eqref{eta_jk} and Lemma \ref{GL} that
\begin{align}\label{B1}
c_{\ast}^{-1}	\varrho \Id - \mathring{R}_{\ell} &=\frac{1}{2}\sum_{j}\sum_{\zeta\in\Lambda_j}  a_{\zeta}^2 (\Id-\zeta\otimes \zeta),
\end{align}
holds pointwise.  Also, it is adequate to think of index sets $\Lambda_0$ and $\Lambda_1$ as having 12 elements, and we write $\Lambda_j=\Lambda_{j\mod 2}$ for $j\in\mathbb{Z}$.

\subsubsection{Construction of velocity perturbation $\omega_{q+1}$} 
Firstly, we define the principal part $\omega_{q+1}^{(p)}$ of the velocity perturbation $\omega_{q+1}$ using \eqref{ak}. For this purpose, we make use of Beltrami wave as presented in \cite{Buckmaster+Vicol_2019_Notes} (see \cite{Lu+Zhu_Arxiv} also), which we have recalled in Appendix \ref{Beltrami}. More precisely, for $t \in [k, k + 1]$, we define
\begin{align*}
	\omega_{\zeta}(t,x):= a_{q+1,j,\zeta}(t,x) B_{\zeta}e^{i\lambda_{q+1}\zeta\cdot\Phi_{k,j}(t,x)},
\end{align*}
and 
\begin{align}\label{wp}
	\omega_{q+1}^{(p)}(t,x)&:=\sum_{j}\sum_{\zeta\in\Lambda_j} a_{q+1,j,\zeta}(t,x)B_{\zeta}e^{i\lambda_{q+1}\zeta\cdot\Phi_{k,j}(t,x)} \nonumber\\
	& =\sum_{j}\sum_{\zeta\in\Lambda_j} c_{\ast}^{-\frac12} \varrho^{\frac12}(t)\cdot\eta_{j,k}(t) \cdot \Gamma_{\zeta}^{(j)}\left(\Id- c_{\ast}\frac{ \mathring{R}_{\ell}(t,x)}{\varrho(t)}\right)\cdot B_{\zeta}e^{i\lambda_{q+1}\zeta\cdot\Phi_{k,j}(t,x)},
\end{align}
where $B_{\zeta}$ is an appropriate vector defined in Appendix \ref{Beltrami}. Since $a_{\zeta}$ and $\Phi_{k,j}(t,x)$ are $\{\mathcal{F}_{t}\}_{t\in\R}$-adapted, we deduce that $\omega_{q+1}^{(p)}$ is $\{\mathcal{F}_{t}\}_{t\in\R}$-adapted.

Next, we define the incompressible corrector $\omega_{q+1}^{(c)}$. In order to make $\omega_{q+1}^{(p)}$ perfect curl and so divergence-free, we want to add a corrector to it. The following scalar phase function is helpful in achieving this
\begin{align*}
	\phi_{\zeta}(t,x):= e^{i\lambda_{q+1} \zeta \cdot (\Phi_{k,j}(t,x)-x)}, \;\;\; t\in[k,k+1].
\end{align*}
Next we define 
\begin{align*}
	\Wb_{\zeta}(x) :=B_{\zeta}e^{i\lambda_{q+1}\zeta\cdot x}.
\end{align*}
Since $\curl\Wb_{\zeta}=\lambda_{q+1}\Wb_{\zeta}$ (see Appendix \ref{Beltrami}), and $\phi_{\zeta}$ and $a_{\zeta}$ are scalar functions, we have 
\begin{align*}
	a_\zeta \phi_{\zeta} \Wb_\zeta =\frac{1}{\lambda_{q+1}} \curl\left(a_\zeta\phi_{\zeta}\Wb_\zeta\right)-\frac{1}{\lambda_{q+1}}\nabla\left(a_\zeta\phi_{\zeta}\right)\times\Wb_{\zeta}.
\end{align*}
We therefore define
\begin{align*}
	\omega_{\zeta}^{(c)}(t,x)&:= \frac{1}{\lambda_{q+1}} \nabla\left(a_\zeta\phi_{\zeta}\right)\times B_\zeta e^{i\lambda_{q+1}\zeta\cdot x} \nonumber\\
	& = \left(\frac{\nabla a_\zeta}{\lambda_{q+1}}+i a_\zeta (\nabla\Phi_{k,j}(t,x)-\Id)\zeta \right) \Wb_{\zeta}(\Phi_{k,j}(t,x)), \;\;\; t\in[k,k+1].
\end{align*}
The incompressibility corrector $\omega_{q+1}^{(c)}$ of perturbation is defined by
\begin{align}\label{wc}
	\omega_{q+1}^{(c)}(t,x)&:= \sum_{j}\sum_{\zeta\in\Lambda_j} \omega_{\zeta}^{(c)}(t,x)
\end{align}
for all $t\in[k,k+1]$, $k\in\mathbb{Z}$. Since the coefficients $a_{\zeta}$ and $\Phi_{k,j}(t,x)$ are $\{\mathcal{F}_{t}\}_{t\in\R}$-adapted, we deduce that $\omega_{q+1}^{(c)}$ is $\{\mathcal{F}_{t}\}_{t\in\R}$-adapted.

Finally, in view of \eqref{wp} and \eqref{wc}, the new velocity perturbation is defined as
\begin{align}\label{w_q+1_Curl}
	\omega_{q+1} := \omega_{q+1}^{(p)}+\omega_{q+1}^{(c)} = \frac{1}{\lambda_{q+1}}\sum_{j}\sum_{\zeta\in\Lambda_j}\curl \left(a_\zeta \phi_{\zeta} \Wb_\zeta \right),
\end{align}
and so clearly $\omega_{q+1}$ is mean zero, divergence-free and $\{\mathcal{F}_t\}_{t\in\R}$-adapted. At last, we define the new velocity field $\v_{q+1}$ as
\begin{align*}
	\v_{q+1}:= \v_{\ell} +\omega_{q+1}.
\end{align*}
Thus, by the previous discussion, it is also $\{\mathcal{F}_t\}_{t\in\R}$-adapted.

\subsection{New Reynolds stress $ \mathring{R}_{q+1}$}\label{NRS}
We have the system \eqref{eqn_v_q} at the level $q+1$ as follows:
\begin{equation}\label{eqn_v_q+1}
	\left\{
	\begin{aligned}
		\partial_t\v_{q+1}  +\nu(-\Delta)^{\alpha} \v_{q+1}   +\mathrm{div}\left((\v_{q+1}+\z_{q+1}) \otimes (\v_{q+1}+\z_{q+1}) \right)  -\z_{q+1} +\nabla p_{q+1} &=\diver \mathring{R}_{q+1}, \\
		\mathrm{div}\;\v_{q+1}&=0.
	\end{aligned}
	\right.
\end{equation}
Subtracting the system \eqref{eqn_random_v_l} from the system \eqref{eqn_v_q+1}, we get 
\begin{align*}
	& \diver \mathring{R}_{q+1} - \nabla p_{q+1} 
	\nonumber\\ & = \partial_t\v_{q+1} +\nu(-\Delta)^{\alpha} \v_{q+1}   +\mathrm{div}\left((\v_{q+1}+\z_{q+1}) \otimes (\v_{q+1}+\z_{q+1}) \right)  -\z_{q+1} 
	\nonumber\\ & = \partial_t(\v_{\ell} +\omega_{q+1}) +\nu(-\Delta)^{\alpha} (\v_{\ell} +\omega_{q+1})   +\mathrm{div}\left((\v_{\ell} +\omega_{q+1}+\z_{q+1}) \otimes (\v_{\ell} +\omega_{q+1}+\z_{q+1}) \right) 
	   -\z_{q+1} 
	\nonumber\\ & =   -\nabla p_\ell + \diver (\mathring{R}_{\ell}+R_{com1})+(\z_{\ell} -\z_{q+1}) 
	 + \partial_t \omega^{(p)}_{q+1} + \partial_t \omega^{(c)}_{q+1} +\nu(-\Delta)^{\alpha} \omega_{q+1}  
	\nonumber\\ & \quad + \diver \big(\v_{\ell}\otimes  \omega^{(p)}_{q+1} + \v_{\ell}\otimes \omega^{(c)}_{q+1}+\omega_{q+1}\otimes \v_{\ell}+  \omega^{(p)}_{q+1}\otimes  \omega^{(p)}_{q+1} + \omega^{(p)}_{q+1}\otimes \omega^{(c)}_{q+1}   +  \omega^{(c)}_{q+1}\otimes\omega_{q+1} \big)
	\nonumber\\& \quad + \diver\big(\v_{q+1}\otimes(\z_{q+1}-\z_{\ell})  +\omega_{q+1}\otimes \z_{\ell} + (\z_{q+1}-\z_{\ell})\otimes \v_{q+1}  +\z_{\ell}\otimes \omega_{q+1}\big)
 	 \nonumber\\&\quad +\diver\big( (\z_{q+1}-\z_{\ell})\otimes \z_{\ell} + \z_{q+1}\otimes(\z_{q+1}-\z_{\ell})\big)
 	 \nonumber\\& =  \underbrace{\nu(-\Delta)^{\alpha} \omega_{q+1} }_{\diver R_{lin}} + \underbrace{(\partial_t+(\v_{\ell}+\z_{\ell})\cdot\nabla)\omega^{(p)}_{q+1}}_{\diver R_{trans}} + \underbrace{(\omega_{q+1}\cdot\nabla)(\v_{\ell}+\z_{\ell})}_{\diver R_{Nash}}
 	 +\underbrace{\diver\big(\omega_{q+1}^{(p)}\otimes\omega_{q+1}^{(p)}+\mathring{R}_{\ell}\big)}_{\diver R_{osc}+\nabla p_{osc}}
 	    \nonumber\\& \quad + \underbrace{(\partial_t+(\v_{\ell}+\z_{\ell})\cdot\nabla)\omega^{(c)}_{q+1} + \diver (\omega_{q+1}^{(c)}\otimes\omega_{q+1}+\omega_{q+1}^{(p)}\otimes\omega_{q+1}^{(c)})}_{\diver R_{corr}+\nabla p_{corr}}
 	 \nonumber\\& \quad + \underbrace{\diver \big( \v_{q+1}\otimes(\z_{q+1}-\z_{\ell})   + (\z_{q+1}-\z_{\ell})\otimes \v_{q+1} + \z_{q+1}\otimes\z_{q+1} -\z_{\ell}\otimes \z_{\ell}\big) -(\z_{q+1}-\z_{\ell})}_{\diver R_{com2}+\nabla p_{com2}} 
 	 \nonumber\\& \quad + \diver R_{com1} -\nabla p_{\ell}.
\end{align*}
Here $R_{com1}$ and $p_{\ell}$ are defined in \eqref{R_com1} and \eqref{p_l}, respectively, and by making use of the operator $\mathcal{R}$ discussed in Subsection \ref{IDO-R}, we define
\begin{align*}
	R_{lin}&:=  \mathcal{R}\bigg[\nu(-\Delta)^{\alpha} \omega_{q+1}\bigg],\\
R_{trans}&:= \mathcal{R}\bigg[(\partial_t+(\v_{\ell}+\z_{\ell})\cdot\nabla)\omega^{(p)}_{q+1}\bigg], \\
R_{Nash} &:= \mathcal{R}\left[(\omega_{q+1}\cdot\nabla)(\v_{\ell}+\z_{\ell})\right], \\ 
R_{corr}& :=  \mathcal{R}\bigg[(\partial_t+(\v_{\ell}+\z_{\ell})\cdot\nabla)\omega^{(c)}_{q+1}\bigg] + \omega_{q+1}^{(c)}\mathring{\otimes}\omega_{q+1}+\omega_{q+1}^{(p)}\mathring{\otimes}\omega_{q+1}^{(c)}, \\
R_{com2}& :=   \v_{q+1}\mathring{\otimes}(\z_{q+1}-\z_{\ell})   + (\z_{q+1}-\z_{\ell})\mathring{\otimes} \v_{q+1} + \z_{q+1}\mathring{\otimes}\z_{q+1}
 - \z_{\ell}\mathring{\otimes}\z_{\ell} - \mathcal{R}(\z_{q+1}-\z_{\ell}), \\
 p_{corr}& :=  \frac{1}{3}\left(2\omega_{q+1}^{(p)}\cdot\omega_{q+1}^{(c)} +\left|\omega_{q+1}^{(c)}\right|^2\right), \\
 p_{com2}& := \frac13\big(\v_{q+1}\cdot(\z_{q+1}-\z_{\ell})   + (\z_{q+1}-\z_{\ell})\cdot \v_{q+1} + |\z_{q+1}|^2-|\z_{\ell}|^2)\big).
\end{align*}

In order to define $R_{osc}$ and $p_{osc}$, we first pay attention that for $j,j'$ such that $|j-j^{\prime}|\geq2$, we have $\eta_{j,k}(t)\eta_{j',k}(t)=0$. Second, we have $\Lambda_{j}\cap\Lambda_{j'}=\varnothing$ for $|j-j'|=1$, and by Lemma \ref{Beltrami_P} we have $$\diver(\Wb_{\zeta}\otimes\Wb_{\zeta'}+\Wb_{\zeta'}\otimes\Wb_{\zeta})=\nabla(\Wb_{\zeta}\cdot\Wb_{\zeta'})$$ (using $(A\cdot\nabla) B+(B\cdot\nabla)A=\nabla(A\cdot B)-A\times\nabla\times B-B\times\nabla\times A$). Therefore, in view of Lemma \ref{GL} and \eqref{B1}, we obtain
\begin{align*}
 &	\diver\big(\omega_{q+1}^{(p)}\otimes\omega_{q+1}^{(p)}+\mathring{R}_{\ell}\big)
	\nonumber\\ & =\diver\bigg(\sum_{j,j',\zeta,\zeta'}\omega_{\zeta}\otimes\omega_{\zeta'}+\mathring{R}_{\ell}\bigg)
	\nonumber\\ & = \diver \left(\sum_{j,\zeta}c_{\ast}^{-1}\varrho\cdot \eta_{j,k}^2 \{\Gamma_{\zeta}^{(j)}(\Id-c_{\ast}\varrho^{-1}\mathring{R}_{\ell})\}^2B_{\zeta}\otimes B_{-\zeta} + \mathring{R}_{\ell}\right)
	+ \diver \bigg(\sum_{j,j', \zeta+\zeta'\neq0}a_{\zeta}a_{\zeta'}\Wb_{\zeta}(\Phi_{k,j})\otimes\Wb_{\zeta'}(\Phi_{k,j'})\bigg)
	\nonumber\\ & = \diver \bigg(c_{\ast}^{-1}\varrho(\Id-c_{\ast}\varrho^{-1}\mathring{R}_{\ell})+\mathring{R}_{\ell}\bigg) +   \diver \bigg(\sum_{j,j', \zeta+\zeta'\neq0}a_{\zeta}a_{\zeta'} \phi_{\zeta}\phi_{\zeta'}\Wb_{\zeta}\otimes\Wb_{\zeta'}\bigg)
	\nonumber\\ & = \nabla(c_{\ast}^{-1}\varrho) +    \sum_{j,j', \zeta+\zeta'\neq0}\Wb_{\zeta}\otimes\Wb_{\zeta'} \nabla\left(a_{\zeta}a_{\zeta'} \phi_{\zeta}\phi_{\zeta'}\right) + \frac12 \sum_{j,j', \zeta+\zeta'\neq0}a_{\zeta}a_{\zeta'} \phi_{\zeta}\phi_{\zeta'}\nabla( \Wb_{\zeta}\cdot\Wb_{\zeta'})
	\nonumber\\ & = \nabla(c_{\ast}^{-1}\varrho) + \frac12\nabla\left( \sum_{j,j', \zeta+\zeta'\neq0}a_{\zeta}a_{\zeta'} \phi_{\zeta}\phi_{\zeta'}( \Wb_{\zeta}\cdot\Wb_{\zeta'})\right)
	+    \sum_{j,j', \zeta+\zeta'\neq0}\biggl\{\Wb_{\zeta}\otimes\Wb_{\zeta'}-\frac{\Wb_{\zeta}\cdot\Wb_{\zeta'}}{2}  \Id \biggr\} \nabla\left(a_{\zeta}a_{\zeta'} \phi_{\zeta}\phi_{\zeta'}\right).
\end{align*}
Therefore
\begin{align*}
	R_{osc} &:=  \sum_{j,j', \zeta+\zeta'\neq0} \mathcal{R} \bigg[\biggl\{\Wb_{\zeta}\otimes\Wb_{\zeta'}-\frac{\Wb_{\zeta}\cdot\Wb_{\zeta'}}{2} \Id \biggr\} \nabla\left(a_{\zeta}a_{\zeta'} \phi_{\zeta}\phi_{\zeta'}\right)\bigg],\\
	p_{osc} &:= c_{\ast}^{-1}\varrho + \frac12  \sum_{j,j', \zeta+\zeta'\neq0}a_{\zeta}a_{\zeta'} \phi_{\zeta}\phi_{\zeta'}( \Wb_{\zeta}\cdot\Wb_{\zeta'}).
\end{align*}

Finally, we have 
\begin{align}\label{R_q+1-split}
	\mathring{R}_{q+1} = R_{lin}+ R_{trans}+R_{Nash}+R_{osc}+R_{Corr}+R_{com1}+R_{com2},
\end{align}
and 
\begin{align*}
	p_{q+1}=p_{\ell} - p_{osc}-p_{corr} - p_{com2}.
\end{align*}
Note that $\omega_{q+1}^{(p)}$, $\omega_{q+1}^{(c)}$, $\v_{\ell}$, $\z_{\ell}$, $\v_{q+1}$ and $\z_{q+1}$ are $\{\mathcal{F}_t\}_{t\in\R}$-adapted, therefore $R_{q+1}$ and $p_{q+1}$ are also $\{\mathcal{F}_t\}_{t\in\R}$-adapted.

\section{Inductive estimates and proof of Proposition \ref{Iterations}}\setcounter{equation}{0}\label{sec5}
The main aim of this section is to demonstrate the proof of Proposition \ref{Iterations} by proving the all necessary inductive estimates for $\v_{q+1}$ and $\mathring{R}_{q+1}$. In the next proposition, we estimate the $C^N_{t,x}$-norm of the amplitude function $a_{\zeta}$ defined in Subsection \ref{V+P} (see \eqref{ak} above).

\begin{proposition}\label{5.1}
	Let $a_{\zeta}$ be given by \eqref{ak}. Then, we have 
	\begin{align}
    \|a_{\zeta}\|_{C_{t,x}^{0}} &   \lesssim \delta_{q+1}^{\frac12} + \|\mathring{R}_{q}\|_{C^{0}_{[t-1,t+1],x}}^{\frac12} \label{ak_0} \\
    \|a_{\zeta}\|_{C_{t,x}^{1}} &   \lesssim m_q^{-1}   (\delta_{q+1}^{\frac12} + \|\mathring{R}_{q}\|^{\frac12}_{C^{0}_{[t-1,t+1],x}} ),  \label{ak_1} \\
    \|a_{\zeta}\|_{C_{t,x}^{2}} &   \lesssim m_q^{-1}\lambda_q^{\frac{92}{15}} (\delta_{q+1}^{\frac12}+ \|\mathring{R}_{q}\|_{C^0_{[t-1,t+1],x}}^{\frac12}),  \label{ak_2} \\
		\|a_{\zeta}\|_{C_{t,x}^{N}} & \lesssim  m_q^{-N} \lambda_q^{3}
        , && \hspace{-20mm} \text{ for any } N\geq0, \label{ak_N}
	\end{align}
	where the implicit constants are independent of the parameter $q$.
\end{proposition}

\begin{proof}
First observe that,  for big enough $a>\max\left\{\frac{1}{(3rL^2)^{\frac{5}{8}}}, \frac{1}{2^{\frac{5}{46}}} \right\}$ and  $\beta <\frac{1}{200}$, we have $\ell\leq \delta_{q+1}$ for any $q\geq 0$. Also, note that $m_q^{-1} = \lambda_q^{\frac{27}{8}(\frac{3}{2}-\beta)}=\lambda_q^{5 + (\frac{1}{16} - \frac{27}{8}\beta)}$.
\vskip 2mm
\noindent
\textbf{Step 1.}  Here we give the proof of \eqref{ak_0}.
 From \eqref{varrho}, \eqref{ak}  and \eqref{C^n_M}, we get
\begin{align}\label{varrho_0}
	\|a_{\zeta}\|_{C_{t,x}^{0}} & \lesssim 	\|\varrho^{\frac{1}{2}}\|_{C^0_{t,x}} \lesssim \ell^{\frac12}+ \|\mathring{R}_{\ell}\|_{C^0_{t,x}}^{\frac12} \lesssim \delta_{q+1}^{\frac12}+ \|\mathring{R}_{q}\|_{C^0_{[t-1,t+1],x}}^{\frac12}.
	\end{align}

\vskip 2mm
\noindent
\textbf{Step 2.} Here we give the proof of \eqref{ak_1}. 
A straightforward calculation, in view of \eqref{varrho}, reveals that
\begin{align}
\| \nabla_{t,x}(\varrho^{\frac12})\|_{C^{0}_{t,x}} & \lesssim \ell^{-\frac{3}{2}} \|\mathring{R}_{q}\|_{C^{0}_{[t-1,t+1],x}},\label{varrho_half_C1} \\
\| \nabla_{t,x}\varrho\|_{C^{0}_{t,x}} & \lesssim \ell^{-1} \|\mathring{R}_{q}\|_{C^{0}_{[t-1,t+1],x}}.\label{varrho_C1}
\end{align}
Using $|\frac{\mathring{R}_{\ell}}{\varrho}|\leq 1$, $\varrho\geq \ell$, \eqref{C^n_M} and \eqref{varrho_C1}, we get
\begin{align}
\left\|\Gamma_{\zeta}^{(j)}\left(\Id-\frac{c_{\ast} \mathring{R}_{\ell}}{\varrho}\right)\right\|_{C^1_{t,x}} & \lesssim  \ell^{-2} \|\mathring{R}_{q}\|_{C^{0}_{[t-1,t+1],x}}. \label{Gamma-estimate}
\end{align}
From \eqref{ak}, we have 
\begin{align} 
	\nabla_{t,x} a_{\zeta}(t,x) & =  c_{\ast}^{-\frac12} \nabla_{t,x}(\varrho^{\frac12})\cdot\eta_{j,k}(t) \cdot \Gamma_{\zeta}^{(j)}\left(\Id-\frac{c_{\ast} \mathring{R}_{\ell}}{\varrho}\right) + c_{\ast}^{-\frac12} \varrho^{\frac12}\cdot \partial_t (\eta_{j,k}(t)) \cdot \Gamma_{\zeta}^{(j)}\left(\Id-\frac{c_{\ast} \mathring{R}_{\ell}}{\varrho}\right)
	\nonumber\\ & \quad  + c_{\ast}^{-\frac12}  \varrho^{\frac12} \cdot\eta_{j,k}(t) \cdot \nabla_{t,x} \left[\Gamma_{\zeta}^{(j)}\left(\Id-\frac{c_{\ast} \mathring{R}_{\ell}}{\varrho}\right)\right].
\end{align}
therefore, from \eqref{varrho_half_C1}, \eqref{varrho_0},   and \eqref{Gamma-estimate}, we have
\begin{align}
	\|a_{\zeta}\|_{C_{t,x}^{1}} & \lesssim \| \nabla_{t,x}(\varrho^{\frac12})\|_{C_{t,x}^{0}} +  m_q^{-1}\|\varrho^{\frac12}\|_{C_{t,x}^{0}}  +  \| \varrho^{\frac12} \|_{C_{t,x}^{0}} \left\|\Gamma_{\zeta}^{(j)}\left(\Id-\frac{c_{\ast} \mathring{R}_{\ell}}{\varrho}\right)\right\|_{C_{t,x}^{1}}
	\nonumber\\ & \lesssim \ell^{-\frac32} \|\mathring{R}_{q}\|_{C^{0}_{[t-1,t+1],x}} +  m_q^{-1}(\ell^{\frac12} + \|\mathring{R}_{q}\|^{\frac12}_{C^{0}_{[t-1,t+1],x}})  + (\ell^{\frac12} + \|\mathring{R}_{q}\|^{\frac12}_{C^{0}_{[t-1,t+1],x}} )  \ell^{-2} \|\mathring{R}_{q}\|_{C^{0}_{[t-1,t+1],x}}
	\nonumber\\ & \lesssim  m_q^{-1} (\ell^{\frac12} + \|\mathring{R}_{q}\|^{\frac12}_{C^{0}_{[t-1,t+1],x}})  +   (\ell^{\frac12} + \|\mathring{R}_{q}\|^{\frac12}_{C^{0}_{[t-1,t+1],x}} )  \ell^{-2} \|\mathring{R}_{q}\|_{C^{0}_{[t-1,t+1],x}}
	\nonumber\\ &  \lesssim   (\ell^{\frac12} + \|\mathring{R}_{q}\|^{\frac12}_{C^{0}_{[t-1,t+1],x}} ) (m_q^{-1} +  \ell^{-2} \|\mathring{R}_{q}\|_{C^{0}_{[t-1,t+1],x}})
    \nonumber\\ &  
    \lesssim   (\delta_{q+1}^{\frac12} + \|\mathring{R}_{q}\|^{\frac12}_{C^{0}_{[t-1,t+1],x}} )  (m_q^{-1} + \lambda_q^{\frac{58}{15}}) \lesssim   m_q^{-1} (\delta_{q+1}^{\frac12} + \|\mathring{R}_{q}\|^{\frac12}_{C^{0}_{[t-1,t+1],x}} )  ,
\end{align}
where we have used  \eqref{ell} and \eqref{R_q_C0} in the final inequality.

\vskip 2mm
\noindent
\textbf{Step 3.} Here we give the proof of \eqref{ak_2} and \eqref{ak_N}. Let $\Psi_1(y)=\sqrt{\ell^2+y^2}$. We have that $|D^m\Psi_1(y)|\lesssim \ell^{-m+1}$, for $m\geq 1$. Using Lemma \ref{Diff_Comp}, we obtain for any $N\geq1$
	\begin{align}\label{varrho1}
		 \|\varrho\|_{C_{t,x}^{N}}
		 &  \lesssim   \left\|\sqrt{\ell^2+ |\mathring{R}_{\ell}|^2}\right\|_{C^0_{t,x}}  + \|D\Psi_1\|_{C^{0}}\|\mathring{R}_{\ell}\|_{C^N_{t,x}} +  \|D\Psi_1\|_{C^{N-1}_{t,x}}\|\mathring{R}_{\ell}\|^{N}_{C^{1}_{t,x}} 
		\nonumber\\ &   \lesssim  \ell+\|\mathring{R}_{q}\|_{C^0_{[t-1,t+1],x}}  + \ell^{-N}\|\mathring{R}_{q}\|_{C^0_{[t-1,t+1],x}}  +  \ell^{-2N+1} \|\mathring{R}_{q}\|^{N}_{C^{0}_{[t-1,t+1],x}}
        \nonumber\\ &   \lesssim  \lambda_q^{-\frac85} + \lambda_q^{\frac23}  + \lambda_q^{\frac85 N}\lambda_q^{\frac23}  +  \lambda_q^{\frac{8}{5}(2N-1)} \lambda_q^{\frac23 N}
           \lesssim     \lambda_q^{\frac{58}{15} N -\frac85},
	\end{align}
    and for $N=0$
\begin{align}\label{varrho1-0}
		 \|\varrho\|_{C_{t,x}^{0}}
		    \lesssim  \ell+\|\mathring{R}_{q}\|_{C^0_{[t-1,t+1],x}} \lesssim \lambda_q^{\frac{2}{3}}.
	\end{align}
	 Let $\Psi_2(y)=y^{\frac12}$. We have $|D^{m}\Psi_2(y)|\lesssim |y|^{\frac12-m}$, for $m=1,\ldots,N$, and use Lemma \ref{Diff_Comp}, \eqref{varrho_0}, \eqref{varrho1} and $\varrho\geq \ell $ to derive for $m\geq1$ 
	\begin{align}\label{varrho_m}
		\|\varrho^{\frac12}\|_{C^m_{t,x}}  
		& \lesssim \|\varrho^{\frac12}\|_{C^0_{t,x}}+ \ell^{-\frac12}\|\varrho\|_{C^m_{t,x}}+ \ell^{\frac12-m} \|\varrho\|^m_{C^1_{t,x}}
		\nonumber \\ & \lesssim\ell^{\frac12} +  \|\mathring{R}_{q}\|_{C^0_{[t-1,t+1],x}}^{\frac12} 
		+  \lambda_q^{\frac{4}{5}} \lambda_q^{\frac{58}{15}m-\frac{8}{5}}  +  \lambda_q^{\frac{8}{5}m-\frac{4}{5}}  \lambda_q^{\frac{58}{15}m-\frac{8}{5}m} 
          \lesssim   \lambda_q^{\frac{58}{15}m-\frac{4}{5}}.
	\end{align} 
Let $\Psi_3(y)=\frac{1}{y}$, for $y\geq \ell$. We have $|D^m\Psi_3(y)|\lesssim \ell^{-(m+1)}$, for $m\geq0$. We estimate
	\begin{align}\label{5.6}
		\left\|\frac{1}{\varrho}\right\|_{C^0_{t,x}}\lesssim \ell^{-1} = \lambda_q^{\frac{8}{5}},
	\end{align} 
	and for $m\geq1$, using Lemma \ref{Diff_Comp} and \eqref{varrho1}, we have
	\begin{align}\label{5.7}
		\left\|\frac{1}{\varrho}\right\|_{C^{m}_{t,x}} 
		 & \lesssim \left\|\frac{1}{\varrho}\right\|_{C^0_{t,x}} + \ell^{-2}\|\varrho\|_{C^m_{t,x}} + \ell^{-(m+1)}\|\varrho\|^{m}_{C^1_{t,x}}
		\nonumber\\ & \lesssim \lambda_q^{\frac85} +  \lambda_q^{\frac{16}{5}}  \lambda_q^{\frac{58}{15}m-\frac{8}{5}} + \lambda_q^{\frac85(m+1)}  \lambda_q^{\frac{58}{15}m-\frac{8}{5}m} \lesssim \lambda_q^{\frac{58}{15}m + \frac{8}{5}}.
	\end{align} 
	For $m\geq1$, using chain rule, \eqref{R_q_C0} and \eqref{5.6}-\eqref{5.7}, we have
	\begin{align}\label{5.8}
		\left\|\frac{\mathring{R}_\ell}{\varrho}\right\|_{C^{m}_{t,x}} 
		& \lesssim    \sum_{k=0}^{m} \|\mathring{R}_{\ell}\|_{C^{k}_{t,x}}\left\|\frac{1}{\varrho}\right\|_{C^{m-k}_{t,x}}  
    \lesssim    \sum_{k=0}^{m} \ell^{-k} \|\mathring{R}_{q}\|_{C^0_{[t-1,t+1],x}} \lambda_q^{\frac{58}{15}(m-k) + \frac{8}{5}}
        \nonumber\\  &
        \lesssim    \sum_{k=0}^{m} 
        \lambda_q^{\frac85 k} \lambda_q^{\frac23} \lambda_q^{\frac{58}{15}(m-k) + \frac{8}{5}}
       \lesssim     \lambda_q^{\frac{58}{15}m + \frac{34}{15}}.
	\end{align}
Using $|\frac{\mathring{R}_l}{\varrho}|\leq 1 $, $\varrho \geq \ell$ and \eqref{R_q_C0}, we have 
\begin{align}
    \left\|\frac{\nabla_{t,x}\mathring{R}_{\ell}}{\varrho}\right\|^{m}_{C^0_{t,x}} & \lesssim \ell^{-2m}\|\mathring{R}_{q}\|_{C^0_{[t-1,t+1],x}}^m \lesssim \lambda_q^{\frac{16}{5}m} \lambda_q^{\frac{2}{3}m} =\lambda_q^{\frac{58}{15}m} , \label{5.9}\\ 
    \left\|\frac{\mathring{R}_{\ell}}{\varrho^2}\right\|^{m}_{C^0_{t,x}} & \lesssim \ell^{-m} = \lambda_q^{\frac{8}{5}m}.\label{5.10}
\end{align}
	
    For $m\geq1$, using Lemma \ref{Diff_Comp}, \eqref{C^n_M}, \eqref{5.8}, \eqref{5.9}, \eqref{5.10} and \eqref{varrho1}, we find
	\begin{align}\label{Gamma_m}
		 	\left\|\Gamma_{\zeta}^{(j)}\left(\Id-\frac{\mathring{R}_{\ell}}{\varrho}\right)\right\|_{C^{m}_{t,x}} 
		 &   \lesssim \left\|\frac{\mathring{R}_\ell}{\varrho}\right\|_{C^{m}_{t,x}} +\left\|\frac{\nabla_{t,x}\mathring{R}_{\ell}}{\varrho}\right\|^{m}_{C^0_{t,x}} +\left\|\frac{\mathring{R}_{\ell}}{\varrho^2}\right\|^{m}_{C^0_{t,x}}\|\varrho\|^{m}_{C^1_{t,x}}
        \nonumber\\ &   \lesssim \lambda_q^{\frac{58}{15}m + \frac{34}{15}} + \lambda_q^{\frac{58}{15}m } +  \lambda_q^{\frac{8}{5}m } \lambda_q^{\frac{58}{15} m -\frac85 m} \lesssim \lambda_q^{\frac{58}{15}m + \frac{34}{15}},
	\end{align}
	and for $m=0$, using \eqref{C^n_M}, we find
	\begin{align}\label{Gamma_0}
		\left\|\Gamma_{\zeta}^{(j)}\left(\Id-\frac{\mathring{R}_{\ell}}{\varrho}\right)\right\|_{C^{0}_{t,x}} \lesssim 1.
	\end{align}
    By chain rule, and using \eqref{varrho_0}, \eqref{varrho_m} and \eqref{Gamma_m}-\eqref{Gamma_0}, we have
    \begin{align}
        \|a_{\zeta}\|_{C^{2}_{t,x}}   & \lesssim  \|\varrho^{\frac12}\|_{C^0_{t,x}}\|\eta_{j,k}\|_{C^0_{t,x}}\left\|\Gamma_{\zeta}^{(j)}\left(\Id-\frac{\mathring{R}_{\ell}}{\varrho}\right)\right\|_{C^{2}_{t,x}} + \|\varrho^{\frac12}\|_{C^0_{t,x}}\|\eta_{j,k}\|_{C^1_{t,x}}\left\|\Gamma_{\zeta}^{(j)}\left(\Id-\frac{\mathring{R}_{\ell}}{\varrho}\right)\right\|_{C^{1}_{t,x}} 
        \nonumber\\ & \quad + \|\varrho^{\frac12}\|_{C^0_{t,x}}\|\eta_{j,k}\|_{C^2_{t,x}}\left\|\Gamma_{\zeta}^{(j)}\left(\Id-\frac{\mathring{R}_{\ell}}{\varrho}\right)\right\|_{C^{0}_{t,x}} + 
\|\varrho^{\frac12}\|_{C^1_{t,x}}\|\eta_{j,k}\|_{C^0_{t,x}}\left\|\Gamma_{\zeta}^{(j)}\left(\Id-\frac{\mathring{R}_{\ell}}{\varrho}\right)\right\|_{C^{1}_{t,x}}
\nonumber\\ & \quad + \|\varrho^{\frac12}\|_{C^1_{t,x}}\|\eta_{j,k}\|_{C^1_{t,x}}\left\|\Gamma_{\zeta}^{(j)}\left(\Id-\frac{\mathring{R}_{\ell}}{\varrho}\right)\right\|_{C^{0}_{t,x}} + \|\varrho^{\frac12}\|_{C^2_{t,x}}\|\eta_{j,k}\|_{C^0_{t,x}}\left\|\Gamma_{\zeta}^{(j)}\left(\Id-\frac{\mathring{R}_{\ell}}{\varrho}\right)\right\|_{C^{0}_{t,x}}
\nonumber\\  & \lesssim  (\delta_{q+1}^{\frac12}+ \|\mathring{R}_{q}\|_{C^0_{[t-1,t+1],x}}^{\frac12})  \lambda_q^{\frac{58}{15}\times 2 + \frac{34}{15}} + (\delta_{q+1}^{\frac12}+ \|\mathring{R}_{q}\|_{C^0_{[t-1,t+1],x}}^{\frac12}) m_q^{-1} \lambda_q^{\frac{58}{15}+\frac{34}{15}} 
        \nonumber\\ & \quad + (\delta_{q+1}^{\frac12}+ \|\mathring{R}_{q}\|_{C^0_{[t-1,t+1],x}}^{\frac12}) m_q^{-2}  + 
\ell^{-\frac{3}{2}} \|\mathring{R}_{q}\|_{C^0_{[t-1,t+1],x}}   \lambda_q^{\frac{58}{15} + \frac{34}{15}}
  + \ell^{-\frac{3}{2}} \|\mathring{R}_{q}\|_{C^0_{[t-1,t+1],x}} 
m_q^{-1}  + \ell^{\frac{1}{2}} \lambda_q^{\frac{58}{15}\times 2} 
    \nonumber\\  & \lesssim  (\delta_{q+1}^{\frac12}+ \|\mathring{R}_{q}\|_{C^0_{[t-1,t+1],x}}^{\frac12})(\lambda_q^{10} + m_q^{-1}\lambda_q^{\frac{92}{15}} + m_q^{-2} ) 
   \nonumber\\ &  \lesssim  m_q^{-1}\lambda_q^{\frac{92}{15}} (\delta_{q+1}^{\frac12}+ \|\mathring{R}_{q}\|_{C^0_{[t-1,t+1],x}}^{\frac12}).
    \end{align}
Again by chain rule, and using \eqref{varrho_0}, \eqref{varrho_m} and \eqref{Gamma_m}-\eqref{Gamma_0}, we have for $N\geq0$ 
	\begin{align*}
		\|a_{\zeta}\|_{C^{N}_{t,x}}   & \lesssim \sum_{m=0}^{N}\sum_{n=0}^{N-m}  \|\varrho^{\frac12}\|_{C^m_{t,x}}\|\eta_{j,k}\|_{C^n_{t,x}}\left\|\Gamma_{\zeta}^{(j)}\left(\Id-\frac{\mathring{R}_{\ell}}{\varrho}\right)\right\|_{C^{N-m-n}_{t,x}}
        \nonumber\\ & = \|\varrho^{\frac12}\|_{C^0_{t,x}} \sum_{n=0}^{N}  \|\eta_{j,k}\|_{C^n_{t,x}}\left\|\Gamma_{\zeta}^{(j)}\left(\Id-\frac{\mathring{R}_{\ell}}{\varrho}\right)\right\|_{C^{N-n}_{t,x}} 
        + \sum_{m=1}^{N}\sum_{n=0}^{N-m}  \|\varrho^{\frac12}\|_{C^m_{t,x}}\|\eta_{j,k}\|_{C^n_{t,x}}\left\|\Gamma_{\zeta}^{(j)}\left(\Id-\frac{\mathring{R}_{\ell}}{\varrho}\right)\right\|_{C^{N-m-n}_{t,x}}
        \nonumber\\ & \lesssim  \lambda_q^{\frac{1}{3}} \sum_{n=0}^{N}  m_q^{-n} \lambda_q^{\frac{58}{15} (N-n) + \frac{34}{15}} 
        + \sum_{m=1}^{N}\sum_{n=0}^{N-m}  \lambda_q^{\frac{58}{15}m-\frac{4}{5}} m_q^{-n} \lambda_q^{\frac{58}{15} (N-m-n) + \frac{34}{15}} 
        \nonumber\\ & \lesssim m_q^{-N} \lambda_q^{\frac{13}{5}} 
        + \sum_{m=1}^{N}\sum_{n=0}^{N-m}  \lambda_q^{\frac{58}{15}m-\frac{4}{5}} m_q^{-n} \lambda_q^{\frac{58}{15} (N-m-n) + \frac{34}{15}} 
        \nonumber\\ & \lesssim m_q^{-N} \lambda_q^{3} 
        + m_q^{-(N-1)} \lambda_q^{\frac{16}{3}}  \lesssim m_q^{-N} \lambda_q^{3}.
	\end{align*} 
This completes the proof.
\end{proof}

\subsection{Inductive estimates for $\v_{q+1}$}\label{v_q+1}
In this subsection, we will verify the inductive estimates \eqref{v_q_C0}-\eqref{v_q_EN} at the level $q+1$ as well as the inequality \eqref{v_diff_EN}. 
Let us recall the definition of $\omega_{q+1}^{(p)}$ as follows:
\begin{align*}
	\omega_{q+1}^{(p)}(t,x)&:=\sum_{j}\sum_{\zeta\in\Lambda_j} a_{q+1,j,\zeta}(t,x)B_{\zeta}e^{i\lambda_{q+1}\zeta\cdot\Phi_{k,j}(t,x)}.
\end{align*}
By the definition of $\omega_{q+1}^{(p)}$, we find (using \eqref{varrho_0} and \eqref{R_q_C0})
 \begin{align}\label{w_p_C0}
	\|\omega_{q+1}^{(p)}\|_{C^0_{t,x}}  & \lesssim 2|\Lambda_{j}|c_{\ast}^{-\frac{1}{2}}M \|\varrho^{\frac12}\|_{C^0_{t,x}}
	\nonumber\\ &  \lesssim 2|\Lambda_{j}|c_{\ast}^{-\frac{1}{2}}M \left[\delta_{q+1}^{\frac12} + \|\mathring{R}_{q}\|_{C^0_{[t-1,t+1],x}}^{\frac12} \right]
	 \lesssim 2|\Lambda_{j}|c_{\ast}^{-\frac{1}{2}}M \left[ \delta_{q+1}^{\frac12} + \lambda_{q}^{\frac13} \right] \lesssim \lambda_{q}^{\frac13} \leq \frac14\lambda_{q+1}^{\frac13},
\end{align}
where we select $a$ that is big enough to absorb the constant. Moreover, we also have (using \eqref{varrho_0} and  \eqref{R_q_EN}) 
\begin{align}\label{w_p_EN}
	\EN\omega_{q+1}^{(p)}\EN_{C^0,2r}  & \lesssim 2|\Lambda_{j}|c_{\ast}^{-\frac{1}{2}}M \EN\varrho^{\frac12}\EN_{C^0,2r}
	\nonumber\\ &  \lesssim 2|\Lambda_{j}|c_{\ast}^{-\frac{1}{2}}M \left[ \delta_{q+1}^{\frac12} + \EN\mathring{R}_{q}\EN_{C^0,r}^{\frac12}  \right]
     \lesssim 4 |\Lambda_{j}|c_{\ast}^{-\frac{1}{2}}M  \delta_{q+1}^{\frac12}  \leq \frac{1}{4}\bar{M}\delta_{q+1}^{\frac{1}{2}},
\end{align}
where $\bar{M}$ is a universal constant satisfying $17|\Lambda_{j}|c_{\ast}^{-\frac{1}{2}}M<\bar{M}$ and we select $a$ that is big enough to absorb the constant.  Using \eqref{R_q_C0}, \eqref{ak_0}, \eqref{ak_1}, \eqref{B.9}, \eqref{B.11}, and   $0<\beta<  \frac{1}{200} $, we estimate

\begin{align}\label{w_p_C1}
	& \|\omega_{q+1}^{(p)}\|_{C^1_{t,x}} 
      \lesssim \sup_{j}\sum_{\zeta\in\Lambda_j}\bigg[\|a_{\zeta}\|_{C^1_{t,x}} + \|a_{\zeta}\|_{C^0_{t,x}}\cdot\lambda_{q+1} \left(\|\partial_t\Phi_{k,j}\|_{C^{0}_{t,x}} + \|\nabla\Phi_{k,j}\|_{C^{0}_{t,x}}\right)\bigg]
	\nonumber\\ & \lesssim m_q^{-1}  (\delta_{q+1}^{\frac12} + \|\mathring{R}_{q}\|^{\frac12}_{C^{0}_{[t-1,t+1],x}} ) +  (\delta_{q+1}^{\frac12} +  \|\mathring{R}_{q}\|_{C^{0}_{[t-1,t+1],x}}^{\frac12}) \cdot\lambda_{q}^{\frac{23}{4}} \left(\lambda_{q}^{\frac13} + 1\right)
      \lesssim  \lambda_{q}^{\frac{77}{12}} \leq  \frac14 \lambda_{q+1}^{\frac75}\delta_{q+1}^{\frac12},
\end{align}
where we select $a$ that is big enough to absorb the constant. Let us recall the definition of  $\omega_{q+1}^{(c)}$ as follows:
\begin{align*}
	\omega_{q+1}^{(c)}(t,x)&:= \sum_{j}\sum_{\zeta\in\Lambda_j} \left(\frac{\nabla a_\zeta}{\lambda_{q+1}}+i a_\zeta (\nabla\Phi_{k,j}(t,x)-\Id)\zeta \right)B_{\zeta}e^{i\lambda_{q+1}\zeta\cdot \Phi_{k,j}(t,x)}.
\end{align*}
Using \eqref{ak_0}-\eqref{ak_1}, \eqref{R_q_C0} and \eqref{B.8}, we obtain   
 \begin{align}
	\|\omega_{q+1}^{(c)}\|_{C^0_{t,x}}& \lesssim \sup_{j}\sum_{\zeta\in\Lambda_j} \left(\frac{\|\nabla a_{\zeta}\|_{C^0_{t,x}}}{\lambda_{q+1}}+\| a_{\zeta}\|_{C^0_{t,x}}\|\nabla\Phi_{k,j}(t,x)-\Id\|_{C^0_{t,x}}\right)
	\nonumber\\& \lesssim  m_q^{-1} \lambda_q^{-\frac{23}{4}} (\delta_{q+1}^{\frac12} + \|\mathring{R}_{q}\|^{\frac12}_{C^{0}_{[t-1,t+1],x}} )  + (\delta_{q+1}^{\frac12} + \|\mathring{R}_{q}\|^{\frac12}_{C^0_{[t-1,t+1],x}}) \lambda_{q}^{-\frac{8}{5}}
    \nonumber\\& \lesssim  \lambda_{q}^{-\frac{11}{16}-\frac{27}{8}\beta}  (\delta_{q+1}^{\frac12} + \|\mathring{R}_{q}\|^{\frac12}_{C^{0}_{[t-1,t+1],x}} ) \label{w_c_C0_0}
    \\ & \lesssim  \lambda_{q}^{-\frac{17}{48} - \frac{27}{8}\beta}  \leq \frac14 \lambda_{q+1}^{\frac13}.  \label{w_c_C0}
\end{align}
 Taking expectation of \eqref{w_c_C0_0} and using \eqref{R_q_EN}, we obtain    
\begin{align}
	\EN\omega_{q+1}^{(c)}\EN_{C^0,2r}  & \lesssim \lambda_{q}^{-\frac{11}{16}-\frac{27}{8}\beta} ( \delta_{q+1}^{\frac12} + \EN\mathring{R}_{q}\EN_{C^0,r}^{\frac{1}{2}}) \lesssim \lambda_{q}^{-\frac{11}{16}-\frac{27}{8}\beta} \delta_{q+1}^{\frac12} 
	   \leq \frac{1}{4}\bar{M}\delta_{q+1}^{\frac12},\label{w_c_EN}
\end{align}
where we select $a$ that is big enough to absorb the constant in \eqref{w_c_EN}. 

 Using \eqref{R_q_C0}, \eqref{w_c_C0}, \eqref{ak_N},  \eqref{B.8}-\eqref{B.12}, and   $0 < \beta < \frac{1}{200} $, we obtain
\begin{align}\label{w_c_C1}
		\|\omega_{q+1}^{(c)}\|_{C^1_{t,x}} 
        & \lesssim  \lambda_{q+1} \|\omega_{q+1}^{(c)}\|_{C^0_{t,x}} \sup_{j} \left(\|\partial_t\Phi_{k,j}\|_{C^{0}_{t,x}} + \|\nabla\Phi_{k,j}\|_{C^{0}_{t,x}}\right)
	\nonumber\\& \quad + \sup_{j} \sum_{\zeta\in\Lambda_j}\bigg[\frac{\|\nabla a_{\zeta}\|_{C^1_{t,x}}}{\lambda_{q+1}} + \|a_{\zeta}\|_{C^1_{t,x}}\|\nabla\Phi_{k,j}-\Id\|_{C^0_{t,x}}
	+ \|a_{\zeta}\|_{C^0_{t,x}}(\|\nabla^2\Phi_{k,j}\|_{C^0_{t,x}}+\|\partial_t\nabla\Phi_{k,j}\|_{C^0_{t,x}}) \bigg]
   \nonumber\\ & \lesssim  \lambda_{q}^{\frac{23}{4}} \lambda_{q}^{-\frac{17}{48} - \frac{27}{8}\beta} \left( \lambda_q^{\frac13} + 1\right)
 + \bigg[\lambda_q^{-\frac{23}{4}}   m_q^{-1} \lambda_q^{\frac{92}{15}}   \lambda_q^{\frac{1}{3}}   + m_q^{-1} \lambda_q^{\frac{1}{3}-\frac{8}{5}}
	+ \lambda_q^{\frac13} (1 + \lambda_q^{\frac75}) \bigg]
    \nonumber\\ & \lesssim   \lambda_{q}^{\frac{275}{48} - \frac{27}{8}\beta}   + \lambda_{q}^{ \frac{1387}{240} - \frac{27}{8}\beta } + \lambda_{q}^{ \frac{911}{240} - \frac{27}{8}\beta }
	+ \lambda_q^{\frac{26}{15}} 
      \lesssim  \lambda_{q}^{ \frac{1387}{240} - \frac{27}{8}\beta }   \leq \frac{1}{4} \lambda_{q+1}^{\frac75}\delta_{q+1}^{\frac12},
\end{align}
where we select $a$ that is big enough to absorb the constant. Now, in view of mollification estimates and \eqref{v_q_C1}, we find for any $t\in\R$ and $0<\beta< \frac{1}{200}$
\begin{align}\label{5.16}
	\|(\v_{q}-\v_{\ell})(t)\|_{\L^{\infty}} \lesssim \ell \|\v_{q}\|_{C^1_{[t-1,t],x}}\leq  \lambda_{q}^{-\frac85} \lambda_{q}^{\frac{7}{5}}\delta_{q}^{\frac12} \leq \frac{1}{4} \bar{M}\delta_{q+1}^{\frac12},
\end{align}
where we select $a$ that is big enough to absorb the constant. Taking expectation of \eqref{5.16}, we obtain
\begin{align}\label{vq-vl_EN}
	\EN\v_{\ell}-\v_{q}\EN_{C^0,2r} \leq \frac{1}{4} \bar{M}\delta_{q+1}^{\frac12}.
\end{align}
Making use of \eqref{w_p_EN}, \eqref{w_c_EN} and \eqref{vq-vl_EN}, we reach at
\begin{align*}
	\EN\v_{q+1}-\v_{q}\EN_{C^0,2r}\leq \EN \v_{\ell}-\v_{q}\EN_{C^0,2r} + \EN\omega_{q+1}\EN_{C^0,2r}\leq  \bar{M}\delta_{q+1}^{\frac12},
\end{align*}
which completes the proof of \eqref{v_diff_EN}. Combining \eqref{v_q_C0}, \eqref{w_p_C0} and \eqref{w_c_C0}, we arrive at
\begin{align}\label{v_q+1_C0}
	\|\v_{q+1}\|_{C^0_{t,x}} & \leq \|\v_{\ell}\|_{C^0_{t,x}} + \|\omega_{q+1}^{(p)}\|_{C^0_{t,x}}  + \|\omega_{q+1}^{(c)}\|_{C^0_{t,x}} 
	   \leq  \lambda_{q}^{\frac13}+ \frac12 \lambda_{q+1}^{\frac13}
	   \leq  \lambda_{q+1}^{\frac13},
\end{align}
which completes the proof of \eqref{v_q_C0}.  In view of \eqref{v_q_C1}, \eqref{w_p_C1} and \eqref{w_c_C1}, we obtain 
\begin{align}\label{v_q+1_C1}
		\|\v_{q+1}\|_{C^1_{t,x}} & \leq \|\v_{\ell}\|_{C^1_{t,x}} + \|\omega_{q+1}^{(p)}\|_{C^1_{t,x}}  + \|\omega_{q+1}^{(c)}\|_{C^1_{t,x}} 
	  \leq    \lambda_{q}^{\frac75}\delta_{q}^{\frac12}+\frac{1}{2} \lambda_{q+1}^{\frac75}\delta_{q+1}^{\frac12} \leq \lambda_{q+1}^{\frac75}\delta_{q+1}^{\frac12},
\end{align}
which completes the proof of \eqref{v_q_C1}. Combining \eqref{v_q_EN}, \eqref{w_p_C0} and \eqref{w_c_C0}, we find
 \begin{align}\label{v_q+1_EN}
	\EN \v_{q+1}\EN_{C^0,2r} \leq \EN \v_{q}\EN_{C^0,2r} + \EN \v_{q+1} - \v_{q}\EN_{C^0,2r}\leq  6rL^2 - \delta_{q}^{\frac{1}{2}} + \bar{M}\delta_{q+1}^{\frac{1}{2}} \leq 6rL^2 -\delta_{q+1}^{\frac{1}{2}},
\end{align}
  for big enough $a$ and $L\geq \frac{1+\bar{M}}{\sqrt{6r}}$, which completes the proof of \eqref{v_q_EN}.

\subsection{Inductive estimates for $\mathring{R}_{q+1}$}\label{R_q+1}
 In this subsection, we will verify the inductive estimates \eqref{R_q_C0} and \eqref{R_q_EN} at the level $q+1$. For this subsection, we choose two parameters $\varpi>0$ and $m\in\N$ such that  $0<\varpi<\min\left\{1-2\alpha-\frac{21}{2}\beta, \frac{73}{1380} - \frac{228}{23}\beta  \right\} $ and $ m > 33 $. Next, we will estimate each term on the right hand side of \eqref{R_q+1-split} separately.
\subsubsection{Estimate on $R_{lin}$.}
Let us recall the definition of $R_{lin}$ as follows:
\begin{align*}
	R_{lin}&:=  \mathcal{R}\bigg[\nu(-\Delta)^{\alpha} \omega_{q+1}\bigg].
\end{align*}
We know by Lemma \ref{Lem_C.2} that 
\begin{align}\label{R_lin}
	\left\|\mathcal{R}\bigg[\nu(-\Delta)^{\alpha} \omega_{q+1}\bigg]\right\|_{C^0_{t}C^0_{x}} \lesssim  \left\|\mathcal{R}\omega_{q+1}^{(p)}\right\|_{C^0_{t}C^{2\alpha+\varpi}_{x}} + \left\|\mathcal{R}\omega_{q+1}^{(c)}\right\|_{C^0_{t}C^{2\alpha+\varpi}_{x}}.
\end{align}
\vskip 2mm
\noindent
We have by Lemma \ref{SPL} that 
\begin{align}
\left\|\mathcal{R}\omega_{q+1}^{(p)}\right\|_{C^0_{t}C^{2\alpha+\varpi}_{x}} 
	 & \lesssim \sup_{j} \sum_{\zeta\in\Lambda_j}\left[	\frac{\|a_{\zeta}\|_{C^0_{t}C^0_{x}}}{\lambda_{q+1}^{1-2\alpha-\varpi}} + \frac{\|a_{\zeta}\|_{C^0_{t}C^{m,2\alpha+\varpi}_{x}}+\|a_{\zeta}\|_{C^0_{t}C^0_{x}}\|\nabla\Phi\|_{C^0_{t}C^{m,2\alpha+\varpi}_{x}}}{\lambda_{q+1}^{m-2\alpha-\varpi}}\right]
	\nonumber\\ & \lesssim \sup_{j} \sum_{\zeta\in\Lambda_j}\left[	\frac{\|a_{\zeta}\|_{C^0_{t}C^0_{x}}}{\lambda_{q+1}^{1-2\alpha-\varpi}} + \frac{\|a_{\zeta}\|_{C^0_{t}C^{m+1}_{x}}+\|a_{\zeta}\|_{C^0_{t}C^0_{x}}\|\nabla\Phi\|_{C^0_{t}C^{m+1}_{x}}}{\lambda_{q+1}^{m-2\alpha-\varpi}}\right]
    \nonumber\\ & \lesssim 	\frac{(\delta_{q+1}^{\frac12} + \|\mathring{R}_{q}\|^{\frac12}_{C^0_{[t-1,t+1],x}})}{\lambda_{q+1}^{1-2\alpha-\varpi}} + \frac{ m_q^{-(m+1)}\lambda_q^{3} + ( \delta_{q+1}^{\frac12} + \|\mathring{R}_{q}\|^{\frac12}_{C^0_{[t-1,t+1],x}} ) \lambda_q^{\frac{8}{5}m}}{\lambda_{q+1}^{m-2\alpha-\varpi}} 
    \nonumber\\ & \lesssim 	\frac{(\delta_{q+1}^{\frac12} + \|\mathring{R}_{q}\|^{\frac12}_{C^0_{[t-1,t+1],x}})}{\lambda_{q+1}^{1-2\alpha-\varpi}} + \frac{1}{7\cdot 4}\delta_{q+2},\label{Rw_p_lin0}
\end{align}
where we have also used \eqref{ak_N}, \eqref{B.10}, $m>33$ and $\beta<\frac{1}{200}$.

Now, we estimate (using \eqref{ell}, \eqref{ak_N}, \eqref{B.8} and \eqref{B.10})

\begin{align}
    \|a_{\zeta}(\nabla\Phi_{k,j}-\Id)\|_{{C^0_{t}C^0_{x}}} & \leq \|a_{\zeta}\|_{{C^0_{t}C^0_{x}}}\|\nabla\Phi_{k,j}-\Id\|_{{C^0_{t}C^0_{x}}}
     \lesssim (\delta_{q+1}^{\frac12} + \|\mathring{R}_{q}\|^{\frac12}_{C^0_{[t-1,t+1],x}}) \lambda_q^{-\frac85},\label{aPhi_k0}
\end{align}
and, for all $N\geq0$
\begin{align}
	&\|a_{\zeta}(\nabla\Phi_{k,j}-\Id)\|_{{C^0_{t}C^N_{x}}}  
	 \lesssim \sum_{k=0}^{N}\|a_{\zeta}\|_{C^k_{t,x}}\|\nabla\Phi_{k,j}-\Id\|_{C^0_{t}C^{N-k}_{x}}
       \lesssim \sum_{k=0}^{N} m_q^{-k} \lambda_q^{3} \lambda_q^{\frac{8}{5}(N-k-1)} \lesssim  m_q^{-N} \lambda_q^{ 2 }.
	\label{aPhi_kN}
\end{align}

Again, we have by Lemma \ref{SPL} and  \eqref{aPhi_kN} that   
\begin{align}
&\left\|\mathcal{R}\omega_{q+1}^{(c)}\right\|_{C^0_{t}C^{2\alpha+\varpi}_{x}} 
\nonumber\\ 	& \lesssim \sup_{j} \sum_{\zeta\in\Lambda_j}\frac{1}{\lambda_{q+1}}\left[	\frac{\|\nabla a_{\zeta}\|_{C^0_{t}C^0_{x}}}{\lambda_{q+1}^{1-2\alpha-\varpi}} + \frac{\| \nabla a_{\zeta}\|_{C^0_{t}C^{m,2\alpha+\varpi}_{x}}+\|\nabla a_{\zeta}\|_{C^0_{t}C^0_{x}}\|\nabla\Phi_{k,j}\|_{C^0_{t}C^{m,2\alpha+\varpi}_{x}}}{\lambda_{q+1}^{m-2\alpha-\varpi}}\right]
	\nonumber\\& \quad + \sup_{j} \sum_{\zeta\in\Lambda_j}\bigg[	\frac{\|a_{\zeta}(\nabla\Phi_{k,j}-\Id)\|_{C^0_{t}C^0_{x}}}{\lambda_{q+1}^{1-2\alpha-\varpi}}   + \frac{\|a_{\zeta}(\nabla\Phi_{k,j}-\Id)\|_{C^0_{t}C^{m,2\alpha+\varpi}_{x}}+\|a_{\zeta}(\nabla\Phi_{k,j}-\Id)\|_{C^0_{t}C^0_{x}}\|\nabla\Phi_{k,j}\|_{C^0_{t}C^{m,2\alpha+\varpi}_{x}}}{\lambda_{q+1}^{m-2\alpha-\varpi}}\bigg]
	\nonumber\\& \lesssim \sup_{j} \sum_{\zeta\in\Lambda_j}\frac{1}{\lambda_{q+1}}\left[	\frac{\|\nabla a_{\zeta}\|_{C^0_{t}C^0_{x}}}{\lambda_{q+1}^{1-2\alpha-\varpi}} + \frac{\| \nabla a_{\zeta}\|_{C^0_{t}C^{m+1}_{x}}+\|\nabla a_{\zeta}\|_{C^0_{t}C^0_{x}}\|\nabla\Phi_{k,j}\|_{C^0_{t}C^{m+1}_{x}}}{\lambda_{q+1}^{m-1}}\right]
	\nonumber\\& \quad + \sup_{j} \sum_{\zeta\in\Lambda_j}\bigg[	\frac{\|a_{\zeta}(\nabla\Phi_{k,j}-\Id)\|_{C^0_{t}C^0_{x}}}{\lambda_{q+1}^{1-2\alpha-\varpi}}   + \frac{\|a_{\zeta}(\nabla\Phi_{k,j}-\Id)\|_{C^0_{t}C^{m+1}_{x}}+\|a_{\zeta}(\nabla\Phi_{k,j}-\Id)\|_{C^0_{t}C^0_{x}}\|\nabla\Phi_{k,j}\|_{C^0_{t}C^{m+1}_{x}}}{\lambda_{q+1}^{m-1}}\bigg]
    \nonumber\\& \lesssim \lambda_{q}^{-b} \left[	\frac{m_q^{-1}(\delta_{q+1}^{\frac12} + \|\mathring{R}_{q}\|^{\frac12}_{C^0_{[t-1,t+1],x}})}{\lambda_{q+1}^{1-2\alpha-\varpi}} + \frac{ m_q^{-(m+2)}\lambda_q^{3} + m_q^{-1} \lambda_q^{\frac{8}{5}m }  }{\lambda_{q+1}^{m-1}}\right]
	\nonumber\\& \quad + 	\frac{(\delta_{q+1}^{\frac12} + \|\mathring{R}_{q}\|^{\frac12}_{C^0_{[t-1,t+1],x}}) \lambda_q^{-\frac85}}{\lambda_{q+1}^{1-2\alpha-\varpi}}   + \frac{m_q^{-(m+1)}\lambda_q + \lambda_q \cdot \lambda_q^{\frac85 m} }{\lambda_{q+1}^{m-1}}
    \nonumber\\& \lesssim  	\frac{(\delta_{q+1}^{\frac12} + \|\mathring{R}_{q}\|^{\frac12}_{C^0_{[t-1,t+1],x}})}{\lambda_{q+1}^{1-2\alpha-\varpi}}  + \frac{1}{7\cdot 4}\delta_{q+2},  \label{Rw_c_lin0}
\end{align}
where we have also used \eqref{ak_N}, \eqref{B.10}, $m>33$ and $\beta<\frac{1}{200}$. Combining \eqref{R_lin}, \eqref{Rw_p_lin0} and \eqref{Rw_c_lin0}, we get
\begin{align}
	\|R_{lin}\|_{C^0_{t,x}} & \lesssim  \frac{(\delta_{q+1}^{\frac12} + \|\mathring{R}_{q}\|^{\frac12}_{C^0_{[t-1,t+1],x}})}{\lambda_{q+1}^{1-2\alpha-\varpi}}  + \frac{1}{7\cdot 2}\delta_{q+2}   \label{R_lin1}
	\\ & \leq  \frac{1}{7}\lambda_{q+1}^{{\frac23}},\label{R_lin2}
\end{align}
where we have also used $0<\varpi< 1-2\alpha-\frac{21}{2} \beta$ and $m> 33$, and we select $a$ that is big enough to absorb the constant. Taking expectation of \eqref{R_lin1}, we obtain for $0 < \beta < \min\left\{ \frac{2(1-2\alpha)}{21}, \frac{1}{200}\right\}$
\begin{align}\label{R_lin_r}
	\EN R_{lin}\EN_{C^0,r}
	&  \lesssim 	\frac{\delta_{q+1}^{\frac12}}{\lambda_{q+1}^{1-2\alpha-\varpi}}  + \frac{1}{7\cdot 2}\delta_{q+2}
  \leq \frac{1}{7 } \delta_{q+2},
\end{align}
where we have used \eqref{R_q_EN}, $0<\varpi< 1-2\alpha-\frac{21}{2} \beta$ and $m> 33$, and we select $a$ that is big enough to absorb the constant. 

\subsubsection{Estimate on $R_{trans}$.}
Let us recall the definition of $R_{trans}$ as follows:
\begin{align*}
	R_{trans}&:= \mathcal{R}\bigg[(\partial_t+(\v_{\ell}+\z_{\ell})\cdot\nabla)\omega^{(p)}_{q+1}\bigg].
\end{align*}
We know by the definition of $\omega_{q+1}^{(p)}$ that for $t\in[k,k+1]$, $k\in\mathbb{Z}$
\begin{align}\label{3.29}
	(\partial_t+(\v_{\ell}+\z_{\ell})\cdot\nabla)\omega^{(p)}_{q+1} = \sum_{j}\sum_{\zeta\in\Lambda_j}(\partial_t+(\v_{\ell}+\z_{\ell})\cdot\nabla)a_{\zeta}B_{\zeta} e^{i\lambda_{q+1}\zeta\cdot\Phi_{k,j}}.
\end{align}
Again by chain rule, \eqref{ak_N}, \eqref{v_q_C0} and \eqref{z_q_C0}, we have 
\begin{align}\label{31244}
    \|(\partial_t+(\v_{\ell}+\z_{\ell})\cdot\nabla)a_{\zeta}\|_{C^0_tC^{0}_{x}} & \lesssim \|a_{\zeta}\|_{C^{1}_{t,x}} + \|\v_{\ell}+\z_{\ell}\|_{C^0_{t,x}} \|a_{\zeta}\|_{C^{1}_{t,x}}
    \nonumber\\ & \lesssim m_q^{-1} (\delta_{q+1}^{\frac12} + \|\mathring{R}_{q}\|^{\frac12}_{C^0_{[t-1,t+1],x}}) (1+\|\v_{\ell}\|_{C^0_{t,x}} + \|\z_{\ell}\|_{C^0_{t,x}}),
\end{align}
and, for $N\geq0$ 
\begin{align}\label{3124}
\|(\partial_t+(\v_{\ell}+\z_{\ell})\cdot\nabla)a_{\zeta}\|_{C^0_tC^{N}_{x}} 
& \lesssim \|a_{\zeta}\|_{C^{N+1}_{t,x}} + \sum_{k=0}^{N} \|\v_{\ell}+\z_{\ell}\|_{C^0_tC^{k}_{x}} \|\nabla a_{\zeta}\|_{C^{N-k}_{t,x}}  
\nonumber  \\
& \lesssim m_q^{-(N+1)} \lambda_q^{3}  + \sum_{k=0}^{N} \ell^{-k} \|\v_{q}+\z_{q}\|_{C^0_{[t-1,t+1],x}} m_q^{-(N-k+1)} \lambda_q^{3}  
\nonumber  \\
& \lesssim m_q^{-(N+1)} \lambda_q^{3}   (1+ \|\v_{q}\|_{C^{0}_{[t-1,t+1],x}}+\|\z_{q}\|_{C^{0}_{[t-1,t+1],x}})
\lesssim m_q^{-(N+1)} \lambda_q^{\frac{10}{3}} .
\end{align}
Making use of Lemma \ref{SPL}, \eqref{3.29}-\eqref{3124}, \eqref{v_q_C0}, \eqref{R_q_C0} and \eqref{z_q_C0}, we get  
 \begin{align}
 &\|R_{trans}\|_{C^0_{t,x}}  \lesssim 	\left\|R_{trans}\right\|_{C^0_tC^{\varpi}_x} 
\nonumber \\ & \lesssim \frac{\|(\partial_t+(\v_{\ell}+\z_{\ell})\cdot\nabla)a_{\zeta}\|_{C^0_{t,x}}}{\lambda_{q+1}^{1-\varpi}} 
  + \frac{\|(\partial_t+(\v_{\ell}+\z_{\ell})\cdot\nabla)a_{\zeta}\|_{C^0_tC^{m,\varpi}_{x}}+\|(\partial_t+(\v_{\ell}+\z_{\ell})\cdot\nabla)a_{\zeta}\|_{C^0_{t,x}}\|\nabla\Phi_{k,j}\|_{C^0_tC^{m,\varpi}_x}}{\lambda_{q+1}^{m-\varpi}}
\nonumber \\ & \lesssim \frac{\|(\partial_t+(\v_{\ell}+\z_{\ell})\cdot\nabla)a_{\zeta}\|_{C^0_{t,x}}}{\lambda_{q+1}^{1-\varpi}} 
 + \frac{\|(\partial_t+(\v_{\ell}+\z_{\ell})\cdot\nabla)a_{\zeta}\|_{C^0_tC^{m+1}_{x}}+\|(\partial_t+(\v_{\ell}+\z_{\ell})\cdot\nabla)a_{\zeta}\|_{C^0_{t,x}}\|\nabla\Phi_{k,j}\|_{C^0_tC^{m+1}_x}}{\lambda_{q+1}^{m-1}}
\nonumber \\ & \lesssim \frac{m_q^{-1} (\delta_{q+1}^{\frac12} + \|\mathring{R}_{q}\|^{\frac12}_{C^0_{[t-1,t+1],x}}) (1+\|\v_{q}\|_{C^0_{[t-1,t+1],x}} + \|\z_{q}\|_{C^0_{[t-1,t+1],x}})}{\lambda_{q+1}^{1-\varpi}} 
 + \frac{ m_q^{-(m+2)} \lambda_q^{\frac{10}{3}} }{\lambda_{q+1}^{m-1}}
\nonumber \\ & \lesssim \frac{m_q^{-1} (\delta_{q+1}^{\frac12} + \|\mathring{R}_{q}\|^{\frac12}_{C^0_{[t-1,t+1],x}}) (1+\|\v_{q}\|_{C^0_{[t-1,t+1],x}} + \|\z_{q}\|_{C^0_{[t-1,t+1],x}})}{\lambda_{q+1}^{1-\varpi}} 
 + \frac{1}{7\cdot 2} \delta_{q+2} \label{R_tra_C00}
  \\ & \lesssim \frac{m_q^{-1}\lambda_q^{\frac{2}{3}}}{\lambda_{q+1}^{1-\varpi}} 
 + \frac{1}{7\cdot 2} \delta_{q+2} 
 \leq \frac17 \lambda_{q+1}^{\frac23}, \label{R_tra_C0}
\end{align}
where we have also used $\varpi< \frac{73}{1380} - \frac{228}{23}\beta < \frac{15}{276} $, $m> 33$ and $\beta<\frac{1}{200}$, and we select $a$ that is big enough to absorb the constant. Taking expectation of \eqref{R_tra_C00} and using H\"older's inequality, we obtain  for $0<\beta< \frac{1}{200}$
\begin{align}\label{R_tra_EN}
	\EN R_{trans} \EN_{C^0,r}
  &  \lesssim  \frac{m_q^{-1} \delta_{q+1}^{\frac12} }{\lambda_{q+1}^{1-\varpi}}  +  \frac{1}{7\cdot 2} \delta_{q+2} 
   \leq \frac17\delta_{q+2},
\end{align}
where we have used \eqref{v_q_EN}, \eqref{z_q_EN1}, \eqref{R_q_EN}, $\varpi< \frac{73}{1380} - \frac{228}{23}\beta $ and $m> 33$, and we select $a$ that is big enough to absorb the constant.

\subsubsection{Estimate on $R_{Nash}$.} 
Let us recall the definition of $R_{Nash}$ as follows:
\begin{align*}
		R_{Nash} &:= \mathcal{R}\left[(\omega_{q+1}\cdot\nabla)(\v_{\ell}+\z_{\ell})\right].
\end{align*}
We write 
\begin{align}\label{R_nash}
	(\omega_{q+1}\cdot\nabla)(\v_{\ell}+\z_{\ell})= (\omega_{q+1}^{(p)}\cdot\nabla)(\v_{\ell}+\z_{\ell})+(\omega_{q+1}^{(c)}\cdot\nabla)(\v_{\ell}+\z_{\ell}).
\end{align}
For the term corresponding to the principle part $\omega_{q+1}^{(p)}$, we have for $t\in[k,k+1]$, $k\in\mathbb{Z}$
\begin{align}\label{R_nash1}
	(\omega_{q+1}^{(p)}\cdot\nabla)(\v_{\ell}+\z_{\ell})=\sum_{j}\sum_{\zeta\in\Lambda_j}a_{\zeta}\left(B_{\zeta}\cdot\nabla\right)(\v_{\ell}+\z_{\ell})e^{i\lambda_{q+1}\zeta\cdot\Phi_{k,j}}.
\end{align}
Let us now estimate $\|a_{\zeta}\left(B_{\zeta}\cdot\nabla\right)(\v_{\ell}+\z_{\ell})\|_{C^N_{t,x}}$. Using \eqref{ell}, \eqref{v_q_C0}, \eqref{z_q_C0}, \eqref{ak_0}, \eqref{ak_1} and \eqref{ak_N}, we have 
 \begin{align}\label{5488}
    \|a_{\zeta}\left(B_{\zeta}\cdot\nabla\right)(\v_{\ell}+\z_{\ell})\|_{C^0_{t,x}} 
    & \leq \|a_{\zeta}\|_{C^{0}_{t,x}}\|\nabla(\v_{\ell}+\z_{\ell})\|_{C^{0}_{t,x}} 
   \nonumber\\ &    \lesssim \lambda_{q}^{\frac85} (\delta_{q+1}^{\frac12} + \|\mathring{R}_{q}\|^{\frac12}_{C^0_{[t-1,t+1],x}})  (\|\v_{q}\|_{C^{0}_{[t-1,t+1],x}}+ \|\z_{q}\|_{C^{0}_{[t-1,t+1],x}}),
\end{align}
and for $N\geq 0$
\begin{align}\label{548}
	&\|a_{\zeta}\left(B_{\zeta}\cdot\nabla\right)(\v_{\ell}+\z_{\ell})\|_{C^N_{t,x}}
	 \leq \|a_{\zeta}\cdot\nabla(\v_{\ell}+\z_{\ell})\|_{C^N_{t,x}}
	\nonumber\\& \lesssim  \sum_{k=0}^{N}\|a_{\zeta}\|_{C^{k}_{t,x}}\|\nabla(\v_{\ell}+\z_{\ell})\|_{C^{N-k}_{t,x}}
	 \lesssim  \sum_{k=0}^{N} m_q^{-k} \lambda_q^{3}  \ell^{-N+k-1}(\|\v_{q}\|_{C^{0}_{[t-1,t+1],x}}+\|\z_{q}\|_{C^{0}_{[t-1,t+1],x}}) \lesssim    m_q^{-N} \lambda_q^{\frac{74}{15}} .
\end{align}
Therefore, using Lemma \ref{SPL} and \eqref{R_nash1}-\eqref{548}, we obtain 
\begin{align}
	& \left\|\mathcal{R}\left((\omega_{q+1}^{(p)}\cdot\nabla)(\v_{\ell}+\z_{\ell})\right)\right\|_{C^{0}_{t,x}} 
 \leq \left\|\mathcal{R}\left((\omega_{q+1}^{(p)}\cdot\nabla)(\v_{\ell}+\z_{\ell})\right)\right\|_{C^0_tC^{\varpi}_x} 
	\nonumber\\ & \lesssim \sum_{j}\sum_{\zeta\in\Lambda_j} \frac{\|a_{\zeta}\left(B_{\zeta}\cdot\nabla\right)(\v_{\ell}+\z_{\ell})\|_{C^0_tC^0_x}}{\lambda_{q+1}^{1-\varpi}}  \nonumber\\ &\quad+ \frac{\|a_{\zeta}\left(B_{\zeta}\cdot\nabla\right)(\v_{\ell}+\z_{\ell})\|_{C^0_tC^{m,\varpi}_x}+\|a_{\zeta}\left(B_{\zeta}\cdot\nabla\right)(\v_{\ell}+\z_{\ell})\|_{C^0_tC^0_{x}}\|\nabla\Phi_{k,j}\|_{C^0_tC^{m,\varpi}_x}}{\lambda_{q+1}^{m-\varpi}}
	\nonumber\\ & \lesssim \sum_{j}\sum_{\zeta\in\Lambda_j} \frac{\|a_{\zeta}\cdot(\v_{\ell}+\z_{\ell})\|_{C^0_tC^0_x}}{\lambda_{q+1}^{1-\varpi}}   + \frac{\|a_{\zeta} \cdot\nabla(\v_{\ell}+\z_{\ell})\|_{C^0_tC^{m+1}_x}+\|a_{\zeta} \cdot\nabla(\v_{\ell}+\z_{\ell})\|_{C^0_tC^0_x}\|\nabla\Phi_{k,j}\|_{C^0_tC^{m+1}_x}}{\lambda_{q+1}^{m-1}} 
    \nonumber\\ & \lesssim \frac{\lambda_{q}^{\frac85} (\delta_{q+1}^{\frac12} + \|\mathring{R}_{q}\|^{\frac12}_{C^0_{[t-1,t+1],x}})  (\|\v_{q}\|_{C^{0}_{[t-1,t+1],x}}+ \|\z_{q}\|_{C^{0}_{[t-1,t+1],x}})}{\lambda_{q+1}^{1-\varpi}}   + \frac{ m_q^{-(m+1)}\lambda_q^{\frac{74}{15}} }{\lambda_{q+1}^{m-1}}
    \nonumber\\ & \lesssim \frac{\lambda_{q}^{\frac85} (\delta_{q+1}^{\frac12} + \|\mathring{R}_{q}\|^{\frac12}_{C^0_{[t-1,t+1],x}})  (\|\v_{q}\|_{C^{0}_{[t-1,t+1],x}}+ \|\z_{q}\|_{C^{0}_{[t-1,t+1],x}})}{\lambda_{q+1}^{1-\varpi}}   + \frac{1}{7\cdot 4}\delta_{q+2} , \label{R_Nash_p0}
\end{align}
 where we have also used  $m> 33$ and $\beta<\frac{1}{200}$.
 
For the term corresponding to the corrector part $\omega_{q+1}^{(c)}$, we have for $t\in[k,k+1]$, $k\in\mathbb{Z}$
\begin{align*}
	&(\omega_{q+1}^{(c)}\cdot\nabla)(\v_{\ell}+\z_{\ell})
	  =\sum_{j}\sum_{\zeta\in\Lambda_j}\left(\left(\left(\frac{\nabla a_{\zeta}}{\lambda_{q+1}}+ia_\zeta(\nabla\Phi_{k,j}-\Id)\zeta\right)\times B_{\zeta}\right)\cdot\nabla\right)(\v_{\ell}+\z_{\ell})e^{i\lambda_{q+1}\zeta\cdot\Phi_{k,j}}.
\end{align*}
 Let us estimate $\left\|\left(\frac{\nabla a_{\zeta}}{\lambda_{q+1}}\cdot\nabla\right)(\v_{\ell}+\z_{\ell})\right\|_{C^N_{t,x}}$.  We have
\begin{align}\label{5533}
    \left\|\left(\frac{\nabla a_{\zeta}}{\lambda_{q+1}}\cdot\nabla\right)(\v_{\ell}+\z_{\ell})\right\|_{C^0_{t,x}} & \leq \lambda_q^{-\frac{23}{4}} \|\nabla a_{\zeta}\|_{C^0_{t,x}} \|\nabla(\v_{\ell}+\z_{\ell})\|_{C^{0}_{t,x}}
    \nonumber\\ & \lesssim m_q^{-1}\lambda_{q}^{-\frac{23}{4}} (\delta_{q+1}^{\frac12} + \|\mathring{R}_{q}\|^{\frac12}_{C^0_{[t-1,t+1],x}})  (\|\v_{q}\|_{C^{0}_{[t-1,t+1],x}}+ \|\z_{q}\|_{C^{0}_{[t-1,t+1],x}}),
\end{align}
and for $N\geq0$
 \begin{align}\label{553}
 	 \left\|\left(\frac{\nabla a_{\zeta}}{\lambda_{q+1}}\cdot\nabla\right)(\v_{\ell}+\z_{\ell})\right\|_{C^N_{t,x}}
 	 & \lesssim \lambda_{q}^{-\frac{23}{4}}  \sum_{k=0}^{N} \|\nabla a_{\zeta}\|_{C^k_{t,x}} \|\nabla(\v_{\ell}+\z_{\ell})\|_{C^{N-k}_{t,x}}
 	\nonumber\\& \lesssim  \lambda_{q}^{-\frac{23}{4}}  \sum_{k=0}^{N} m_q^{-(k+1)} \lambda_q^{3} \lambda_q^{\frac{8}{5}(N-k+1)} (\|\v_{q}\|_{C^{0}_{[t-1,t+1],x}}+\|\z_{q}\|_{C^{0}_{[t-1,t+1],x}})
      \lesssim m_q^{-(N+1)}  \lambda_{q}^{-\frac{49}{60}} .
 \end{align}
Let us estimate $\|a_{\zeta}(\nabla\Phi_{k,j}-\Id)\cdot\nabla(\v_{\ell}+\z_{\ell})\|_{C^0_{t}C^N_x}$ as follows: using \eqref{ak_0}, \eqref{ell} and \eqref{B.8}, we have
\begin{align}\label{5544}
	\|a_{\zeta}(\nabla\Phi_{k,j}-\Id)\cdot\nabla(\v_{\ell}+\z_{\ell})\|_{C^0_{t}C^0_x} 
	  & \lesssim   \|a_{\zeta}\|_{C^0_{t}C^0_x}\|\nabla(\v_{\ell}+\z_{\ell})\|_{C^0_{t}C^0_x} \|\nabla \Phi_{k,j}-\Id\|_{C^0_tC^{0}_x}
      \nonumber\\ & \lesssim  (\delta_{q+1}^{\frac12} + \|\mathring{R}_{q}\|^{\frac12}_{C^0_{[t-1,t+1],x}})  (\|\v_{q}\|_{C^{0}_{[t-1,t+1],x}} +  \|\z_{q}\|_{C^{0}_{[t-1,t+1],x}}),
\end{align}
and for $N\geq0$, using \eqref{548} and \eqref{B.10}, we have
\begin{align}\label{554}
	\|a_{\zeta}(\nabla\Phi_{k,j}-\Id)\cdot\nabla(\v_{\ell}+\z_{\ell})\|_{C^0_{t}C^N_x} 
	  & \lesssim  \sum_{k=0}^{N}\|a_{\zeta}\nabla(\v_{\ell}+\z_{\ell})\|_{C^0_{t}C^k_x} \|\nabla \Phi_{k,j}-\Id\|_{C^0_tC^{N-k}_x}
	\nonumber\\ & \lesssim \sum_{k=0}^{N} m_q^{-k}\lambda_q^{\frac{74}{15}} \lambda_q^{\frac{8}{5}(N-k-1)} \lesssim m_q^{-N}\lambda_q^{\frac{10}{3}}.
\end{align}
Therefore, using Lemma \ref{SPL} and  \eqref{5533}-\eqref{554}, we obtain  
\begin{align}
&\left\|\mathcal{R}\left((\omega_{q+1}^{(c)}\cdot\nabla)(\v_{\ell}+\z_{\ell})\right)\right\|_{C^{0}_{t,x}} 
	\leq \left\|\mathcal{R}\left((\omega_{q+1}^{(c)}\cdot\nabla)(\v_{\ell}+\z_{\ell})\right)\right\|_{C^0_tC^{\varpi}_x} 
    \nonumber\\ & \lesssim   \frac{m_q^{-1}\lambda_{q}^{-\frac{23}{4}} (\delta_{q+1}^{\frac12} + \|\mathring{R}_{q}\|^{\frac12}_{C^0_{[t-1,t+1],x}})  (\|\v_{q}\|_{C^{0}_{[t-1,t+1],x}}+ \|\z_{q}\|_{C^{0}_{[t-1,t+1],x}})}{\lambda_{q+1}^{1-\varpi}}
	+ \frac{m_q^{-(m+2)}  \lambda_{q}^{-\frac{49}{60}}}{\lambda_{q+1}^{m-\varpi}} 
    \nonumber\\ & \quad +   \frac{ (\delta_{q+1}^{\frac12} + \|\mathring{R}_{q}\|^{\frac12}_{C^0_{[t-1,t+1],x}})  (\|\v_{q}\|_{C^{0}_{[t-1,t+1],x}} +  \|\z_{q}\|_{C^{0}_{[t-1,t+1],x}})}{\lambda_{q+1}^{1-\varpi}}
	+ \frac{m_q^{-(m+1)}\lambda_q^{\frac{10}{3}}}{\lambda_{q+1}^{m-\varpi}}
    \nonumber\\ & \lesssim \frac{\lambda_{q}^{\frac85} (\delta_{q+1}^{\frac12} + \|\mathring{R}_{q}\|^{\frac12}_{C^0_{[t-1,t+1],x}})  (\|\v_{q}\|_{C^{0}_{[t-1,t+1],x}}+ \|\z_{q}\|_{C^{0}_{[t-1,t+1],x}})}{\lambda_{q+1}^{1-\varpi}}   + \frac{1}{7\cdot 4}\delta_{q+2}
    \nonumber\\ & \lesssim \frac{\lambda_{q}^{\frac85} (\delta_{q+1}^{\frac12} + \|\mathring{R}_{q}\|^{\frac12}_{C^0_{[t-1,t+1],x}})  (\|\v_{q}\|_{C^{0}_{[t-1,t+1],x}}+ \|\z_{q}\|_{C^{0}_{[t-1,t+1],x}})}{\lambda_{q+1}^{1-\varpi}}   + \frac{1}{7\cdot 4}\delta_{q+2} ,
	\label{R_Nash_c0}
\end{align}
 where we have also used $m> 33$ and $\beta<\frac{1}{200}$.

Combining \eqref{R_nash}, \eqref{R_Nash_p0} and \eqref{R_Nash_c0}, we obtain
\begin{align}
	\left\|R_{Nash}\right\|_{C^{0}_{t,x}}  &  \lesssim  \left\|\mathcal{R}\left((\omega_{q+1}^{(p)}\cdot\nabla)(\v_{\ell}+\z_{\ell})\right)\right\|_{C^{0}_{t,x}} +\left\|\mathcal{R}\left((\omega_{q+1}^{(c)}\cdot\nabla)(\v_{\ell}+\z_{\ell})\right)\right\|_{C^{0}_{t,x}}  
  \nonumber\\ & \lesssim  \frac{\lambda_{q}^{\frac85}   (\delta_{q+1}^{\frac12} + \|\mathring{R}_{q}\|^{\frac12}_{C^0_{[t-1,t+1],x}})  (\|\v_{q}\|_{C^{0}_{[t-1,t+1],x}}+\|\z_{q}\|_{C^{0}_{[t-1,t+1],x}})}{\lambda_{q+1}^{1-\varpi}}     + \frac{1}{7\cdot 2}\delta_{q+2} \label{R_Nash_C0_N}
	 \\ & \lesssim \frac{\lambda_{q}^{\frac85}   \lambda_{q}^{\frac13}  (\lambda_{q}^{\frac13}+\lambda_{q}^{\frac13})}{\lambda_{q+1}^{1-\varpi}}   + \frac{1}{7\cdot 2}\delta_{q+2}
  \leq \frac{1}{7} \lambda_{q+1}^{\frac23},\label{R_Nash_C0}
\end{align}
where we have also used \eqref{v_q_C0}, \eqref{R_q_C0}, \eqref{z_q_C0},  $\varpi< \frac{73}{1380} - \frac{228}{23}\beta < \frac{15}{276} $, $m> 33$ and $\beta<\frac{1}{200}$, and we select $a$ that is big enough to absorb the constant.  
Taking the expectation of \eqref{R_Nash_C0_N} and using H\"older's inequality, we get for $0<\beta< \frac{1}{200}$
\begin{align}\label{R_Nash_EN}
	 \EN R_{Nash}\EN_{C^0,r}
	 & \lesssim   \frac{\lambda_{q}^{\frac85}  \delta_{q+1}^{\frac12} }{\lambda_{q+1}^{1-\varpi}}  + \frac{1}{7\cdot 2}\delta_{q+2}
	   \leq \frac{1}{7}\delta_{q+2},
\end{align}
where we have used \eqref{z_q_EN1}, \eqref{v_q_EN},  \eqref{R_q_EN}, $\varpi< \frac{73}{1380} - \frac{228}{23}\beta  $ and $m> 33$, and we select $a$ that is big enough to absorb the constant.    

\subsubsection{Estimate on $R_{osc}$.}
Let us recall the definition of $R_{osc}$ as follows:
\begin{align*}
		R_{osc}&:=   \sum_{j,j', \zeta+\zeta'\neq0} \mathcal{R} \bigg[\biggl\{\Wb_{\zeta}\otimes\Wb_{\zeta'}-\frac{\Wb_{\zeta}\cdot\Wb_{\zeta'}}{2} \Id \biggr\} \nabla\left(a_{\zeta}a_{\zeta'} \phi_{\zeta}\phi_{\zeta'}\right)\bigg].
\end{align*}
In order to apply Lemma \ref{SPL}, let us rewrite oscillation error as follows, for $t\in[k,k+1]$, $k\in\mathbb{Z}$
\begin{align*}
	R_{osc}&:=   \sum_{j,j', \zeta+\zeta'\neq0} \mathcal{R} \bigg[\biggl\{B_{\zeta}\otimes B_{\zeta'}-\frac{B_{\zeta}\cdot B_{\zeta'}}{2} \Id \biggr\}e^{i\lambda_{q+1}\{\zeta\cdot\Phi_{k,j}+\zeta'\cdot\Phi_{k,j'}\}}
	\nonumber\\ & \qquad \times \left(\nabla\left(a_{\zeta}a_{\zeta'}\right) + i\lambda_{q+1}a_{\zeta}a_{\zeta'}\left((\nabla\Phi_{k,j}-\Id)\cdot\zeta+(\nabla\Phi_{k,j'}-\Id)\cdot\zeta'\right)\right)\bigg].
\end{align*}

We have (using \eqref{ell}, \eqref{ak_0}-\eqref{ak_1} and \eqref{ak_N})
\begin{align}\label{a_kk'_0}
	\|\nabla(a_{\zeta}a_{\zeta'})\|_{C^0_{t,x}} & \lesssim  \|a_{\zeta'}\|_{C^1_{t,x}} \|a_{\zeta}\|_{C^{0}_{t,x}} + \|a_{\zeta}\|_{C^1_{t,x}} \|a_{\zeta'}\|_{C^{0}_{t,x}}
	\lesssim m_q^{-1}(\delta_{q+1} + \|\mathring{R}_{q}\|_{C^0_{[t-1,t+1],x}}),
\end{align}
and for $N\geq0$,
\begin{align}
	\|\nabla(a_{\zeta}a_{\zeta'})\|_{C^N_{t,x}} & \lesssim  \sum_{k=0}^{N+1}\|a_{\zeta}\|_{C^k_{t,x}} \|a_{\zeta'}\|_{C^{N+1-k}_{t,x}}
    \lesssim  \sum_{k=0}^{N+1} m_q^{-k}\lambda_q^3 m_q^{-(N+1-k)}\lambda_q^3 
      \lesssim  m_q^{-(N+1)}\lambda_q^6.
    \label{a_kk'_N}
\end{align}
Using \eqref{ak_0}, \eqref{a_kk'_N}, \eqref{B.8} and \eqref{B.10}, we have 
\begin{align}\label{5622}
	&\lambda_{q+1}\|a_{\zeta}a_{\zeta'}\left((\nabla\Phi_{k,j}-\Id)\cdot\zeta+(\nabla\Phi_{k,j'}-\Id)\cdot\zeta'\right)\|_{C^0_tC^0_x}
	\nonumber\\& \lesssim    \lambda_{q}^{\frac{23}{4}} \|a_{\zeta}\|_{C^0_tC^0_x} \|a_{\zeta'}\|_{C^0_tC^0_x}\|(\nabla\Phi_{k,j}-\Id)\cdot\zeta+(\nabla\Phi_{k,j'}-\Id)\cdot\zeta'\|_{C^0_tC^{0}_x}
      \lesssim \lambda_q^{\frac{83}{20}} (\delta_{q+1} + \|\mathring{R}_{q}\|_{C^0_{[t-1,t+1],x}}),
\end{align}
and for $N\geq0$
\begin{align}\label{562}
	&\lambda_{q+1}\|a_{\zeta}a_{\zeta'}\left((\nabla\Phi_{k,j}-\Id)\cdot\zeta+(\nabla\Phi_{k,j'}-\Id)\cdot\zeta'\right)\|_{C^0_tC^N_x}
	\nonumber\\& \lesssim    \lambda_{q}^{\frac{23}{4}} \sum_{k=0}^{N}\|a_{\zeta}a_{\zeta'}\|_{C^0_tC^k_x}\|(\nabla\Phi_{k,j}-\Id)\cdot\zeta+(\nabla\Phi_{k,j'}-\Id)\cdot\zeta'\|_{C^0_tC^{N-k}_x}
      \lesssim    \lambda_{q}^{\frac{23}{4}} \sum_{k=0}^{N} m_q^{-k}\lambda_q^6   \lambda_q^{\frac{8}{5}(N-k-1)}
     \lesssim    m_q^{-N} \lambda_q^{ \frac{203}{20}}.
\end{align}
Applying Lemma \ref{SPL} and using \eqref{a_kk'_0}-\eqref{562}, we get  
\begin{align}
	 \|R_{osc}\|_{C^0_{t,x}}   
     &  \lesssim  \frac{m_q^{-1}(\delta_{q+1} + \|\mathring{R}_{q}\|_{C^0_{[t-1,t+1],x}}) }{\lambda_{q+1}^{1-\varpi}} 
	 + \frac{m_q^{-(m+2)} \lambda_q^{6} + m_q^{-(m+1)} \lambda_q^{ \frac{203}{20}}}{\lambda_{q+1}^{m-1}}
     \nonumber\\ &   \lesssim  \frac{m_q^{-1}(\delta_{q+1} + \|\mathring{R}_{q}\|_{C^0_{[t-1,t+1],x}}) }{\lambda_{q+1}^{1-\varpi}} 
	 + \frac{1}{7\cdot 2}\delta_{q+2}
     \label{R_osc_C01} 
	 \\ &  \lesssim    \frac{m_q^{-1}\lambda_{q}^{\frac{2}{3}}}{\lambda_{q+1}^{1-\varpi}} + \frac{1}{7\cdot 2}\delta_{q+2}
	  \leq \frac17\lambda_{q+1}^{\frac23},
	 \label{R_osc_C0}
\end{align}
where we have also used \eqref{R_q_C0}, $\varpi< \frac{73}{1380} - \frac{228}{23}\beta < \frac{15}{276}  $, $m > 33$ and $\beta< \frac{1}{200}$, and we select $a$ that is big enough to absorb the constant.  Taking expectation of \eqref{R_osc_C01}, we obtain for $0<\beta<  \frac{1}{200}$
\begin{align}\label{R_osc_EN}
	\EN R_{osc}\EN_{C^0,r}  
 &  \lesssim \frac{m_q^{-1} \delta_{q+1} }{\lambda_{q+1}^{1-\varpi}} +  \frac{1}{7\cdot 2}\delta_{q+2}
	  \leq   \frac17\delta_{q+2},
\end{align}
where we have used \eqref{R_q_EN},  $\varpi< \frac{73}{1380} - \frac{228}{23}\beta  $ and $m > 33$, and we select $a$ that is big enough to absorb the constant.

\subsubsection{Estimate on $R_{corr}$.}
Let us recall the definition of $R_{corr}$ as follows:
\begin{align*}
		R_{corr}& :=  \mathcal{R}\bigg[(\partial_t+(\v_{\ell}+\z_{\ell})\cdot\nabla)\omega^{(c)}_{q+1}\bigg] + \omega_{q+1}^{(c)}\mathring{\otimes}\omega_{q+1}+\omega_{q+1}^{(p)}\mathring{\otimes}\omega_{q+1}^{(c)}.
\end{align*}
By \eqref{w_p_C0} and \eqref{w_c_C0},  we estimate  
 \begin{align}\label{5.52}
\|\omega_{q+1}^{(c)}\mathring{\otimes}\omega_{q+1}+\omega_{q+1}^{(p)}\mathring{\otimes}\omega_{q+1}^{(c)}\|_{C^0_{t,x}}  
& \leq 2\| \omega_{q+1}^{(c)} \|_{C^0_{t,x}}\|\omega_{q+1}^{(p)} \|_{C^0_{t,x}}   + \| \omega_{q+1}^{(c)}\|^2_{C^0_{t,x}}   
	\nonumber\\  & \lesssim  \lambda_q^{\frac13} \lambda_{q}^{-\frac{17}{48} - \frac{27}{8}\beta} +  \lambda_{q}^{-\frac{17}{24} - \frac{27}{4}\beta}
      \leq \frac{1}{7\cdot2}\lambda_{q+1}^{\frac23},
\end{align} 
and by H\"older's inequality, \eqref{w_p_EN} and \eqref{w_c_EN}, we estimate for  $0<\beta < \frac{1}{200}$

\begin{align}\label{5.53}
	\EN\omega_{q+1}^{(c)}\mathring{\otimes}\omega_{q+1}+\omega_{q+1}^{(p)}\mathring{\otimes}\omega_{q+1}^{(c)}\EN_{C^0,r}  
 & \leq 2 \EN \omega_{q+1}^{(c)} \EN_{C^0,2r}\EN\omega_{q+1}^{(p)} \EN_{C^0,2r} + \EN \omega_{q+1}^{(c)}\EN_{C^0,2r}^2   
	\nonumber\\ & \lesssim  \lambda_{q}^{-\frac{11}{16}-\frac{27}{8}\beta} \delta_{q+1} + \lambda_{q}^{-\frac{11}{8}-\frac{27}{4}\beta} \delta_{q+1}   
        \lesssim \lambda_{q}^{-\frac{11}{16}-\frac{27}{8}\beta} \delta_{q+1} 
	  \leq   \frac{1}{7\cdot2}\delta_{q+2}.
\end{align}

Let us now find the estimates for the error $\mathcal{R}\left[(\partial_t+(\v_{\ell}+\z_{\ell})\cdot\nabla)\omega^{(c)}_{q+1}\right]$. We rewrite
\begin{align*}
	& (\partial_t+(\v_{\ell}+\z_{\ell})\cdot\nabla)\omega^{(c)}_{q+1}
	  =\sum_{j}\sum_{\zeta\in\Lambda_j} \left((\partial_t+(\v_{\ell}+\z_{\ell})\cdot\nabla)\left(\frac{\nabla a_{\zeta}}{\lambda_{q+1}}+i a_{\zeta}(\nabla\Phi_{k,j}-\Id)\zeta\right)\right)\times B_{\zeta}e^{i\lambda_{q+1}\zeta\cdot\Phi_{k,j}}.
\end{align*}
Using \eqref{ell}, \eqref{ak_2} and \eqref{ak_N}, we obtain
 \begin{align}\label{5711}
	\|(\partial_t+(\v_{\ell}+\z_{\ell})\cdot\nabla) \frac{\nabla a_{\zeta}}{\lambda_{q+1}}\|_{C^0_{t,x}}  
	    & \lesssim \lambda_{q}^{-\frac{23}{4}} \|a_{\zeta}\|_{C^{2}_{t,x}} + \lambda_{q}^{-\frac{23}{4}}  \|\v_{\ell}+\z_{\ell}\|_{C^{0}_{t,x}} \|a_{\zeta}\|_{C^{2}_{t,x}}
    \nonumber\\ & \lesssim m_q^{-1}\lambda_q^{\frac{23}{60}}  (\delta_{q+1}^{\frac12} + \|\mathring{R}_{q}\|^{\frac12}_{C^0_{[t-1,t+1],x}}) (1+\|\v_{q}\|_{C^0_{[t-1,t+1],x}} + \|\z_{q}\|_{C^0_{[t-1,t+1],x}}),
\end{align}
and for $N\geq 0$
\begin{align}\label{571}
	& \|(\partial_t+(\v_{\ell}+\z_{\ell})\cdot\nabla) \frac{\nabla a_{\zeta}}{\lambda_{q+1}}\|_{C^N_{t,x}}  
	   \lesssim \lambda_{q}^{-b} \|a_{\zeta}\|_{C^{N+2}_{t,x}} + \lambda_{q}^{-\frac{23}{4}} \sum_{k=0}^{N}\|a_{\zeta}\|_{C^{k+2}_{t,x}} \|\v_{\ell}+\z_{\ell}\|_{C^{N-k}_{t,x}}
    \nonumber\\ & \lesssim \lambda_{q}^{-\frac{23}{4}+3} m_q^{-(N+2)}  + \lambda_{q}^{-\frac{23}{4}+3} \sum_{k=0}^{N} m_q^{-(k+2)}   \ell^{-N+k} \|\v_{q}+\z_{q}\|_{C^{0}_{[t-1,t+1],x}} 
      \lesssim \lambda_{q}^{-\frac{49}{12}} m_q^{-(N+2)} .
\end{align}
Using \eqref{Phi_kj}, \eqref{31244}-\eqref{3124},  \eqref{5544}-\eqref{554}, \eqref{5488}-\eqref{548} and \eqref{B.10}, we obtain 
 \begin{align}\label{5722}
&	\|(\partial_t+(\v_{\ell}+\z_{\ell})\cdot\nabla)\left( a_{\zeta}(\nabla\Phi_{k,j}-\Id)\right)\|_{C^0_tC^0_{x}}
\nonumber\\ & \leq \|(\partial_t+(\v_{\ell}+\z_{\ell})\cdot\nabla) a_{\zeta} \|_{C^0_tC^0_{x}} \|\nabla\Phi_{k,j}-\Id\|_{C^0_tC^0_{x}} +  \| a_{\zeta}(\nabla\Phi_{k,j}-\Id)\cdot \nabla (\v_{\ell}+\z_{\ell})\|_{C^0_tC^0_{x}}  + \| a_{\zeta}\nabla (\v_{\ell}+\z_{\ell})\|_{C^0_tC^0_{x}} 
\nonumber\\ & \lesssim m_q^{-1}(\delta_{q+1}^{\frac12} + \|\mathring{R}_{q}\|^{\frac12}_{C^0_{[t-1,t+1],x}}) (1+\|\v_{q}\|_{C^0_{[t-1,t+1],x}} + \|\z_{q}\|_{C^0_{[t-1,t+1],x}}) \lambda_q^{-\frac85} 
\nonumber\\ & \quad +  (\delta_{q+1}^{\frac12} + \|\mathring{R}_{q}\|^{\frac12}_{C^0_{[t-1,t+1],x}})  (\|\v_{q}\|_{C^{0}_{[t-1,t+1],x}}+ \|\z_{q}\|_{C^{0}_{[t-1,t+1],x}})  
\nonumber\\ & \quad +  \lambda_{q}^{\frac85} (\delta_{q+1}^{\frac12} + \|\mathring{R}_{q}\|^{\frac12}_{C^0_{[t-1,t+1],x}})  (\|\v_{q}\|_{C^{0}_{[t-1,t+1],x}}+ \|\z_{q}\|_{C^{0}_{[t-1,t+1],x}}) 
\nonumber\\ & \lesssim m_q^{-1} \lambda_q^{-\frac{8}{5}}(\delta_{q+1}^{\frac12} + \|\mathring{R}_{q}\|^{\frac12}_{C^0_{[t-1,t+1],x}}) (1+\|\v_{q}\|_{C^0_{[t-1,t+1],x}} + \|\z_{q}\|_{C^0_{[t-1,t+1],x}}).
\end{align}
and for $N\geq 0$
\begin{align}\label{572}
&	\|(\partial_t+(\v_{\ell}+\z_{\ell})\cdot\nabla)\left( a_{\zeta}(\nabla\Phi_{k,j}-\Id)\right)\|_{C^0_tC^N_{x}}
\nonumber\\ & \leq \|[(\partial_t+(\v_{\ell}+\z_{\ell})\cdot\nabla) a_{\zeta} ] (\nabla\Phi_{k,j}-\Id)\|_{C^0_tC^N_{x}} +  \| a_{\zeta}(\nabla\Phi_{k,j}-\Id)\cdot \nabla (\v_{\ell}+\z_{\ell})\|_{C^0_tC^N_{x}}  + \| a_{\zeta}\nabla (\v_{\ell}+\z_{\ell})\|_{C^0_tC^N_{x}} 
\nonumber\\ &\lesssim \sum_{k=0}^{N} \|(\partial_t+(\v_{\ell}+\z_{\ell})\cdot\nabla) a_{\zeta}\|_{C^0_tC^k_{x}} \|\nabla\Phi_{k,j}-\Id\|_{C^0_tC^{N-k}_{x}} + m_q^{-N} \lambda_q^{\frac{10}{3}}  + m_q^{-N} \lambda_q^{\frac{74}{15}}
\nonumber\\ &\lesssim \sum_{k=0}^{N} m_q^{-(k+1)} \lambda_q^{\frac{10}{3}} \lambda_q^{\frac{8}{5}(N-k-1)}    + m_q^{-N} \lambda_q^{\frac{74}{15}}
 \lesssim    m_q^{-(N+1)} \lambda_q^{\frac{26}{15}}.
\end{align}
In view of Lemma \ref{SPL} and using \eqref{5711}-\eqref{572}, we estimate 
	\begin{align}
	& \|\mathcal{R}\left[(\partial_t+(\v_{\ell}+\z_{\ell})\cdot\nabla)\omega^{(c)}_{q+1}\right]\|_{C^0_{t,x}}
   \nonumber\\  & \lesssim \frac{m_q^{-1}\lambda_q^{\frac{23}{60}} (\delta_{q+1}^{\frac12} + \|\mathring{R}_{q}\|^{\frac12}_{C^0_{[t-1,t+1],x}}) (1+\|\v_{q}\|_{C^0_{[t-1,t+1],x}} + \|\z_{q}\|_{C^0_{[t-1,t+1],x}})}{\lambda_{q+1}^{1-\varpi}} + \frac{1}{7\cdot 4}\delta_{q+2}  \label{R_corr_C01}
\\ & \lesssim \frac{m_q^{-1}\lambda_q^{\frac{23}{60}} \lambda_q^{\frac23} }{\lambda_{q+1}^{1-\varpi}} + \frac{1}{7\cdot 4}\delta_{q+2} \leq    \frac{1}{7\cdot2}\lambda_{q+1}^{\frac23},\label{R_corr_C0}
\end{align}
where we have used \eqref{z_q_C0}, \eqref{v_q_C0}, \eqref{B.10}, $\varpi< \frac{73}{1380} - \frac{228}{23}\beta < \frac{15}{276}$ and $m > 33$, and we select $a$ that is big enough to absorb the constant. 
Taking expectation of \eqref{R_corr_C01} and using H\"older's inequality, \eqref{z_q_EN1}, \eqref{v_q_EN} and \eqref{R_q_EN}, we estimate for $0<\beta< \frac{1}{200} $
\begin{align}\label{559}
	 \EN\mathcal{R}\left[(\partial_t+(\v_{\ell}+\z_{\ell})\cdot\nabla)\omega^{(c)}_{q+1}\right]\EN_{C^0,r}
	& \lesssim \lambda_q^{-\frac{23}{4} + \frac{23}{4}\varpi}  m_q^{-1}\lambda_q^{\frac{23}{60}}  \delta_{q+1}^{\frac12}  
	+ \frac{1}{7\cdot 4}\delta_{q+2}
	  \leq  \frac{1}{7\cdot2} \delta_{q+2},
\end{align}
where we have used \eqref{z_q_EN1}, \eqref{v_q_EN}, \eqref{R_q_EN},  $\varpi< \frac{73}{1380} - \frac{228}{23}\beta $ and $m > 33$, and  we select $a$ that is big enough to absorb the constant.   Combining \eqref{5.52}-\eqref{5.53} and \eqref{R_corr_C0}-\eqref{559}, we arrive at  
\begin{align}\label{R_corrr_C0}
	\left\|R_{corr}\right\|_{C^{0}_{t,x}}  &  \lesssim  \left\|\omega_{q+1}^{(c)}\mathring{\otimes}\omega_{q+1}+\omega_{q+1}^{(p)}\mathring{\otimes}\omega_{q+1}^{(c)}\right\|_{C^{0}_{t,x}} +\left\|\mathcal{R}\left[(\partial_t+(\v_{\ell}+\z_{\ell})\cdot\nabla)\omega^{(c)}_{q+1}\right]\right\|_{C^{0}_{t,x}}  
	     \leq \frac17\lambda_{q+1}^{\frac23},
\end{align} 
and 
\begin{align}\label{R_corrr_EN}
	\EN R_{corr}\EN_{C^0,r}
	&  \lesssim \EN\omega_{q+1}^{(c)}\mathring{\otimes}\omega_{q+1}+\omega_{q+1}^{(p)}\mathring{\otimes}\omega_{q+1}^{(c)}\EN_{C^0,r} +\EN\mathcal{R}\left[(\partial_t+(\v_{\ell}+\z_{\ell})\cdot\nabla)\omega^{(c)}_{q+1}\right]\EN_{C^0,r}  
	     \leq \frac17\delta_{q+2},
\end{align}
for  $0<\beta< \frac{1}{200} $.
\subsubsection{Estimate on $R_{com1}$.} 
Let us recall the definition of $R_{com1}$ as follows:
\begin{align*}
		R_{com1} &:=  (\v_\ell+\z_\ell) \mathring{\otimes} (\v_\ell+\z_\ell) - \left(\left((\v_q+\z_q) \mathring{\otimes} (\v_q+\z_q)\right)\ast_{x}\uppsi_{\ell}\right)\ast_{t}\psi_{\ell}.
\end{align*}
Using \eqref{z_q_C0} and \eqref{v_q_C0}, we estimate
\begin{align}\label{R_com1_C0}
	\|R_{com1}\|_{C^0_{t,x}}\leq 2\|\v_{q}+\z_{q}\|^2_{C^0_{t,x}}\leq 4(\lambda_{q}^{\frac23}+\lambda_{q}^{\frac23})\leq \frac{1}{7}\lambda_{q+1}^{\frac23},
\end{align}
where  we select $a$ that is big enough to absorb the constant. Expanding $R_{com1}$ and applying an $L^{\infty}$-based version of the mollification estimate from \cite[Lemma 5.1]{BCD}, we obtain the following estimate for any
$t\in\R$ and $\delta\in(0,\frac{1}{12})$:
\begin{align}\label{3.64}
	&\|R_{com1}(t)\|_{\L^{\infty}}
	  \lesssim \ell^2\|\v_{q}\|^2_{C^1_{[t-1,t],x}}+\ell^2\|\z_{q}\|^2_{C^0_{[t-1,t]}C^1_x}+\ell^{1-2\delta}\|\z_{q}\|^2_{C^{\frac12-\delta}_{[t-1,t]}C^0_x}.
\end{align}
Taking expectation of \eqref{3.64} and using \eqref{v_q_C1}, \eqref{z_q_EN1}-\eqref{z_q_EN2} and \eqref{z_q_EN3}, we find for $0<\delta<\frac{1}{12}$ and $0<\beta< \frac{1}{200}$
\begin{align}\label{R_com1_EN}
	&\EN R_{com1}\EN_{C^0,r} \lesssim \left(\ell^2\lambda_{q}^{\frac{14}{5}}\delta_{q} + \ell^2 \EN \z_{q}\EN_{C^1,2r}^2 +\ell^{1-2\delta}\EN\z_{q}\EN_{C_t^{\frac12-\delta}C^0_x,2r}^2\right)
	\nonumber\\ & \lesssim \left(\ell^2\lambda_{q}^{\frac{14}{5}}\delta_{q} + \ell^2 \lambda_{q}^{\frac23}(2rL)^2 +\ell^{1-2\delta}\lambda_{q}^{\frac23}(2rL^2)^2\right)
	 \lesssim \lambda_{q}^{-\frac25} {\delta_{q}} +\lambda_{q}^{\frac{16}{5}\delta-\frac85}\lambda_{q}^{\frac23}
	  \lesssim \lambda_{q}^{-\frac25} {\delta_{q}} +\lambda_{q}^{-\frac{4}{3}}\lambda_{q}^{\frac23}
	    \leq {\frac{1}{7}\delta_{q+2}},
\end{align}
where we select $a$ that is big enough to absorb the constant.

\subsubsection{Estimate on $R_{com2}$.}
Let us recall the definition of $R_{com2}$ as follows:
\begin{align*}
		R_{com2}& :=   \v_{q+1}\mathring{\otimes}(\z_{q+1}-\z_{\ell})   + (\z_{q+1}-\z_{\ell})\mathring{\otimes} \v_{q+1} + \z_{q+1}\mathring{\otimes}\z_{q+1}
		- \z_{\ell}\mathring{\otimes}\z_{\ell} - \mathcal{R}(\z_{q+1}-\z_{\ell}).
\end{align*}
Using \eqref{v_q+1_C0}, \eqref{z_q_C0} and \eqref{IDO_B}, we estimate
\begin{align}
	\|\v_{q+1}\mathring{\otimes}(\z_{q+1}-\z_{\ell})   + (\z_{q+1}-\z_{\ell})\mathring{\otimes} \v_{q+1}\|_{C^0_{t,x}} 
	 & \leq  2\|\v_{q+1}\|_{C^0_{t,x}}\|\z_{q+1}-\z_{\ell}\|_{C^0_{t,x}} 
	\nonumber \\ & \leq \frac{1}{7\cdot4\cdot2}\left[\lambda_{q+1}^{\frac23}  + \lambda_{q+1}^{\frac13}\lambda_{q}^{\frac13}\right]
	  \leq \frac{1}{7\cdot4}\lambda_{q+1}^{\frac23},\label{5.68}\\
	\|\z_{q+1}\mathring{\otimes}\z_{q+1}
	- \z_{\ell}\mathring{\otimes}\z_{\ell}\|_{C^0_{t,x}}  \leq \|\z_{q+1}\|_{C^0_{t,x}}^2 + \|\z_{q}\|_{C^0_{t,x}}^2
	  & \lesssim \frac{1}{(7\cdot4\cdot2\cdot2)^2} \left[ \lambda_{q+1}^{\frac23}+ \lambda_{q}^{\frac23} \right]
	  \leq \frac{1}{7\cdot4}\lambda_{q+1}^{\frac23},\label{5.69}\\
	\|\mathcal{R}(\z_{q+1}-\z_\ell)\|_{C^0_{t,x}}  \leq \|\z_{q+1}-\z_\ell\|_{C^0_{t,x}} 
	 &  \leq \frac{1}{7\cdot4\cdot2\cdot2}\left[\lambda_{q+1}^{\frac13}+\lambda_{q}^{\frac13}\right]
	  \leq \frac{1}{7\cdot4}\lambda_{q+1}^{\frac23},\label{5.70}
\end{align}
where we select $a$ that is big enough to absorb the constant. Combining \eqref{5.68}-\eqref{5.70}, we get 
\begin{align}\label{R_com2_C0}
	\|R_{com2}\|_{C^0_{t,x}} \leq \frac{1}{7}\lambda_{q+1}^{\frac23}.
\end{align}
Let us now do the expectation estimates for $R_{com2}$. Using H\"older's inequality, \eqref{z_diff_EN} and \eqref{v_q+1_EN}, we estimate
\begin{align*}
	\EN\v_{q+1}\mathring{\otimes}(\z_{q+1}-\z_{q})   + (\z_{q+1}-\z_{q})\mathring{\otimes} \v_{q+1}\EN_{C^0,r} &\leq \EN\v_{q+1}\EN_{C^0,2r} \EN\z_{q+1}-\z_{q}\EN_{C^0,2r}
\lesssim (1-\delta_{q+1}^{\frac12}) \lambda_{q}^{-(\frac13-\varepsilon)}
	  \leq \lambda_{q}^{-(\frac13-\varepsilon)}.
\end{align*}
Using H\"older's inequality, \eqref{z_q_EN1} and \eqref{z_diff_EN}, we obtain
\begin{align*}
	\EN\z_{q+1}\mathring{\otimes}(\z_{q+1}-\z_{q})  \EN_{C^0,r} &\leq \EN\z_{q+1}\EN_{C^0,2r} \EN\z_{q+1}-\z_{q}\EN_{C^0,2r}
	  \lesssim  \lambda_{q}^{-(\frac13-\varepsilon)}.
\end{align*}
  A standard mollification estimate gives for any $\delta\in(0,\frac{1}{12})$
\begin{align*}
	\|\z_{\ell}(t)- \z_{q}(t)\|_{\L^{\infty}}\lesssim \ell \|\z_{q}\|_{C^0_{[t-1,t]}C^1_{x}} + \ell^{\frac12-\delta}\|\z_{q}\|_{C^{\frac12-\delta}_{[t-1,t]}C^0_x},
\end{align*}
which gives (using \eqref{z_q_EN2} and \eqref{z_q_EN3})
\begin{align}\label{z_ell_diff}
	\EN\z_{\ell}-\z_{q}\EN_{C^0,2r} & \lesssim \ell\sup_{t\in\R}\bigg(\mathbb{E}\left[\|\z_{q}\|^{2r}_{C^0_{[t-1,t+1]}C^1_x}\right]\bigg)^{\frac{1}{2r}} + \ell^{\frac12-\delta}\sup_{t\in\R}\bigg(\mathbb{E}\left[\|\z_{q}\|^{2r}_{C^{\frac12-\delta}_{[t-1,t+1]}C^0_x}\right]\bigg)^{\frac{1}{2r}}
	 \lesssim \ell\lambda_{q}^{\frac13} + \ell^{\frac{5}{12}}\lambda_{q}^{\frac{1}{3}} \lesssim \lambda_{q}^{-\frac13}.
\end{align}
In view of  H\"older's inequality, \eqref{z_q_EN1}, \eqref{z_diff_EN}, \eqref{v_q+1_EN} and \eqref{z_ell_diff}, we estimate for  $0<\beta < \frac{1}{200}$
\begin{align}\label{R_com2_EN}
	& \EN R_{com2}\EN_{C^0,r}
	\nonumber\\&  \leq \EN \v_{q+1}\mathring{\otimes}(\z_{q+1}-\z_{\ell})\EN_{C^0,r}   + \EN(\z_{q+1}-\z_{\ell})\mathring{\otimes} \v_{q+1}\EN_{C^0,r} + \EN\z_{q+1}\mathring{\otimes}\z_{q+1}
	- \z_{\ell}\mathring{\otimes}\z_{\ell}\EN_{C^0,r} 
	 +\EN\mathcal{R}(\z_{q+1}-\z_{\ell})\EN_{C^0,r}
	\nonumber\\ & \leq  \EN \v_{q+1}\mathring{\otimes}(\z_{q+1}-\z_{q})\EN_{C^0,r} + \EN \v_{q+1}\mathring{\otimes}(\z_{q}-\z_{\ell})\EN_{C^0,r} + \EN(\z_{q+1}-\z_{q})\mathring{\otimes} \v_{q+1}\EN_{C^0,r}  
	\nonumber \\& \quad + \EN(\z_{q}-\z_{\ell})\mathring{\otimes} \v_{q+1}\EN_{C^0,r}  + \EN\z_{q+1}\mathring{\otimes}(\z_{q+1}-\z_{q})\EN_{C^0,r} +\EN(\z_{q+1}-\z_{q})\mathring{\otimes}\z_{q+1}\EN_{C^0,r} 
	\nonumber\\ & \quad +\EN(\z_{q}-\z_{\ell})\mathring{\otimes}\z_{q}\EN_{C^0,r} +\EN\z_{\ell}\mathring{\otimes}(\z_{q}-\z_{\ell})\EN_{C^0,r} +\EN\mathcal{R}(\z_{q+1}-\z_{q})\EN_{C^0,r} +\EN\mathcal{R}(\z_{q}-\z_{\ell})\EN_{C^0,r}
	\nonumber\\ & \leq  2\EN \v_{q+1}\EN_{C^0,2r}\EN\z_{q+1}-\z_{q}\EN_{C^0,2r} + 2\EN \v_{q+1}\EN_{C^0,2r}\EN\z_{q}-\z_{\ell}\EN_{C^0,2r} 
	   	  + 2 \EN\z_{q+1}\EN_{C^0,2r}\EN\z_{q+1}-\z_{q}\EN_{C^0,2r}  \nonumber\\ & \quad
 +2\EN\z_{q}-\z_{\ell}\EN_{C^0,2r}\EN\z_{q}\EN_{C^0,2r}  +\EN\z_{q+1}-\z_{q}\EN_{C^0,r} +\EN\z_{q}-\z_{\ell}\EN_{C^0,r}
 \nonumber \\ & \lesssim  \lambda_{q}^{-(\frac13-\varepsilon)} + \lambda_{q}^{-\frac13}
   \lesssim  \lambda_{q}^{-(\frac13-\varepsilon)}
    \leq    \frac17\delta_{q+2},
\end{align}
where we have used  $0<\varepsilon< \frac{1}{198.375}-\beta$ and  we select $a$ that is big enough to absorb the constant. 

Combining \eqref{R_lin2}-\eqref{R_lin_r}, \eqref{R_tra_C0}-\eqref{R_tra_EN}, \eqref{R_Nash_C0}-\eqref{R_Nash_EN}, \eqref{R_osc_C0}-\eqref{R_osc_EN}, \eqref{R_corrr_C0}-\eqref{R_corrr_EN}, \eqref{R_com1_C0}-\eqref{R_com1_EN} and \eqref{R_com2_C0}-\eqref{R_com2_EN}, we  obtain
\begin{align*}
	\|\mathring{R}_{q+1}\|_{C^0_{t,x}}\leq \lambda_{q+1}^{\frac23} \;\; \text{ and }\;\; \EN\mathring{R}_{q+1}\EN_{C^0,r}\leq \delta_{q+2}.
\end{align*}
Hence the inductive estimates \eqref{R_q_C0} and \eqref{R_q_EN} hold true at the level $q+1$.

\section{Stationary solutions to stochastic hypodissipative Navier-Stokes equations}\setcounter{equation}{0}\label{sec6}
 Let us take the trajectory space $\mathcal{T}=C(\R;C^{\kappa})\times C(\R;C^{\kappa})$ and $\kappa=\frac{\vartheta}{2}$ in this section. The corresponding shifts $S_{t}$, $t\in\R$, on trajectories are given by 
\begin{align*}
	S_{t}(\u,\W)(\cdot)=(\u(\cdot+t),\W(\cdot+t)-\W(t)),\;\;\; t\in\R, \;\; \; (\u,\W)\in\mathcal{T}.
\end{align*}
 Recall the concept of stationary solution that was presented in Definition \ref{SS}. Our main result of this section is existence of stationary solutions which are constructed as limits of ergodic averages of solutions from Theorem \ref{MR-SHNSE}.  Specifically, this indicates that they are not unique.
\begin{theorem}\label{MR-SS-SHNSE}
	Let $\u$ be a solution of system \eqref{eqn_stochatic_u-SHNSE} obtained in Theorem \ref{MR-SHNSE} and satisfying \eqref{u_EN}. Then there exists a sequence $T_{n}\to\infty$ and a stationary solution $\left((\widetilde{\Omega},\widetilde{\mathcal{F}},\widetilde{\mathbb{P}}),\widetilde{\u},\widetilde{\W}\right)$ to \eqref{eqn_stochatic_u-SHNSE} such that 
	\begin{align*}
		\frac{1}{T_{n}}\int_{0}^{T_{n}}\mathcal{L}\left[S_{t}(\u,\W)\right]\d t\to \mathcal{L}[\widetilde{\u},\widetilde{\W}],
	\end{align*}
weakly in the sense of probability measures on $\mathcal{T}$ as $n\to\infty$.
\end{theorem}
\begin{proof}
	First, we claim that there exists $\beta^{\prime}$, $\beta^{\prime\prime}>0$ such that for any $N\in\N$
	\begin{align*}
		\sup_{s\in\R}\mathbb{E}\left[\|\u(\cdot+s)\|^{2r}_{C^{\beta^{\prime}}([-N,N];C^{\beta^{\prime\prime}}_{x})}\right]\lesssim N.
	\end{align*}
Indeed, reviewing the proof of Theorem \ref{MR-SHNSE}, and by interpolation inequality, we deduce for some $\beta^{\prime}\in(0,\frac{\vartheta}{2})$ and $\beta^{\prime\prime}\in(\frac{\vartheta}{2},\vartheta)$ satisfying $\beta^{\prime}+\beta^{\prime\prime}<\vartheta<\frac57\beta$
	\begin{align*}
	&\sum_{q\geq0}\EN\v_{q+1}-\v_{q}\EN_{C^{\beta^{\prime}}_{t}C^{\beta^{\prime\prime}}_{x},2r} \lesssim
	\sum_{q\geq0}\EN\v_{q+1}-\v_{q}\EN_{C^{\vartheta}_{t,x},2r} 
	\nonumber\\ & \lesssim \sum_{q\geq0} \EN\v_{q+1}-\v_{q}\EN_{C^{0},2r}^{1-\vartheta} \EN\v_{q+1}-\v_{q}\EN_{C^{1}_{t,x},2r}^{\vartheta} 
	 \lesssim \sum_{q\geq0} \delta_{q+1}^{\frac{1-\vartheta}{2}}\lambda_{q+1}^{\frac{7\vartheta}{5}} \delta_{q+1}^{\frac{\vartheta}{2}} 
	  \leq \sqrt{3r}L a^{\frac75b\vartheta}+ \lambda_2^{\beta}\sum_{q\geq1}\lambda_{q+1}^{\frac75\vartheta-\beta}<\infty.
\end{align*}
Hence, we conclude that $\v=\lim\limits_{q\to\infty}\v_{q}$ exists and belongs to the space $\L^{2r}(\Omega;C^{\beta^{\prime}}(\R;C^{\beta^{\prime\prime}}))$. Similarly, using \eqref{z_q_EN2}, \eqref{z_q_EN3}, \eqref{z_q_EN4}, \eqref{z_diff_EN} and interpolation inequality, we deduce for same $\beta^{\prime}\in(0,\frac{\vartheta}{2})$ and $\beta^{\prime\prime}\in(\frac{\vartheta}{2},\vartheta)$ satisfying $\beta^{\prime}+\beta^{\prime\prime}<\vartheta<\frac57\beta$ and any $p\geq2$
	\begin{align*}
	&	\sum_{q\geq0}\EN\z_{q+1}-\z_{q}\EN_{C^{\beta^{\prime}}_{t}C^{\beta^{\prime\prime}}_{x},p}
		\nonumber\\
	& \lesssim \sum_{q\geq0}\EN\z_{q+1}-\z_{q}\EN_{C^{0}_{t,x},p}^{(\frac14-\beta^{\prime})(1-\beta^{\prime\prime})}\EN\z_{q+1}-\z_{q}\EN_{C^{0}_{t}C^{1}_{x},p}^{(\frac14-\beta^{\prime})\beta^{\prime\prime}}\EN\z_{q+1}-\z_{q}\EN_{C^{\frac14}_{t}C^{0}_{x},p}^{\beta^{\prime}(1-\beta^{\prime\prime})}\EN\z_{q+1}-\z_{q}\EN_{C^{\frac14}_{t}C^{1}_{x},p}^{\beta^{\prime}\beta^{\prime\prime}}
	\nonumber\\  & \lesssim \sum_{q\geq0}\lambda_{q}^{-(\frac13-\varepsilon)(\frac14-\beta^{\prime})(1-\beta^{\prime\prime})}\lambda_{q+1}^{\frac13(\frac14-\beta^{\prime})\beta^{\prime\prime}}\lambda_{q+1}^{\frac13\beta^{\prime}(1-\beta^{\prime\prime})}\lambda_{q+1}^{\frac23\beta^{\prime}\beta^{\prime\prime}}
	\nonumber\\  & = \sum_{q\geq0}\lambda_{q}^{-(\frac13-\varepsilon)(\frac14-\beta^{\prime})(1-\beta^{\prime\prime})+\frac23(\beta^{\prime}+\frac14\beta^{\prime\prime})}
  \lesssim \sum_{q\geq0}\lambda_{q}^{-\frac16(\frac14-\beta^{\prime})(1-\beta^{\prime\prime})+\frac23(\beta^{\prime}+\frac14\beta^{\prime\prime})}
	   \lesssim \sum_{q\geq0}\lambda_{q}^{-\frac{5}{6}\left[\frac{1}{20}-(\beta^{\prime}+\frac14\beta^{\prime\prime})\right]}<\infty,
\end{align*}
where we have used  $\varepsilon< \frac{1}{198.375} -\beta$ and  $\beta^{\prime}+\frac{\beta^{\prime\prime}}{4} < \vartheta <\frac{1}{20}$. Hence, we get $\lim\limits_{q\to\infty}\z_{q}=\z$ in $\L^{p}(\Omega;C^{\beta^{\prime}}(\R;C^{\beta^{\prime\prime}}))$, for any $p\geq2$. Then letting $\u=\v+\z$, we get an $\{\mathcal{F}_{t}\}_{t\in\R}$-adapted analytically weak solution $\u$ to \eqref{eqn_stochatic_u-SHNSE} and 
\begin{align*}
	\EN\u\EN_{C^{\beta^{\prime}}_{t}C^{\beta^{\prime\prime}}_{x},2r}<\infty.
\end{align*}
The rest of the proof follows by Arzelà–Ascoli Theorem, Jakubowski–Skorokhod representation theorem and the same argument as in \cite[Theorem 6.1]{Lu+Zhu_Arxiv} and \cite[Theorem 4.1]{Hofmanova+Zhu+Zhu_Arxiv}, therefore we are not repeating here.
\end{proof}

\section{3D stochastic Euler equations}\setcounter{equation}{0}\label{sec7}
In this section, our aim is to prove the first main result for system \eqref{eqn_stochatic_u-SEE}, that is, we will provide the proof of Theorem \ref{FMR2}. The steps of the proof of the main result for system \eqref{eqn_stochatic_u-SEE} is similar to the system \eqref{eqn_stochatic_u-SHNSE}. We shall provide a quick proof based on the calculations given in the previous sections. 

 Let us first break down the stochastic Euler equations \eqref{eqn_stochatic_u-SEE} into two components: a random partial differential equation and a linear component that incorporates noise. In particular, we decompose $\u=\mathfrak{v}+\mathfrak{z}$ such that $\mathfrak{z}$ solves the following stochastic linear system:
\begin{equation}\label{eqn_z}
	\left\{
	\begin{aligned}
		\d\mathfrak{z}  + \mathfrak{z}\d t &=\d \mathrm{W},\\
		\mathrm{div}\;\mathfrak{z}&=0,
	\end{aligned}
	\right.
\end{equation}
where $\W$ is same as given in Definition \ref{AWS}, and $\mathfrak{v}$ solves the following non-linear random system:
\begin{equation}\label{eqn_random_w2}
	\left\{
	\begin{aligned}
		\partial_t\mathfrak{v}   +\mathrm{div}\left((\mathfrak{v}+\mathfrak{z}) \otimes (\mathfrak{v}+\mathfrak{z}) \right)  -\mathfrak{z} +\nabla \pi &=\boldsymbol{0},\\
		\mathrm{div}\;\mathfrak{v}&=0.
	\end{aligned}
	\right.
\end{equation}
  The pressure term associated with $\mathfrak{v}$ is denoted by $\pi$, and $\mathfrak{z}$ is divergence-free based on the assumptions on the noise $\W$.  Let us now provide the regularity results for $\mathfrak{z}$ in the following Proposition:

\begin{proposition}[{\cite[Proposition 3.1]{Lu+Zhu_Arxiv}}]\label{thm_z}
	Suppose that  $\Tr\left((-\Delta)^{\frac52}GG^{\ast}\right)<\infty$. Then for any $\delta\in(0,\frac12)$, $p\geq2$
	\begin{align*}
		\sup_{t\in\R}	\mathbb{E}\left[\|\mathfrak{z}\|^p_{C^{\frac12-\delta}_{t}C^{0}}\right]\lesssim \sup_{t\in\R}	\mathbb{E}\left[\|\mathfrak{z}\|^p_{C^{\frac12-\delta}_{t}\H^{\frac52}}\right] \leq (p-1)^{\frac{p}{2}}L^{p},
	\end{align*}
	where $L\geq \max\left\{1 , \frac{1+\bar{M}}{\sqrt{6r}}\right\}$
 depends on $\Tr\left((-\Delta)^{\frac52}GG^{\ast}\right)$, $\delta$ and is independent of $p$.
\end{proposition}

\subsection{Iterative proposition}
Now, we apply the convex integration method to the nonlinear equation \eqref{eqn_random_w2}, see Sections \ref{sec3}, \ref{sec4} and \ref{sec5} for more details. The convex integration iteration is indexed by a parameter $q\in\N_0$. In this section, we consider $\lambda_{q}$, $\delta_{q}$ and $\ell$ same as in the previous sections. At each step $q$, a pair $(\mathfrak{v}_{q}, \mathring{N}_{q})$ is constructing solving the following system:
\begin{equation}\label{eqn_w_q}
	\left\{
	\begin{aligned}
		\partial_t\mathfrak{v}_q  +\mathrm{div}\left((\mathfrak{v}_q+\mathfrak{z}_q) \otimes (\mathfrak{v}_q+\mathfrak{z}_q) \right) - \mathfrak{z}_q +\nabla \pi_q &=\diver \mathring{N}_q \\
		\mathrm{div}\;\mathfrak{v}_q&=0.
	\end{aligned}
	\right.
\end{equation}
In the above, we define
\begin{align*}
	\mathfrak{z}_q(t,x) : =\chi_q\left(\|\widetilde{\mathfrak{z}}_q(t)\|_{C^0_{x}}\right)\wi\chi_q\left(\|\widetilde{\mathfrak{z}}_q(t)\|_{C^1_{x}}\right)\widetilde{\mathfrak{z}}_q(t,x),
\end{align*}
and $\widetilde{\mathfrak{z}}_q=\mathbb{P}_{\leq f(q)}\mathfrak{z}$, where $\chi_q$, $\wi\chi_{q}$ and $\mathbb{P}_{\leq f(q)}$ are same as given in Subsection \ref{Itra+Pro}. $\mathring{N}_{q}$ on the right hand side of \eqref{eqn_w_q} is a $3\times3$ matrix which is trace-free and we put he trace part into the pressure. Note that $\mathfrak{z}_q$ satisfies the similar estimates as $\z_{q}$, for example $\|\mathfrak{z}_q\|_{C^0_{t,x}} \leq \frac12\lambda_q^{\frac13}$, $\|\mathfrak{z}_q\|_{C^0_tC^1_{x}} \leq \lambda_{q}^{\frac23}$, $\EN \mathfrak{z}_q\EN_{C^0,p}  \leq (p-1)^{\frac12}L$, etc. (see Remarks \ref{Rem3.2} and \ref{Rem3.3}).    One can obtain the following lemma from Lemma \ref{Lemma-z_diff_EN}.
\begin{lemma}\label{Lemma-y_diff_EN}
	Suppose that  $\Tr\left((-\Delta)^{\frac52}GG^{\ast}\right)$.  For any $p\geq2$, we have for $0<\varepsilon<\frac13$,
	\begin{align*}
		\EN\mathfrak{z}_{q+1}-\mathfrak{z}_{q}\EN_{C^0,p}\lesssim 
		pL^2\lambda_{q}^{-(\frac13-\varepsilon)}.
	\end{align*}
\end{lemma}

Under the above assumptions, we present the following iteration proposition:

\begin{proposition}\label{Iterations-SEE}
	Suppose that  $\Tr\left((-\Delta)^{\frac52}GG^{\ast}\right)$ and let $r>1$ be fixed. Then, for any $\beta\in(0,\frac{1}{200})$, there exists a choice of parameter $a$ such that following holds true:
	
	Let $(\mathfrak{v}_{q},\mathring{N}_{q})$ for some $q\in\N_0$ be an $\{\mathcal{F}_{t}\}_{t\in\R}$-adapted solution to the system \eqref{eqn_w_q} satisfying 
	\begin{align}
		\|\mathfrak{v}_{q}\|_{C^0_{t,x}} &\leq \lambda_{q}^{\frac13}, \label{w_q_C0}\\
		\|\mathfrak{v}_{q}\|_{C^1_{t,x}} & \leq \lambda_{q}^{\frac75}\delta_{q}^{\frac12}, \label{w_q_C1}\\
		 \EN\mathfrak{v}_{q}\EN_{C^0,2r} &  \leq 6rL^2 -\delta_{q}^{\frac12},  \label{w_q_EN}\\
		\|\mathring{N}_q\|_{C^0_{t,x}} & \leq \lambda_{q}^{\frac23}, \label{N_q_C0}\\
		\EN \mathring{N}_{q}\EN_{C^0,r} &\leq \delta_{q+1}   \label{N_q_EN}.
	\end{align}
	There exists an {$\{\mathcal{F}_{t}\}_{t\in\R}$}-adapted process $(\mathfrak{v}_{q+1},\mathring{N}_{q+1})$ which solves the system \eqref{eqn_w_q}, obeys \eqref{w_q_C0}-\eqref{N_q_EN} at the level $q+1$ and satisfies 
	\begin{align}\label{w_diff_EN}
		\EN\mathfrak{v}_{q+1}-\mathfrak{v}_{q}\EN_{C^0,2r} \leq \bar{M}  \delta_{q+1}^{\frac12},
	\end{align}
	where $\bar{M}$ is same as in \eqref{v_diff_EN}.
\end{proposition}
\begin{proof}
	The iteration starts at 
$
		\mathfrak{v}_{0}=\boldsymbol{0}, \;\;  \mathring{N}_{0}=-\mathcal{R}\mathfrak{z}_0+\mathfrak{z}_0  \mathring{\otimes} \mathfrak{z}_0,
$ 
	where $\mathcal{R}$ denotes the reverse-divergence operator. It implies that \eqref{w_q_C0}-\eqref{w_q_EN} are satisfied at the level $q=0$ immediately. From the definition of $\mathfrak{z}_{q}$, we have 
	\begin{align*}
		\|\mathring{N}_{0}\|_{C^0_{t,x}} & \leq \|\mathfrak{z}_{0}\|^2_{C^0_{t,x}} + \|\mathfrak{z}_{0}\|_{C^0_{t,x}}  \leq \frac14\lambda^{\frac23}_{0} + \frac12\lambda^{\frac13}_{0}  \leq {\lambda^{\frac23}_{0}},
	\end{align*}
	and, we also obtain
	\begin{align*}
		\EN \mathring{N}_{0}\EN_{C^0,r} & \leq \EN\mathfrak{z}_{0}\EN^2_{C^0,2r} + \EN\mathfrak{z}_{0}\EN_{C^0,r} 
		\leq 2rL^2+rL\leq \delta_{1}.
	\end{align*}
	Hence \eqref{N_q_C0} and \eqref{N_q_EN} are satisfied at the level $q=0$. Let us now assume that there exists a pair $(\mathfrak{v}_{q},\mathring{N}_{q})$ which satisfies \eqref{w_q_C0}-\eqref{N_q_EN}. Our next aim is to show the existence of a pair $(\mathfrak{v}_{q+1},\mathring{N}_{q+1})$ which satisfies \eqref{w_q_C0}-\eqref{N_q_EN} at level $q+1$. 
	\vskip 2mm
	\noindent
	\textit{Mollification:} 
	Let us define a mollification of $\mathfrak{v}_{q},\mathring{N}_{q}$ and $\mathfrak{z}_q$ in space and time by convolution as follows:
	\begin{align*}
		\mathfrak{v}_\ell:=(\mathfrak{v}_{q}\ast_{x}\uppsi_{\ell})\ast_{t}\psi_{\ell}, \;\;\;   \mathring{N}_\ell&:=( \mathring{N}_{q}\ast_{x}\uppsi_{\ell})\ast_{t}\psi_{\ell},  \;\;\;   \text{ and }\;\;\; \mathfrak{z}_\ell :=(\mathfrak{z}_q\ast_{x}\uppsi_{\ell})\ast_{t}\psi_{\ell},
	\end{align*}
	where $\uppsi_{\ell}$ and 	$\psi_{\ell}$ are given by \eqref{Mol}.	By the definition, it is immediate to see that $\mathfrak{v}_{\ell}, {N}_{\ell}$ and $\mathfrak{z}_{\ell}$ are $\{\mathcal{F}_t\}_{t\in\R}$-adapted. From system \eqref{eqn_w_q}, we find that $(\mathfrak{v}_{\ell},\mathring{N}_{\ell})$ satisfy the following system:
	\begin{equation}\label{eqn_random_w_l}
		\left\{
		\begin{aligned}
			\partial_t\mathfrak{v}_\ell +\mathrm{div}\left((\mathfrak{v}_\ell+\mathfrak{z}_\ell) \otimes (\mathfrak{v}_\ell+\mathfrak{z}_\ell) \right) +\nabla \pi_\ell -\mathfrak{v}_{\ell} &= \diver ( {N}_{\ell}+N_{com1}),\\
			\mathrm{div}\;\mathfrak{v}_\ell&=0,
		\end{aligned}
		\right.
	\end{equation}
	where the commutator stress
	\begin{align}\label{N_com1}
		N_{com1}:= (\mathfrak{v}_\ell+\mathfrak{z}_\ell) \mathring{\otimes} (\mathfrak{v}_\ell+\mathfrak{z}_\ell) - \left(\left((\mathfrak{v}_q+\mathfrak{z}_q) \mathring{\otimes} (\mathfrak{v}_q+\mathfrak{z}_q)\right)\ast_{x}\uppsi_{\ell}\right)\ast_{t}\psi_{\ell},
	\end{align}
	and 
	\begin{align}\label{pi_l}
		\pi_{\ell} :=	(\pi_q\ast_{x}\uppsi_{\ell})\ast_{t}\psi_{\ell} -\frac13\left(|\mathfrak{v}_\ell+\mathfrak{z}_\ell|^2- (|\mathfrak{v}_q+\mathfrak{z}_q|^2\ast_{x}\uppsi_{\ell})\ast_{t}\psi_{\ell} \right).
	\end{align}
	\vskip 2mm
	\noindent
	\textit{New velocity iteration $\mathfrak{v}_{q+1}$:} Let us define 
	\begin{align*}
		\mathfrak{v}_{q+1}:=\mathfrak{v}_{\ell}+\omega_{q+1}=\mathfrak{v}_{q}-(\mathfrak{v}_{q}-\mathfrak{v}_{\ell})+ \omega_{q+1}^{(p)} + \omega_{q+1}^{(c)},
	\end{align*}
	where $\omega_{q+1}$ is given by \eqref{w_q+1_Curl}. It follows by the similar arguments as for \eqref{v_q+1_C0}, \eqref{v_q+1_C1} and \eqref{v_q+1_EN} that 
	\begin{align*}
		\|\mathfrak{v}_{q+1}\|_{C^0_{t,x}} \leq \lambda_{q+1}^{\frac13},\;\;\;	\|\mathfrak{v}_{q+1}\|_{C^1_{t,x}}  \leq \lambda_{q+1}^{\frac75}\delta_{q+1}^{\frac12}, \;\;\; \text{ and }\;\;\;	\EN\mathfrak{v}_{q+1}\EN_{C^0,2r} \leq 1 -\delta_{q+1}^{\frac12}, 
	\end{align*}
	for  $0<\beta<\frac{1}{200}$.
	\vskip 2mm
	\noindent
	\textit{New Reynolds stress $\mathring{N}_{q+1}$:} We infer from \eqref{eqn_random_w_l} and the system \eqref{eqn_w_q} at the level $q+1$ that
	\begin{align*}
		& \diver \mathring{N}_{q+1} - \nabla \pi_{q+1} 
		\nonumber\\& =   \underbrace{(\partial_t+(\mathfrak{v}_{\ell}+\mathfrak{z}_{\ell})\cdot\nabla)\omega^{(p)}_{q+1}}_{\diver N_{trans}} + \underbrace{(\omega_{q+1}\cdot\nabla)(\mathfrak{v}_{\ell}+\mathfrak{z}_{\ell})}_{\diver N_{Nash}}
		+\underbrace{\diver\big(\omega_{q+1}^{(p)}\otimes\omega_{q+1}^{(p)}+\mathring{N}_{\ell}\big)}_{\diver N_{osc}+\nabla \pi_{osc}}
		\nonumber\\& \quad + \underbrace{(\partial_t+(\mathfrak{v}_{\ell}+\mathfrak{z}_{\ell})\cdot\nabla)\omega^{(c)}_{q+1} + \diver (\omega_{q+1}^{(c)}\otimes\omega_{q+1}+\omega_{q+1}^{(p)}\otimes\omega_{q+1}^{(c)})}_{\diver N_{corr}+\nabla \pi_{corr}}
		\nonumber\\& \quad + \underbrace{\diver \big( \mathfrak{v}_{q+1}\otimes(\mathfrak{z}_{q+1}-\mathfrak{z}_{\ell})   + (\mathfrak{z}_{q+1}-\mathfrak{z}_{\ell})\otimes \mathfrak{v}_{q+1} + \mathfrak{z}_{q+1}\otimes\mathfrak{z}_{q+1} -\mathfrak{z}_{\ell}\otimes \mathfrak{z}_{\ell} +(\mathfrak{z}_{q+1}-\mathfrak{z}_{\ell})}_{\diver N_{com2}+\nabla \pi_{com2}} 
		\nonumber\\& \quad + \diver N_{com1} -\nabla \pi_{\ell}.
	\end{align*}
	Here $N_{com1}$ and $\pi_{\ell}$ are as defined in \eqref{N_com1} and \eqref{pi_l}, respectively, and by using the operator $\mathcal{R}$ discussed in Subsection \ref{IDO-R}, we define
	\begin{align*}
		N_{trans}&:= \mathcal{R}\bigg[(\partial_t+(\mathfrak{v}_{\ell}+\mathfrak{z}_{\ell})\cdot\nabla)\omega^{(p)}_{q+1}\bigg], \\
		N_{Nash} &:= \mathcal{R}\left[(\omega_{q+1}\cdot\nabla)(\mathfrak{v}_{\ell}+\mathfrak{z}_{\ell})\right], \\ 
		N_{osc}&:=  \sum_{j,j', \zeta+\zeta'\neq0} \mathcal{R} \bigg[\biggl\{\Wb_{\zeta}\otimes\Wb_{\zeta'}-\frac{\Wb_{\zeta}\cdot\Wb_{\zeta'}}{2} \Id \biggr\} \nabla\left(a_{\zeta}a_{\zeta'} \phi_{\zeta}\phi_{\zeta'}\right)\bigg],\\
		N_{corr}& :=  \mathcal{R}\bigg[(\partial_t+(\mathfrak{v}_{\ell}+\mathfrak{z}_{\ell})\cdot\nabla)\omega^{(c)}_{q+1}\bigg] + \omega_{q+1}^{(c)}\mathring{\otimes}\omega_{q+1}+\omega_{q+1}^{(p)}\mathring{\otimes}\omega_{q+1}^{(c)}, \\
		N_{com2}& :=   \mathfrak{v}_{q+1}\mathring{\otimes}(\mathfrak{z}_{q+1}-\mathfrak{z}_{\ell})   + (\mathfrak{z}_{q+1}-\mathfrak{z}_{\ell})\mathring{\otimes} \mathfrak{v}_{q+1} + \mathfrak{z}_{q+1}\mathring{\otimes}\mathfrak{z}_{q+1}
		- \z_{\ell}\mathring{\otimes}\z_{\ell}  +\mathcal{R}(\z_{q+1}-\z_{\ell}), \\
		\pi_{osc} &:=  c_{\ast}^{-1}\varrho + \frac12  \sum_{j,j', \zeta+\zeta'\neq0}a_{\zeta}a_{\zeta'} \phi_{\zeta}\phi_{\zeta'}( \Wb_{\zeta}\cdot\Wb_{\zeta'}), \\
		\pi_{corr}& :=  \frac{1}{3}\left(2\omega_{q+1}^{(p)}\cdot\omega_{q+1}^{(c)} +\left|\omega_{q+1}^{(c)}\right|^2\right), \\
		\pi_{com2}& := \frac13\big(\mathfrak{v}_{q+1}\cdot(\mathfrak{z}_{q+1}-\mathfrak{z}_{\ell})   + (\mathfrak{z}_{q+1}-\mathfrak{z}_{\ell})\cdot \mathfrak{z}_{q+1} + |\mathfrak{z}_{q+1}|^2-|\mathfrak{z}_{\ell}|^2)\big).
	\end{align*}
	Finally, we have 
	\begin{align*}
		\mathring{N}_{q+1} =   N_{trans}+N_{Nash}+N_{osc}+N_{corr}+N_{com1}+N_{com2},
	\end{align*}
	and 
	\begin{align*}
		\pi_{q+1}=\pi_{\ell} - \pi_{osc} - \pi_{corr} - \pi_{com2}.
	\end{align*}
	We infer by the similar arguments as for \eqref{R_tra_C0}-\eqref{R_tra_EN}, \eqref{R_Nash_C0}-\eqref{R_Nash_EN}, \eqref{R_osc_C0}-\eqref{R_osc_EN}, \eqref{R_corrr_C0}-\eqref{R_corrr_EN}, \eqref{R_com1_C0}-\eqref{R_com1_EN} and \eqref{R_com2_C0}-\eqref{R_com2_EN} that 
	\begin{align*}
		\|\mathring{N}_{q+1}\|_{C^0_{t,x}}  \leq \lambda_{q+1}^{\frac23} \;\;\; \text{ and }\;\;\;	\EN \mathring{N}_{q+1}\EN_{C^0,r} \leq \delta_{q+2},
	\end{align*}
	for $0<\beta<\frac{1}{200}$. This completes the proof.
\end{proof}

The following proposition can be proved by following the similar arguments used in the proof of Proposition \ref{BIP}.
\begin{proposition}[Bifurcating inductive proposition]\label{BIP-SEE}
	Let $(\mathfrak{v}_{q},\mathring{N}_{q})$ be as in the statement of Proposition \ref{Iterations-SEE}. For any interval $\mathcal{I}\subset \R$ with $|\mathcal{I}|\geq 3m_q$ (where $m_q$ is defined in \eqref{m_q}), we can produce a first pair $(\mathfrak{v}_{q+1},\mathring{N}_{q+1})$ and a second pair $(\widetilde{\mathfrak{v}}_{q+1},\widetilde{\mathring{N}}_{q+1})$ which share the same initial data, satisfy the same conclusions of Proposition \ref{Iterations-SEE} and additionally
	\begin{align*}
		\EN\mathfrak{v}_{q+1}-\widetilde{\mathfrak{v}}_{q+1}\EN_{\L^{2}_{x},2} \geq \delta_{q+1}^{\frac12}, \;\; \supp_{t}(\mathfrak{v}_{q+1}-\widetilde{\mathfrak{v}}_{q+1})\subset \mathcal{I}.
	\end{align*}
Moreover, if we are given two pairs $(\mathfrak{v}_{q}, {\mathring{N}}_{q})$ and $(\widetilde{\mathfrak{v}}_{q},\widetilde{\mathring{N}}_{q})$ satisfying \eqref{w_q_C0}-\eqref{N_q_EN} and 
\begin{align*}
	\supp_{t}(\mathfrak{v}_{q}-\widetilde{\mathfrak{v}}_{q},\mathring{N}_{q}-\widetilde{\mathring{N}}_{q}) \subset \mathcal{J},
\end{align*}
for some interval $\mathcal{J}\subset\R$, we can exhibit corrected counterparts $(\mathfrak{v}_{q+1}, {\mathring{N}}_{q+1})$ and $(\widetilde{\mathfrak{v}}_{q+1},\widetilde{\mathring{N}}_{q+1})$ again satisfying the same conclusion of Proposition \ref{Iterations-SEE} together with the following control on the support of their difference:
\begin{align*}
	\supp_{t}(\mathfrak{v}_{q+1}-\widetilde{\mathfrak{v}}_{q+1}, \mathring{N}_{q+1}-\widetilde{\mathring{N}}_{q+1}) \subset \mathcal{J}+\lambda_{q}^{-\frac85}.
\end{align*}
\end{proposition}

\subsection{Main results}
We have just established the iteration proposition. In view of Propositions \ref{Iterations-SEE} and \ref{BIP-SEE}, we provide the following result for the system \eqref{eqn_stochatic_u-SEE} which can be proved by following the similar arguments used in the proof of Theorem \ref{MR-SHNSE}.

\begin{theorem}\label{MR-SEE}
	Suppose that $\Tr\left((-\Delta)^{\frac52}GG^{\ast}\right)<\infty$ and let $r>1$ be fixed. Then, for any  $\beta\in(0,\frac{1}{200})$, there exist infinitely many $\{\mathcal{F}_t\}_{t\in\R}$-adapted process $\u(\cdot)$ which belongs to $C(\R;C^{\vartheta})$ $\mathbb{P}$-a.s. for $\vartheta\in(0,\frac57\beta)$ and is an analytically weak solution to \eqref{eqn_stochatic_u-SEE} in the sense of Definition \ref{AWS}. Moreover, the solutions satisfies 
	\begin{align}\label{u_EN-SEE}
		\EN\u\EN_{C^{\vartheta},2r}<\infty.
	\end{align}
\end{theorem}

Our next result is the existence of stationary solutions which are constructed as limits of ergodic averages of solutions from Theorem \ref{MR-SEE}. This in particular implies their non-uniqueness. The proof of the following theorem can be done by following the similar arguments used in the proof of Theorem \ref{MR-SS-SHNSE}.
\begin{theorem}\label{MR-SS-SEE}
	Let $\u$ be a solution of system \eqref{eqn_stochatic_u-SEE} obtained in Theorem \ref{MR-SEE} and satisfying \eqref{u_EN-SEE}. Then there exists a sequence $T_{n}\to\infty$ and a stationary solution $\left((\widetilde{\Omega},\widetilde{\mathcal{F}},\widetilde{\mathbb{P}}),\widetilde{\u},\widetilde{\W}\right)$ to \eqref{eqn_stochatic_u-SEE} such that 
	\begin{align*}
		\frac{1}{T_{n}}\int_{0}^{T_{n}}\mathcal{L}\left[S_{t}(\u,\W)\right]\d t\to \mathcal{L}[\widetilde{\u},\widetilde{\W}],
	\end{align*}
	weakly in the sense of probability measures on $\mathcal{T}$ as $n\to\infty$, where $\mathcal{T}$ is same as defined in Section \ref{sec6}.
\end{theorem}

	\begin{appendix}
	\renewcommand{\thesection}{\Alph{section}}
	\numberwithin{equation}{section}
	\section{Proof of the key results}\label{SupportingResults}\setcounter{equation}{0}
	The purpose of this section is to provide the proof of some key results which are used to prove the main results of this article. In particular, we will provide the proofs of Lemma \ref{Lemma-z_diff_EN} and Proposition \ref{BIP}, respectively.

	\subsection{Proof of Lemma \ref{Lemma-z_diff_EN}}\label{A.2}
By the definition of $\z_{q}$, we deduce for any $p\geq2$
	\begin{align}\label{I+II+III}
		&	\EN\z_{q+1}-\z_{q}\EN_{C^0,p}^p
		=\sup_{t\in\R}\mathbb{E}\left[\sup_{t\leq s\leq t+1}\|\z_{q+1}(s)-\z_{q}(s)\|^p_{C^0_{x}}\right]
		\nonumber\\ & =\sup_{t\in\R}\mathbb{E}\bigg[\sup_{t\leq s\leq t+1}\|\chi_{q+1}\left(\|\wi\z_{q+1}(s)\|_{C^0_{x}}\right)\wi\chi_{q+1}\left(\|\wi\z_{q+1}(s)\|_{C^1_{x}}\right)\wi\z_{q+1}(s) -\chi_q\left(\|\wi\z_q(s)\|_{C^0_{x}}\right)\wi\chi_q\left(\|\wi\z_q(s)\|_{C^1_{x}}\right)\wi\z_q(s)\|^p_{C^0_{x}}\bigg]
		\nonumber\\ & \leq \sup_{t\in\R}\mathbb{E}\left[\sup_{t\leq s\leq t+1}\|\wi\z_{q+1}(s)-\wi\z_{q}(s)\|^p_{C^0_{x}} \cdot \chi_{q+1}^p (\|\wi\z_{q+1}(s)\|_{C^0_{x}}) \cdot \wi\chi_{q+1}^p (\|\wi\z_{q+1}(s)\|_{C^1_{x}} )\right]
		\nonumber\\ & \quad + \sup_{t\in\R}\mathbb{E}\left[\sup_{t\leq s\leq t+1}\|\wi\z_{q}(s)\|^p_{C^0_{x}}\cdot\left|\chi_{q+1} (\|\wi\z_{q+1}(s)\|_{C^0_{x}} ) - \chi_{q} (\|\wi\z_{q}(s)\|_{C^0_{x}} )\right|^p\cdot \wi\chi_{q}^p (\|\wi\z_{q}(s)\|_{C^1_{x}} )\right]
		\nonumber\\ & \quad + \sup_{t\in\R}\mathbb{E}\left[\sup_{t\leq s\leq t+1}\|\wi\z_{q}(s)\|^p_{C^0_{x}}\cdot\chi_{q+1}^p (\|\wi\z_{q+1}(s)\|_{C^0_{x}} )\cdot\left|\wi\chi_{q+1} (\|\wi\z_{q+1}(s)\|_{C^1_{x}} )- \wi\chi_{q} (\|\wi\z_{q}(s)\|_{C^1_{x}} )\right|^p\right]
		\nonumber \\ & := \mathbf{I}+\mathbf{I\!I}+\mathbf{I\!I\!I}.
	\end{align}
	\vskip 2mm
	\noindent
	\textit{Estimate for $\mathbf{I}$.} By the definition of $\wi\z_{q}$, we have
	\begin{align*}
		\wi\z_{q+1}(t,x)-\wi\z_{q}(t,x)=\sum_{\lambda_{q}^{\frac13}<|k|\leq\lambda_{q+1}^{\frac13}}e^{ik\cdot x}\hat{\z}(t,k),
	\end{align*}
	where $\hat{\z}$ is the Fourier transform of $\z$, and $k\in\mathbb{Z}^3$. Then using H\"older's inequality, we deduce for any $t\in\R$ that
	 \begin{align}\label{A.8}
		& \|\wi\z_{q+1}(t)-\wi\z_{q}(t)\|_{\L^{\infty}} 
        \nonumber\\ & \leq	\sum_{\lambda_{q}^{\frac13}<|k|\leq\lambda_{q+1}^{\frac13}}\left|\hat{\z}(t,k)\right|
		  = \sum_{\lambda_{q}^{\frac13}<|k|\leq\lambda_{q+1}^{\frac13}}(1+|k|^2)^{\frac54}\left|\hat{\z}(t,k)\right| (1+|k|^2)^{-\frac54}
		\nonumber\\ & \leq \left(\sum_{\lambda_{q}^{\frac13}<|k|\leq\lambda_{q+1}^{\frac13}}(1+|k|^2)^{\frac52}\left|\hat{\z}(t,k)\right|^2\right)^{\frac12}\cdot\left(\sum_{\lambda_{q}^{\frac13}<|k|\leq\lambda_{q+1}^{\frac13}}  (1+|k|^2)^{-\frac52}\right)^{\frac12}
		\lesssim \lambda_{q}^{-\frac13 } \|\z(t)\|_{\H^{\frac52}}
	\end{align}
	Taking expectation of \eqref{A.8} and using \eqref{z_q_EN1}, we obtain 
	 \begin{align}\label{I}
		\mathbf{I} & \leq  \sup_{t\in\R}\mathbb{E}\left[\sup_{t\leq s\leq t+1} \|\wi\z_{q+1}(s)-\wi\z_{q}(s)\|^p_{C^0_{x}} \right]
		  \leq  \lambda_{q}^{-\frac13 p} \sup_{t\in\R}\mathbb{E}\left[\sup_{t\leq s\leq t+1} \|\z(t)\|_{\H^{\frac52}}^p \right]
		  \leq  \lambda_{q}^{-\frac13 p} \EN\z\EN_{\H^{\frac52},p}^{p}
		  \leq (\lambda_{q}^{-\frac13}\sqrt{p}L)^p.
	\end{align}
	\textit{Estimate for $\mathbf{I\!I}$.} By the definitions of $\chi_{q}$ and $\chi_{q+1}$ (see Subsection \ref{Itra+Pro}), we deduce that
	\begin{align}\label{A.10}
		 \left|\chi_{q+1} (\|\wi\z_{q+1}(s)\|_{C^0_{x}} )- \chi_{q} (\|\wi\z_{q}(s)\|_{C^0_{x}} )\right|
        \leq \mathds{1}_{\left\{\|\wi\z_{q+1}(s)\|_{C^0_x}>\frac{1}{7\cdot4\cdot2\cdot2}\lambda_{q+1}^{\frac13-\varepsilon}\right\}} + \mathds{1}_{\left\{\|\wi\z_{q}(s)\|_{C^0_x}>\frac{1}{7\cdot4\cdot2\cdot2}\lambda_{q}^{\frac13-\varepsilon}\right\}}.
	\end{align}
	Using \eqref{A.10} and H\"older's inequality, we obtain
	\begin{align}\label{A.11}
	&	\mathbf{I\!I} 
		 \leq \sup_{t\in\R}\mathbb{E}\left[\sup_{t\leq s\leq t+1}\|\wi\z_{q}(s)\|^p_{C^0_{x}}\cdot \mathds{1}_{\left\{\|\wi\z_{q+1}(s)\|_{C^0_x}>\frac{1}{7\cdot4\cdot2\cdot2}\lambda_{q+1}^{\frac13-\varepsilon}\right\}}\right]
		+ \sup_{t\in\R}\mathbb{E}\left[\sup_{t\leq s\leq t+1}\|\wi\z_{q}(s)\|^p_{C^0_{x}}\cdot   \mathds{1}_{\left\{\|\wi\z_{q}(s)\|_{C^0_x}>\frac{1}{7\cdot4\cdot2\cdot2}\lambda_{q}^{\frac13-\varepsilon}\right\}}\right]
		\nonumber\\ &   \leq \sup_{t\in\R}\left(\mathbb{E}\left[\sup_{t\leq s\leq t+1}\|\wi\z_{q}(s)\|^{2p}_{C^0_{x}}\right]\right)^{\frac12} \left[  \left(\mathbb{P}\left\{\|\wi\z_{q+1}\|_{C^0_{t,x}}>\frac{1}{112}\lambda_{q+1}^{\frac13-\varepsilon}\right\}\right)^{\frac12}
		 +   \left(\mathbb{P}\left\{\|\wi\z_{q}(s)\|_{C^0_{t,x}}>\frac{1}{112}\lambda_{q}^{\frac13-\varepsilon}\right\}\right)^{\frac12}\right].
	\end{align}
	By Chebyshev's inequality and \eqref{z_q_EN1}, we have 
	\begin{align}\label{A.12}
		\sup_{t\in\R}\left(\mathbb{P}\left\{\|\wi\z_{q}(s)\|_{C^0_{t,x}}>\frac{1}{112}\lambda_{q}^{\frac13-\varepsilon}\right\}\right)^{\frac12} 
		   \leq (112)^p \lambda_{q}^{-(\frac13-\varepsilon)p}  (\sqrt{2p}L)^p,
	\end{align}
	and 
	\begin{align}\label{A.13}
	\sup_{t\in\R}\left(\mathbb{P}\left\{\|\wi\z_{q+1}(s)\|_{C^0_{t,x}}>\frac{1}{112}\lambda_{q+1}^{\frac13-\varepsilon}\right\}\right)^{\frac12} 
		  \leq (112)^p \lambda_{q+1}^{-(\frac13-\varepsilon)p}  (\sqrt{2p}L)^p.
	\end{align}
	Therefore, using \eqref{A.11}-\eqref{A.13} and \eqref{z_q_EN1}, we write 
	\begin{align}\label{II}
		\mathbf{I\!I} &\leq \EN\wi\z_{q}\EN_{C^0,2p}^p (112\sqrt{2p}L)^p(\lambda_{q}^{-(\frac13-\varepsilon)p} +\lambda_{q+1}^{-(\frac13-\varepsilon)p} )
		 \lesssim (112pL^2)^{p}\lambda_{q}^{-(\frac13-\varepsilon)p}.
	\end{align}
	\textit{Estimate for $\mathbf{I\!I\!I}$.} By the definitions of $\wi\chi_{q}$ and $\wi\chi_{q+1}$ (see Subsection \ref{Itra+Pro}), we deduce that
	\begin{align}\label{A.15}
		 \left|\wi\chi_{q+1} (\|\wi\z_{q+1}(s)\|_{C^1_{x}} )- \wi\chi_{q} (\|\wi\z_{q}(s)\|_{C^1_{x}} )\right|
        \leq \mathds{1}_{\left\{\|\wi\z_{q+1}(s)\|_{C^1_x}>\lambda_{q+1}^{\frac23-\varepsilon}\right\}} + \mathds{1}_{\left\{\|\wi\z_{q}(s)\|_{C^1_x}>\lambda_{q}^{\frac23-\varepsilon}\right\}}.
	\end{align}
	Using \eqref{A.15} and H\"older's inequality, we obtain
	\begin{align}\label{A.16}
		\mathbf{I\!I\!I} 
		& \leq \sup_{t\in\R}\mathbb{E}\left[\sup_{t\leq s\leq t+1}\|\wi\z_{q}(s)\|^p_{C^0_{x}}\cdot \mathds{1}_{\left\{\|\wi\z_{q+1}(s)\|_{C^1_x}>\lambda_{q+1}^{\frac23-\varepsilon}\right\}}\right]
		+ \sup_{t\in\R}\mathbb{E}\left[\sup_{t\leq s\leq t+1}\|\wi\z_{q}(s)\|^p_{C^0_{x}}\cdot   \mathds{1}_{\left\{\|\wi\z_{q}(s)\|_{C^1_x}>\lambda_{q}^{\frac23-\varepsilon}\right\}}\right]
		\nonumber\\ &  \leq \sup_{t\in\R}\left(\mathbb{E}\left[\sup_{t\leq s\leq t+1}\|\wi\z_{q}(s)\|^{2p}_{C^0_{x}}\right]\right)^{\frac12}
        \left[ \left(\mathbb{P}\left\{\|\wi\z_{q+1}(s)\|_{C^0_tC^1_x}>\lambda_{q+1}^{\frac23-\varepsilon}\right\}\right)^{\frac12}
		 +  \left(\mathbb{P}\left\{\|\wi\z_{q}(s)\|_{C^0_tC^1_x}>\lambda_{q}^{\frac23-\varepsilon}\right\}\right)^{\frac12}\right].
	\end{align}
	By Chebyshev's inequality and \eqref{z_q_EN2}, we have 
	\begin{align}\label{A.17}
		\sup_{t\in\R}\left(\mathbb{P}\left\{\|\wi\z_{q}(s)\|_{C^0_tC^1_x}>\lambda_{q}^{\frac23-\varepsilon}\right\}\right)^{\frac12} & \leq \lambda_{q}^{-(\frac23-\varepsilon)p} \sup_{t\in\R}\left(\mathbb{E}\left[\sup_{t\leq s\leq t+1}\|\wi\z_{q}(s)\|^{2p}_{C^1_{x}}\right]\right)^{\frac12}
		\leq \lambda_{q}^{-(\frac23-\varepsilon)p} \lambda_{q}^{\frac13p} (\sqrt{2p}L)^p,
	\end{align}
	and 
	\begin{align}\label{A.18}
		\sup_{t\in\R}\left(\mathbb{P}\left\{\|\wi\z_{q+1}(s)\|_{C^0_tC^1_x}>\lambda_{q+1}^{\frac23-\varepsilon}\right\}\right)^{\frac12} & \leq \lambda_{q+1}^{-(\frac23-\varepsilon)p} \sup_{t\in\R}\left(\mathbb{E}\left[\sup_{t\leq s\leq t+1}\|\wi\z_{q+1}(s)\|^{2p}_{C^1_{x}}\right]\right)^{\frac12}
		 \leq \lambda_{q+1}^{-(\frac23-\varepsilon)p} \lambda_{q+1}^{\frac13p} (\sqrt{2p}L)^p.
	\end{align}
	Therefore, using \eqref{A.16}-\eqref{A.18} and \eqref{z_q_EN1}, we write 
	\begin{align}\label{III}
		\mathbf{I\!I\!I} &\leq \EN\wi\z_{q}\EN_{C^0,2p}^p (\sqrt{2p}L)^p(\lambda_{q}^{-(\frac13-\varepsilon)p} +\lambda_{q+1}^{-(\frac13-\varepsilon)p})
		 \lesssim (2pL^2)^{p}\lambda_{q}^{-(\frac13-\varepsilon)p}.
	\end{align}
	Hence, in view of \eqref{I+II+III}, \eqref{I}, \eqref{II} and \eqref{III}, we have 
	\begin{align*}
		\EN\z_{q+1}-\z_{q}\EN_{C^0,p}& \lesssim (\mathbf{I})^{1/p}+(\mathbf{I\!I})^{1/p}+(\mathbf{I\!I\!I})^{1/p}
		 \lesssim \lambda_{q}^{-\frac13}\sqrt{p}L+ pL^2\lambda_{q}^{-(\frac13-\varepsilon)}
	 \lesssim  pL^2(\lambda_{q}^{-\frac13}+\lambda_{q}^{-(\frac13-\varepsilon)})
		  \lesssim	pL^2\lambda_{q}^{-(\frac13-\varepsilon)}.
	\end{align*}
 This completes the proof.

		\subsection{Proof of Proposition \ref{BIP}}\label{A.3}
	Let us choose and fix an interval $\mathcal{I}\subset [k,k+1]$ with $|\mathcal{I}|\geq 3m_q$. Then we can always find $j_0$ such that 
	\begin{align*}
		\supp_{t} (\eta_{j_0,k}(\cdot))=\supp_{t}(\eta(m_q^{-1}(\cdot-k)-j_0))\subset \mathcal{I}.
	\end{align*}
	For $j=j_0$, we replace $\Gamma_{\zeta}^{(j)}$ in $\omega_{q+1}$ by $\wi\Gamma_{\zeta}^{(j)}=-\Gamma_{\zeta}^{(j)}$. In other words, we replace $a_{\zeta}$ by $\wi a_{\zeta}=-a_{\zeta}$. Otherwise, we keep the same $\Gamma_{\zeta}^{(j)}$. Note that $\wi\Gamma_{\zeta}^{(j)}$ still solves \eqref{B.5} and hence $\wi a_{\zeta}$ satisfies \eqref{B1}. Let us denote the velocity perturbation, principle part of velocity perturbation, corrector part of velocity perturbation, new velocity and new stress tensor corresponding to $\wi a_{\zeta}$ by $\wi\omega_{q+1}$, $\wi\omega_{q+1}^{(p)}$, $\wi\omega_{q+1}^{(c)}$, $\wi\v_{q+1}$ and $\widetilde{\mathring{R}}_{q+1}$.  Observe that $\wi a_{\zeta}$ also satisfies the same bound as $a_{\zeta}$ in \eqref{ak_N}. As a result, the pair $(\wi\v_{q+1},\widetilde{\mathring{R}}_{q+1})$ satisfies \eqref{v_q_C0}-\eqref{R_q_EN} at level $q+1$ and \eqref{v_diff_EN} as desired. 
	
	On the other hand, by the construction, $\wi\omega_{q+1}$ differs from $\omega_{q+1}$ on the support of $\supp_{t} (\eta_{j_0,k}(\cdot))$. Therefore, we can easily see
	\begin{align*}
		\supp_{t}(\v_{q+1}-\wi\v_{q+1})= \supp_{t}(\omega_{q+1}-\wi\omega_{q+1})\subset \mathcal{I}.
	\end{align*}
	We  have
	\begin{align}
		& \omega_{q+1}^{(p)}-\wi\omega_{q+1}^{(p)}
		= 2 \sum_{\zeta\in\Lambda_{j_0}} c_{\ast}^{-\frac12} \varrho^{\frac12}\cdot\eta_{j_0,k} \cdot \Gamma_{\zeta}^{(j_0)}\left(\Id- c_{\ast}\frac{ \mathring{R}_{\ell}}{\varrho}\right)\cdot B_{\zeta}e^{i\lambda_{q+1}\zeta\cdot\Phi_{k,j_0}},
	\end{align}
	Consider
	\begin{align}
		& (\omega_{q+1}^{(p)}-\wi\omega_{q+1}^{(p)})\otimes(\omega_{q+1}^{(p)}-\wi\omega_{q+1}^{(p)})
		\nonumber\\ & = 4 \sum_{\zeta} c_{\ast}^{-1} \varrho\cdot\eta^2_{j_0,k} \cdot \left[\Gamma_{\zeta}^{(j_0)}\left(\Id- c_{\ast}\frac{ \mathring{R}_{\ell}}{\varrho}\right)\right]^2\cdot B_{\zeta}\otimes B_{\zeta^{\prime}}
		+4 \sum_{\zeta+\zeta^{\prime}\neq 0} a_{\zeta}a_{\zeta^{\prime}}\Wb_{\zeta}(\Phi_{k,j_0})\otimes \Wb_{\zeta'}(\Phi_{k,j_0})
		\nonumber\\ & = 4  \eta^2_{j_0,k}\left(c_{\ast}^{-1} \varrho\Id-  \mathring{R}_{\ell}\right)
		+ \sum_{\zeta+\zeta^{\prime}\neq 0} a_{\zeta}a_{\zeta^{\prime}}\Wb_{\zeta}(\Phi_{k,j_0})\otimes \Wb_{\zeta'}(\Phi_{k,j_0}),
	\end{align}
	which gives
	\begin{align}\label{omega-diff}
		& |\omega_{q+1}^{(p)}-\wi\omega_{q+1}^{(p)}|^2
		= 12  \eta^2_{j_0,k}c_{\ast}^{-1} \varrho\Id
		+ \sum_{\zeta+\zeta^{\prime}\neq 0} \Tr\left(a_{\zeta}a_{\zeta^{\prime}}\Wb_{\zeta}(\Phi_{k,j_0})\otimes \Wb_{\zeta'}(\Phi_{k,j_0})\right).
	\end{align}

	In view of Lemma \ref{SPL}, we estimate
	\begin{align}\label{4.23}
		&\left|\mathbb{E}\left[ \sum_{ \zeta+\zeta'\neq0}  \int_{\T3} \Tr\left(a_{\zeta}a_{\zeta'} \Wb_{\zeta}(\Phi_{k,j_0})\otimes \Wb_{\zeta'}(\Phi_{k,j_0})\right)\d x\right]\right|
		\leq \mathbb{E}\left[ \sum_{ \zeta+\zeta'\neq0} \left| \int_{\T3} \Tr\left(a_{\zeta}a_{\zeta'} \Wb_{\zeta}(\Phi_{k,j_0})\otimes \Wb_{\zeta'}(\Phi_{k,j_0})\right)\d x\right|\right]
		\nonumber\\ & \lesssim \mathbb{E}\left[\sum_{\zeta+\zeta'\neq0} \left| \int_{\T3} a_{\zeta}a_{\zeta'} B_{\zeta}\otimes B_{-\zeta} e^{i\lambda_{q+1}(\zeta+\zeta')\cdot\Phi_{k,j_0}} \d x\right|\right]
		\nonumber\\ & \lesssim \mathbb{E}\left[\sum_{\zeta+\zeta'\neq0} \frac{\|a_{\zeta}a_{\zeta'}\|_{C^0_{t}C^m_{x}} + \|a_{\zeta}a_{\zeta'}\|_{C^0_{t,x}}\|\nabla\Phi_{k,j_0}\|_{C^0_{t}C^m_{x}}}{\lambda_{q+1}^{m}}   \right]
		\nonumber\\ &  \lesssim \mathbb{E}\left[ \frac{m_q^{-m}\lambda_q^6 + \lambda_q^6 \lambda_q^{\frac{8}{5}(m-1)} }{\lambda_{q+1}^{m}}   \right]
		\lesssim \lambda_{q}^{ \frac{81}{16}m + 6 -\frac{23}{4} m }  \leq  \delta_{q+1},
	\end{align}
	where we have used    \eqref{a_kk'_N}, \eqref{B.10} and $m> 33$,  and $a$ is chosen sufficiently large to absorb the constant. Making use of \eqref{omega-diff}, \eqref{varrho}, \eqref{R_q_EN} and \eqref{4.23}, we estimate
	\begin{align}\label{4.24}
		\EN\omega_{q+1}^{(p)}-\wi\omega_{q+1}^{(p)}\EN_{\L^{2}_{x},2}^2 \geq  12 (2\pi)^3 c_{\ast}^{-1} \delta_{q+1} - \delta_{q+1} \geq 4\delta_{q+1}.
	\end{align}
	Therefore, using \eqref{4.24} and \eqref{w_c_EN}, we obtain
	\begin{align}
		\EN\v_{q+1}-\wi\v_{q+1}\EN_{\L^{2}_{x},2}  & = \EN\omega_{q+1}-\wi\omega_{q+1}\EN_{\L^{2}_{x},2}
		\nonumber\\ & \geq \EN\omega_{q+1}^{(p)}-\wi\omega_{q+1}^{(p)}\EN_{\L^{2}_{x},2} - (2\pi)^{\frac32}(\EN\omega_{q+1}^{(c)}\EN_{C^0,2}+ \EN\wi\omega_{q+1}^{(c)}\EN_{C^0,2})
		 \geq 2\delta_{q+1}^{\frac12} -  \delta_{q+1}^{\frac12} = \delta_{q+1}^{\frac12},
	\end{align}
	where $a$ is chosen sufficiently large to absorb the constant in the inequality $$\EN\omega_{q+1}^{(c)}\EN_{C^0,2r} + \EN\wi \omega_{q+1}^{(c)}\EN_{C^0,2r}  \lesssim     2 \lambda_{q}^{-\frac{11}{16}-\frac{27}{8}\beta} \delta_{q+1}^{\frac{1}{2}} \leq \frac{1}{(2\pi)^{\frac32}}\delta_{q+1}^{\frac12}.$$
	Lastly, we suppose that a pair $(\wi\v_{q},\widetilde{\mathring{R}}_{q})$ (satisfying system \eqref{eqn_v_q}) satisfies \eqref{v_q_C0}-\eqref{R_q_EN} and 
	\begin{align*}
		\supp_{t}(\v_{q}-\wi\v_{q},\mathring{R}_{q}-\widetilde{\mathring{R}}_{q}) \subset \mathcal{J},
	\end{align*}
	for some time interval $\mathcal{J}$. Proceed to construct the regularized flow, $\widetilde{\mathring{R}}_{\ell}$ as we did for $\mathring{R}_{\ell}$ and note that they differ only in $\mathcal{J}+\ell=\mathcal{J}+\lambda_{q}^{-\frac85}$. As a result, $\omega_{q+1}$ differ from $\wi\omega_{q+1}$ only in $\mathcal{J}+\lambda_{q}^{-\frac85}$ and hence the pairs $( \v_{q+1}, {\mathring{R}}_{q+1})$ and $(\wi\v_{q+1},\widetilde{\mathring{R}}_{q+1})$ satisfy 
	\begin{align*}
		\supp_{t}(\v_{q+1}-\wi\v_{q+1}, \mathring{R}_{q+1}-\widetilde{\mathring{R}}_{q+1}) \subset \mathcal{J}+\lambda_{q}^{-\frac85}.
	\end{align*}
	This completes the proof.

	\section{Beltrami waves}\label{Beltrami}\setcounter{equation}{0}
	In this section, we recall the Beltrami waves from \cite[Section 5.4]{Buckmaster+Vicol_2019_Notes} which is adapted to the convex integration technique in Proposition \ref{Iterations}. Note that the construction discussed below is independent of sample points $\omega$, that is, it is purely deterministic. First of all we recall the definition of Beltrami waves.
	
	Given $\zeta\in\mathbb{S}^2\cap\mathbb{Q}^3$, let $A_{\zeta}\in \mathbb{S}^2\cap\mathbb{Q}^3$ satisfy $A_{\zeta}\cdot\zeta=0$ and $A_{-\zeta}=A_{\zeta}$. Furthermore, we also define a complex vector
	\begin{align}\label{Bk}
		B_{\zeta} :=\frac{1}{\sqrt{2}} (A_{\zeta} + i\zeta\times A_{\zeta}) \in \mathbb{C}^3.
	\end{align}
 It is immediate from the definition of $B_{\zeta}$ that it satisfies
 \begin{align*}
 	|B_{\zeta}|=1,\;\; B_{\zeta}\cdot\zeta=0, \;\; i\zeta\times B_{\zeta}=B_{\zeta},\;\; B_{-\zeta}=\overline{B_{\zeta}}.
 \end{align*}
With the above preparation, we infer that for any $\lambda\in\mathbb{Z}$, such that $\lambda\zeta\in\mathbb{Z}^3$, the function
\begin{align*}
	\Wb_{\zeta}:=\Wb_{\zeta,\lambda}(x):= B_{\zeta}e^{i\lambda\zeta\cdot x},
\end{align*}
is $\T3$-periodic, divergence-free, and is an eigenfunction of the curl operator with eigenvalue $\lambda$. That is, $\Wb_{\zeta}$ is a complex Beltrami plane wave. Next, we discuss a useful property of linear combinations of complex Beltrami plane waves.
\begin{lemma}[{\cite[Proposition 5.5]{Buckmaster+Vicol_2019_Notes}}]\label{Beltrami_P}
	Let $\Lambda$ be a given finite subset of $\mathbb{S}^2\cap\mathbb{Q}^3$ such that $-\Lambda=\Lambda$, and let $\lambda\in\mathbb{Z}$ be such that $\lambda\Lambda\subset\mathbb{Z}$. Then for any choice of coefficients $a_{\zeta}\in\mathbb{C}$ with $\overline{a_{\zeta}}=a_{-\zeta}$ the vector field
	\begin{align}\label{Beltrami1}
		\Wb(x) = \sum_{\zeta\in\Lambda} a_\zeta B_\zeta e^{i\lambda\zeta\cdot x},
	\end{align}
	is a real-valued, divergence-free Beltrami vector field $\curl\Wb=\lambda\Wb$, and satisfies stationary Euler equations
	\begin{align}\label{SEE}
		\diver(\Wb\otimes \Wb) = \nabla \frac{\vert \Wb\vert^2}{2}.
	\end{align}
	Furthermore, we also have
	\begin{align}\label{Average}
	\langle \Wb\otimes \Wb \rangle :=  \fint_{\T3} \Wb (x) \otimes \Wb (x) \; \d x = \frac{1}{2} \sum_{\zeta\in\Lambda} \vert a_\zeta\vert^2 \left(\Id - \zeta \otimes \zeta\right).
	\end{align}
\end{lemma}
\begin{remark}
	The key point of the construction is that the abundance of Beltrami flows allows to find several flows $v$ such that
	\begin{align*}
		\langle v \otimes v \rangle(t)  := \fint_{\T3} (v \otimes v)(t,x) dx = R,
	\end{align*}
	for some prescribed symmetric matrix $R$. It is true that we must choose these flows carefully such that they rely smoothly on the matrix $R$, at least when $R$ belongs to a neighborhood of the identity matrix. In view of \eqref{Average}, such selection is possible (see Lemma \ref{GL} below).
\end{remark}

The following lemma can be found in \cite{Buckmaster+Vicol_2019_Notes} (see \cite[Proposition 5.6]{Buckmaster+Vicol_2019_Notes}).
\begin{lemma}[Geometric Lemma]\label{GL}
	There exists a constant (sufficient small) $0<c_{\ast}\leq1$ with the following property. Let $B_{c_{\ast}}(\Id)$ denote the closed ball of symmetric $3\times3$ matrices, centered at $\Id$, of radius $c_\ast$. Then, there exists pairwise disjoint subsets
	\begin{align*}
		\Lambda_{\alpha}\subset \mathbb{S}^2\cap\mathbb{Q}^3, \;\;\; \alpha\in\{0,1\},
	\end{align*}
and smooth positive functions
\begin{align*}
	\Gamma_{\zeta}^{(\alpha)}\in C^{\infty}\left(B_{c_{\ast}}(\Id)\right),\;\;\;\; \alpha\in\{0,1\}, \zeta\in\Lambda_{\alpha},
\end{align*}
such that the following hold: 

For every $\zeta\in\Lambda_{\alpha}$, we have $-\zeta\in\Lambda_{\alpha}$ and $\Gamma_{\zeta}^{(\alpha)}=\Gamma_{-\zeta}^{(\alpha)}$. For each $R\in B_{c_{\ast}}(\Id)$, we have the identity
\begin{align}\label{B.5}
	R=\frac12 \sum_{\zeta\in\Lambda_{\alpha}} \left[\Gamma_{\zeta}^{(\alpha)}(R)\right]^2(\Id-\zeta\otimes\zeta).
\end{align}
\end{lemma}
\begin{remark}
	It has been noticed (see \cite{Buckmaster+Vicol_2019_Notes}) that it is sufficient to consider index sets $\Lambda_0$ and $\Lambda_1$ in Lemma \ref{GL} to have 12 elements. Moreover, by abuse of notation, for $j\in\mathbb{Z}$ we denote $\Lambda_{j}=\Lambda_{j\mod 2}$. Hence in Subsection \ref{V+P}, we also write $\Gamma_{\zeta}^{(\alpha)}$ as $\Gamma_{\zeta}^{(j)}$. Also it is convenient to denote $M$ geometric constant such that 
	\begin{align}\label{C^n_M}
		\sum_{\zeta\in\Lambda_{\alpha}} \|\Gamma_{\zeta}^{(\alpha)}\|_{C^n(B_{c_{\ast}}(\Id))}\leq M,
	\end{align}
holds for $n$ large enough, $\alpha\in\{0,1\}$ and $\zeta\in\Lambda_{\alpha}$. This parameter is universal.
\end{remark}

\section{Estimates for transport equations}\setcounter{equation}{0}\label{Transport}
\numberwithin{equation}{section}
In this section, we present a detailed estimates of the solutions to the transport equation \eqref{Phi_kj} which have been frequently used in this work. Let us consider the following transport equation on $[t_0,T]$, $t_0\geq0$:
\begin{equation*}
	\left\{	\begin{aligned}
			(\partial_t + \bm\cdot\nabla)\f& =\g, \\
		\f(t_0,x) &= \f_0.
	\end{aligned}
	\right.
\end{equation*}
We have the following estimates for $\f$ (see \cite[Proposition D.1, (133), (134)]{Buckmaster+DeLellis+Isett+Szekelyhidi_2015}):
\begin{align}\label{B.2}
	\|\f(t)\|_{C^1_{x}}\leq \|\f_0\|_{C^1_{x}}e^{(t-t_0)\|\bm\|_{C^0_{[t_0,T]}C^1_x}} +\int_{t_0}^{t}e^{(t-\tau)\|\bm\|_{C^0_{[t_0,T]}C^1_x}} \|\g(\tau)\|_{C^1_x} \d\tau,
\end{align}
and more generally, for any $N\geq2$, there exists a constant $C=C_N$ such that
\begin{align}\label{B.3}
	\|\f(t)\|_{C^N_{x}} & \leq \left(\|\f_0\|_{C^N_{x}}+C(t-t_0)\|\bm\|_{C^0_{[t_0,T]}C^N_x}\|\f_0\|_{C^1_x}\right)e^{C(t-t_0)\|\bm\|_{C^0_{[t_0,T]}C^1_x}} 
	\nonumber\\& \quad +\int_{t_0}^{t}e^{C(t-\tau)\|\bm\|_{C^0_{[t_0,T]}C^1_x}} \left(\|\g(\tau)\|_{C^N_x} + C(t-\tau)\|\bm\|_{C^0_{[t_0,t]}C^N_x}\|\g(\tau)\|_{C^1_x}\right) \d\tau.
\end{align}
Consider the following case:
\begin{equation*}
	\left\{	\begin{aligned}
		(\partial_t + \bm\cdot\nabla)\Phi& =0,\\
	\Phi(t_0,x) &= x.
	\end{aligned}
	\right.
\end{equation*}
Now let $\Psi(s,x)=\Phi(s,x)-x$, then $\Psi$ satisfies the following equations:
\begin{equation*}
	\left\{	\begin{aligned}
		(\partial_t + \bm\cdot\nabla)\Psi& =-\bm, \\
		\Psi(t_0,x) &= 0.
	\end{aligned}
	\right.
\end{equation*}
From \eqref{B.2}, we have
\begin{align}\label{B.6}
	\|\nabla\Phi(t)-\Id\|_{C^0_x}=\|\Psi(t)\|_{C^1_x} &\leq \int_{t_0}^{t}e^{(t-\tau)\|\bm\|_{C^0_{[t_0,T]}C^1_{x}}}\|\bm\|_{C^0_{[t_0,T]}C^1_{x}}\d\tau
	  = e^{(t-t_0)\|\bm\|_{C^0_{[t_0,T]}C^1_{x}}}  -1.
\end{align}
Now let us find the similar estimates for the solution $\Phi_{k,j}$ of \eqref{Phi_kj}, for $j\in\{0,1,\ldots,\lceil m_q^{-1}\rceil\}$, where $m_q$ is given by \eqref{m_q}. In view of \eqref{z_q_C0}, \eqref{v_q_C1} and \eqref{m_q}, we obtain for any $k\in\mathbb{Z}$ and  $0<\beta<\frac{1}{200}$
\begin{align}\label{B.7}
	m_q \|\v_{\ell}+\z_{\ell}\|_{C^0_{[k,k+1]}C^1_{x}}& \leq m_q\left({\lambda_{q}^{\frac75}\delta_{q}^{\frac12}}+ \lambda_{q}^{\frac{2}{3}}\right)  \leq \lambda_{q+1}^{-\frac34}\lambda_{q}^{-\frac34}\delta_{q+1}^{-\frac14}\delta_{q}^{-\frac14}\left({\lambda_{q}^{\frac75}\delta_{q}^{\frac12}}+ \lambda_{q}^{\frac{2}{3}}\right) 
	\nonumber\\ & \lesssim  \lambda_{q+1}^{-\frac34}\lambda_{q}^{-\frac34}\delta_{q+1}^{-\frac14}\delta_{q}^{-\frac14} \lambda_{q}^{\frac75}\delta_{q}^{\frac12}
	 \lesssim  \lambda_{q+1}^{-\frac34}\lambda_{q}^{\frac{13}{20}}\delta_{q+1}^{-\frac14}\delta_{q}^{\frac14}  
     \nonumber\\ &   = \lambda_{q}^{-\frac{3}{4}\times \frac{23}{4}} \lambda_{q}^{\frac{13}{20}}\lambda_{q}^{\frac{\beta}{2}\times \frac{23}{4}} \lambda_{q}^{-\frac{\beta}{2}}  \leq \lambda_{q}^{-\frac{8}{5}} <\!< 1,
\end{align}
where  we select $a$ that is big enough to absorb the constant. Since $e^{x}-1\leq2x$ for $x\in[0,1]$, we obtain for \eqref{B.6} and \eqref{B.7} that
\begin{align}\label{B.8}
	\sup_{t\in[k+(j-1)m_q,k+(j+1)m_q]}\|\nabla\Phi_{k,j}(t)-\Id\|_{C^0_x} & \leq e^{2m_q\|\v_{\ell}+\z_{\ell}\|_{C^0_{[k,k+1]}C^1_{x}}}-1
	\nonumber\\ & \lesssim m_q \|\v_{\ell}+\z_{\ell}\|_{C^0_{[k,k+1]}C^1_{x}}
	   \leq \lambda_{q}^{-\frac{8}{5}} <\!< 1,
\end{align}
where  we select $a$ that is big enough to absorb the constant. Furthermore, we also have from \eqref{B.8}
\begin{equation}\label{B.9}
	\left\{	\begin{aligned}
		\sup_{t\in[k+(j-1)m_q,k+(j+1)m_q]}\|\nabla\Phi_{k,j}(t)\|_{C^0_x} &\leq \sup_{t\in[k+(j-1)m_q,k+(j+1)m_q]}\|\nabla\Phi_{k,j}(t)-\Id\|_{C^0_x}+1 \leq 2,\\
	\sup_{t\in[k+(j-1)m_q,k+(j+1)m_q]}\|\nabla\Phi_{k,j}(t)\|_{C^0_x} & \geq 1- \sup_{t\in[k+(j-1)m_q,k+(j+1)m_q]}\|\nabla\Phi_{k,j}(t)-\Id\|_{C^0_x} \geq \frac12.
	\end{aligned}
	\right.
\end{equation}
 Similarly, by \eqref{B.3} and \eqref{B.7}, we get for $N\geq1$
 \begin{align}\label{B.10}
 	\sup_{t\in[k+(j-1)m_q,k+(j+1)m_q]}\|\nabla\Phi_{k,j}(t)\|_{C^N_x} & \lesssim  m_q \|\v_{\ell}+\z_{\ell}\|_{C^0_{[k,k+1]}C^{N+1}_{x}} e^{2m_q \|\v_{\ell}+\z_{\ell}\|_{C^0_{[k,k+1]}C^1_{x}}}
 	\nonumber \\ & \lesssim \ell^{-N}m_q \|\v_{\ell}+\z_{\ell}\|_{C^0_{[k,k+1]}C^1_{x}} 
 	 \lesssim   \lambda_q^{\frac{8}{5}N}\lambda_{q}^{-\frac{8}{5}} = \lambda_{q}^{\frac85(N-1)},
 \end{align}
where we have also used \eqref{ell}. Now, using \eqref{Phi_kj}, \eqref{B.9}, \eqref{v_q_C0} and \eqref{z_q_C0}, we have 
\begin{align}\label{B.11}
&	\sup_{t\in[k+(j-1)m_q,k+(j+1)m_q]}\|\partial_t\Phi_{k,j}\|_{C^0_x} 
\nonumber\\ & \leq \sup_{t\in[k+(j-1)m_q,k+(j+1)m_q]} \|((\v_{\ell}+\z_{\ell})\cdot\nabla)\Phi_{k,j}\|_{C^0_x}
	\nonumber\\& \leq \sup_{t\in[k+(j-1)m_q,k+(j+1)m_q]}\|\v_{\ell}+\z_{\ell}\|_{C^0_x} \sup_{t\in[k+(j-1)m_q,k+(j+1)m_q]} \|\nabla\Phi_{k,j}\|_{C^0_x} 
	  \lesssim \lambda_{q}^{\frac13} + \lambda_{q}^{\frac{1}{3}} \lesssim \lambda_{q}^{\frac{1}{3}}.
\end{align}

Differentiating both side of the first equations of \eqref{Phi_kj}, and using \eqref{B.9}, \eqref{B.10}, \eqref{v_q_C0}, \eqref{v_q_C1} and \eqref{z_q_C0}, we estimate 
\begin{align}\label{B.12}
	&\sup_{t\in[k+(j-1)m_q,k+(j+1)m_q]}\|\partial_t\nabla\Phi_{k,j}\|_{C^0_x}
	\nonumber\\& \lesssim \sup_{t\in[k+(j-1)m_q,k+(j+1)m_q]} \|\nabla(\v_{\ell}+\z_{\ell})\|_{C^0_x} \sup_{t\in[k+(j-1)m_q,k+(j+1)m_q]}\|\nabla\Phi_{k,j}\|_{C^0_x} 
	\nonumber\\& \quad+ \sup_{t\in[k+(j-1)m_q,k+(j+1)m_q]} \|\v_{\ell}+\z_{\ell}\|_{C^0_x} \sup_{t\in[k+(j-1)m_q,k+(j+1)m_q]}\|\nabla^2\Phi_{k,j}\|_{C^0_x}
	\nonumber\\&   \lesssim (\lambda_{q}^{\frac75}\delta_{q}^{\frac12} + \lambda_{q}^{\frac23}) + (\lambda_{q}^{\frac13} + \lambda_{q}^{\frac{1}{3}} )  
  \lesssim  \lambda_{q}^{\frac75}.
\end{align}
	The estimates \eqref{B.9}-\eqref{B.12} have been used in Subsections \ref{v_q+1} and \ref{R_q+1} to estimate the $C^1_{t,x}$-norm of $\omega_{q+1}$ and \eqref{R_q_C0}-\eqref{R_q_EN} at the level of $q+1$.

\section{Useful lemmas}\setcounter{equation}{0}\label{Useful-Lemmas}
Let us first introduce the following lemma which makes rigorous the fact that $\mathcal{R}$ obeys the same elliptic regularity estimates as $|\nabla|^{-1}$. We recall the following stationary phase lemma (see for example \cite[Lemma 5.7]{Buckmaster+Vicol_2019_Notes} and \cite[Lemma 2.2]{Daneri+Szekelyhidi_2017}), adapted to our setting.
\begin{lemma}[Stationary phase Lemma]\label{SPL}
	Given $\zeta\in\mathbb{S}^2\cap\mathbb{Q}^3$, let $\lambda\zeta\in\mathbb{Z}^3$ and $\alpha\in(0,1)$. Assume that $a\in C^{m,\alpha}(\mathbb{T}^3)$ and $\Phi\in C^{m,\alpha}(\mathbb{T}^3;\mathbb{R}^3)$ are smooth functions such that the phase function $\Phi$ obeys
	\begin{align*}
		\frac{1}{C} \leq |\nabla\Phi|\leq C,
	\end{align*}
on $\T3$, for some constant $C\geq1$. Then, with the inverse divergence operator $\mathcal{R}$ defined in Subsection \ref{IDO-R}, we have for any $m\in\N$
\begin{align*}
	\left\|\mathcal{R}\left(a(x)e^{i\lambda\zeta\cdot\Phi(x)}\right)\right\|_{C^{\alpha}} & \lesssim \frac{\|a\|_{C^0}}{\lambda^{1-\alpha}} + \frac{\|a\|_{C^{m,\alpha}}+\|a\|_{C^0}\|\nabla\Phi\|_{C^{m,\alpha}}}{\lambda^{m-\alpha}},
\end{align*}
where the implicit constant depends on $C$, $\alpha$ and $m$ (in particular, not on the frequency $\lambda$).
\end{lemma}

\begin{lemma}[{\cite[Proposition C.1]{Buckmaster+DeLellis+Isett+Szekelyhidi_2015}}]\label{Diff_Comp}
	Let $\Phi:\mathcal{O}\to\R$ and $f:\R^n\to\mathcal{O}$ be two smooth functions, with $\mathcal{O}\subset\R^m$. Then, for every $N\in\N$, there is a constant $C=C(n,m,N)>0$ such that 
	\begin{align*}
		\big[\Phi\circ f\big]_N &\leq C \left(\big[\Phi\big]_{1}\big[f\big]_{N}+\|D\Phi\|_{N-1}\|f\|^{N-1}_{0}\big[f\big]_{N}\right),\\
		\big[\Phi\circ f\big]_N &\leq C \left(\big[\Phi\big]_{1}\big[f\big]_{N}+\|D\Phi\|_{N-1}\big[f\big]^{N}_{1}\right).
	\end{align*}
\end{lemma}

\begin{lemma}[{\cite[Theorem 1.4]{Roncal+Stinga_2016}, \cite[Theorem B.1]{DeRosa_2019}}]\label{Lem_C.2}
	Let $\gamma,\epsilon>0$ and $\beta\geq0$ such that $2\gamma+\beta+\epsilon\leq1$, and let $f(t):\mathbb{T}^3\to\R^3$. If $f\in C^{2\gamma+\beta +\epsilon}_{x}$, then $(-\Delta)^{\gamma}f\in C^{\beta}_x$, and there exists a constant $C=C(\epsilon)>0$ such that
	\begin{align*}
		\|(-\Delta)^{\gamma}f\|_{C^0_tC^{\beta}_{x}} \leq C(\epsilon) [f]_{C^0_tC^{2\gamma+\beta+\epsilon}_x}.
	\end{align*}
\end{lemma}

\end{appendix}

	\medskip\noindent
	{\bf Acknowledgments:}   We express our deep gratitude to the anonymous referee for the careful reading
and valuable suggestions and comments that have improved this manuscript significantly. {The first author would like to thank the Department of Atomic Energy, Government of India, for financial assistance and the Tata Institute of Fundamental Research - Centre for Applicable Mathematics (TIFR-CAM) for providing a stimulating scientific environment and resources. The second author acknowledges the support of the Department of Atomic Energy,  Government of India, under project no.$12$-R$\&$D-TFR-$5.01$-$0520$, and DST-SERB SJF grant DST/SJF/MS/$2021$/$44$. }



\end{document}